\documentclass{article}
\RequirePackage[numbers]{natbib}

\usepackage{amssymb} 
\usepackage{graphicx} 
\usepackage{amsthm} 
\usepackage{color}
\usepackage{scalerel,stackengine}
\usepackage{amsmath}
\usepackage{hyperref}
\usepackage{dsfont}
\usepackage[parfill]{parskip}
\usepackage[a4paper, total={6in, 8in}]{geometry}
\usepackage{subcaption}
\usepackage{bbm}
\usepackage{cleveref}

\stackMath
\newcommand\reallywidehat[1]{%
\savestack{\tmpbox}{\stretchto{%
  \scaleto{%
    \scalerel*[\widthof{\ensuremath{#1}}]{\kern-.6pt\bigwedge\kern-.6pt}%
    {\rule[-\textheight/2]{1ex}{\textheight}}
  }{\textheight}%
}{0.5ex}}%
\stackon[1pt]{#1}{\tmpbox}%
}
\theoremstyle{plain} 
\newtheorem{prop}{Proposition}[section]
\newtheorem{lemma}{Lemma}[section]
\newtheorem{thrm}{Theorem}[section]
\newtheorem{corol}{Corollary}[section]

\theoremstyle{definition} 
\newtheorem{mydef}{Definition}[section]

\theoremstyle{remark} 
\newtheorem{remark}{Remark}[section]

\title{Multi-scale Lipschitz percolation of increasing events for Poisson random walks}
\author{Peter Gracar\footnote{Department of Mathematical Sciences, University of Bath, UK. {\tt p.gracar@bath.ac.uk}}\; and Alexandre Stauffer\footnote{Department of Mathematical Sciences, University of Bath, UK, {\tt a.stauffer@bath.ac.uk}. Supported by a Marie Curie Career Integration Grant PCIG13-GA-2013-618588 DSRELIS, and an EPSRC Early Career Fellowship.}}
\date{}

\begin{document}

\maketitle

\begin{abstract}
	Consider the graph induced by \(\mathbb{Z}^d\), equipped with \emph{uniformly elliptic} random conductances. At time \(0\), place a Poisson point process of particles on \(\mathbb{Z}^d\) and let them perform independent simple random walks. Tessellate the graph into cubes indexed by \(i\in\mathbb{Z}^d\) and tessellate time into intervals indexed by \(\tau\).
	Given a local event $E(i,\tau)$ that depends only on the particles inside the space time region given by the cube $i$ and the time interval $\tau$, we prove the existence of a Lipschitz connected surface of \emph{cells} \((i,\tau)\) that separates the origin from infinity on which $E(i,\tau)$ holds. This gives a directly applicable and robust framework for proving results in this setting that need a multi-scale argument. For example, this allows us to prove that an infection spreads with positive speed among the particles.
\end{abstract}

\section{Introduction}\label{section:intro}

Let \(G=(\mathbb{Z}^d,E)\) be the \(d\)-dimensional square lattice with edges between nearest neighbors: \((x,y)\in E\) iff \(\|x-y\|_1=1\).
Start with a collection of particles given by a Poisson point process on \(\mathbb{Z}^d\) of intensity \(\lambda\), and let the particles move over time as independent continuous time simple random walks on \(G\). We refer to this system of particles as \emph{Poisson random walks}. 

Assume that at time \(0\) there is an infected particle at the origin, and that all other particles are uninfected. As particles move, an uninfected particle gets infected as soon as it shares a site with an infected particle.
Kesten and Sidoravicius \cite{Kesten2005} showed that for all \(\lambda>0\) the infection spreads with positive speed; that is, for all large enough \(t\), at time \(t\) there is an infected particle at distance of order \(t\) from the origin. 
A main challenge in establishing this result is that, as the infection spreads, it finds empty regions (i.e., regions without particles) of arbitrarily large sizes. An empty region \(A\subset\mathbb{Z}^d\) not only delays the spread of the infection locally, but also causes a decrease in the density of particles in a neighborhood around \(A\) as time goes on. A key part of the analysis in \cite{Kesten2005} is to control how often empty regions arise and how big an impact (in space and time) they cause. An additional challenge is that long-range dependences do arise. For example, if at some time the ball \(B(x,r)\) of radius \(r\) centered at \(x\in \mathbb{Z}^d\) is empty, then \(B(x,r/2)\) is likely to remain empty for a time of order \(r^2\). Thus, the probability that the space-time region \(B(x,r/2)\times[0,r^2]\) is empty of particles is at least exponential in \(r^d\), which is only a stretched exponential with respect to the volume of the space-time region. In \cite{Kesten2005}, the effect of empty regions was controlled via an intricate \emph{multi-scale argument}.

The problem of spread of infection among Poisson random walks is just one example where long-range dependences give rise to serious mathematical challenges, and where multi-scale arguments have been applied to 
great success.
In fact, multi-scale arguments have proved to be very useful in the analysis of several models, including the solution of several important questions regarding Poisson random walks 
\cite{Kesten2005,Kesten2006,Peres2012,Stauffer2014}, 
activated random walks \cite{Sidoravicius2014}, 
random interlacements \cite{Sidoravicius2009,Sznitman2012}, 
multi-particle diffusion limited aggregation \cite{Sidoravicius2016} and 
more general dependent percolation \cite{Candellero2015,Teixeira2016}. 

However, the main problem in developing a multi-scale analysis is that the argument is quite involved and can become very technical. 
Also, in each of the examples above, the involved multi-scale argument had to be developed from scratch and be tailored to the specific question being analyzed. 
Our main goal in this paper is to develop a more robust and systematic framework that can be applied to solve questions in the model of Poisson random walks without the need of carrying out a whole multi-scale argument each time.
We do this by showing that given a local event which is translation invariant and whose probability of occurrence is large enough, 
we can find a special percolating structure in space-time where this event holds.

We now explain our idea in a high-level way, deferring precise statements and definitions to \Cref{section:statement}. 
We tesselate space into cubes, indexed by \(i\in\mathbb{Z}^d\), and tessellate time into intervals indexed by \(\tau\in\mathbb{Z}\). Thus \((i,\tau)\) denotes the space-time cell of the tessellation consisting of the 
cube \(i\) and the time interval \(\tau\). Given any increasing, translation invariant event \(E(i,\tau)\) that is local (i.e., measurable with respect to the particles that get within some fixed distance to the space-time 
cell \((i,\tau)\)), if the marginal distribution \(\mathbb{P}(E(i,\tau))\) is large enough, 
our main result gives the existence of a \emph{two-sided Lipschitz surface} of space-time cells where \(E(i,\tau)\) holds for all cells in the surface. 

Once we obtain such a Lipschitz surface, instead of having to carry out a whole multi-scale analysis from scratch to analyze some question involving Poisson random walks, 
one is left with the much easier task of just coming up with a suitable choice of \(E(i,\tau)\). 
For example, for the case of spread of infection mentioned above, a natural choice is to define \(E(i,\tau)\) as the event that an infected particle in the cube \(i\) 
infects several other particles which then move to all cubes neighboring \(i\) by the end of the time interval \(\tau\). 
Then, the existence of the Lipschitz surface and its Lipschitz property ensures that, once the infection enters the surface, it is guaranteed to propagate through the surface. 


We further illustrate the applicability of our Lipschitz surface technique in \cite{Gracar2016}, where we apply the Lipschitz surface to study the spread of infection in the random conductance model.

\section{Setting and precise statement of the results}\label{section:statement}

{\bf Poisson~ random~ walks.}
We~ consider~ the~ graph~ \((\mathbb{Z}^d,E)\)~ with~ conductances \(\{\mu_{x,y}\}_{(x,y)\in E}\), which are i.i.d.\ non-negative weights on the edges of \(G\). In this paper, edges will always be undirected, so \(\mu_{x,y}=\mu_{y,x}\) for all \((x,y)\in E\). We also assume that the conductances are \emph{uniformly elliptic}: that is,
\begin{align}
	\textrm{there exists deterministic \(C_M>0\), such that }\nonumber\\
	\mu_{x,y}\in[C_M^{-1},C_M]\textrm{ for all }(x,y)\in {E},~\mathbb{P}-a.s.\label{eq:mu_bounds_new}
\end{align}
We say \(x\sim y\) if \((x,y)\in E\) and define \(\mu_x=\sum_{y\sim x}\mu_{x,y}\).
At time \(0\), consider a Poisson point process of particles on \(\mathbb{Z}^d\), with intensity measure \(\lambda(x)=\lambda_0\mu_x\) for some constant \(\lambda_0>0\) and all \(x\in\mathbb{Z}^d\). That is, for each \(x\in\mathbb{Z}^d\), the number of particles at \(x\) at time \(0\) is an independent Poisson random variable of mean \(\lambda_0\mu_x\). Then, let the particles perform independent continuous-time simple random walks on the weighted graph; i.e., a particle at \(x\in\mathbb{Z}^d\) jumps to a neighbor \(y\sim x\) at rate \(\frac{\mu_{x,y}}{\mu_x}\). It follows from the thinning property of Poisson random variables that the system of particles is in stationarity; that is, at any time \(t\), the particles are distributed according to a Poisson point process with intensity measure \(\lambda\). We refer to this system of particles as \emph{Poisson random walks} on \((G,\mu)\) with intensity \(\lambda_0\).

{\bf Tessellation.}
We now tesselate the graph \(G=(\mathbb{Z}^d,E)\) into \(d\)-dimensional cubes of side length \(\ell>0\). We index the cubes of the tessellation by integer vectors \(i\in\mathbb{Z}^d\) such that the cube \(i=(i_1,i_2,\dots,i_d)\) corresponds to the region \(\left(\prod_{j=1}^d[i_j\ell,(i_j+1)\ell]\right)\cap\mathbb{Z}^d\). Tessellate time into subintervals of length \(\beta\). We index the subintervals by \(\tau\in\mathbb{Z}\), representing the time interval \([\tau\beta,(\tau+1)\beta]\). We refer to the pair \((i,\tau)\), representing \(\prod_{j=1}^d[i_j\ell,(i_j+1)\ell]\times[\tau\beta,(\tau+1)\beta]\), as a \emph{space-time cell} and define the \emph{region of a cell} as \(R_1(i,\tau)=\prod_{j=1}^d[i_j\ell,(i_j+1)\ell]\times[\tau\beta,(\tau+1)\beta]\).

We will need to consider larger space-time cells as well. Let \(\eta\geq 1\) be an integer. For each cube \(i=(i_1,\dots,i_d)\) and time interval \(\tau\), define the \emph{super cube} \(i\) as \(\prod_{j=1}^d[(i_j-\eta)\ell,(i_j+\eta+1)\ell]\) and the \emph{super interval} \(\tau\) as \([\tau\beta,(\tau+\eta)\beta]\). We define the \emph{super cell} \((i,\tau)\) as the Cartesian product of the super cube \(i\) and the super interval \(\tau\).

{\bf Definitions for events.}
We define a particle system on \(\mathbb{Z}^d\) as a countable family  of not necessarily unique elements of \(\mathbb{Z}^d\), indexed by some countable set \(I\), representing the locations of the particles belonging to the particle system.
Let \((\Pi_s)_{s\geq 0}\) be a sequence of particle systems on \(\mathbb{Z}^d\), with \(\Pi_s\) representing the locations of the particles at time \(s\).
We say a particle system \(\Pi_s\) is distributed according to a Poisson random measure of intensity \(\zeta\), if for every \(A\subset\mathbb{Z}^d\), \(N(A)\) is a Poisson random variable with intensity \(\zeta(A)\), where \(N(A)\) is the number of particles belonging to \(\Pi_s\) that lie in \(A\).
We say an event \(E\) is \emph{increasing} for \((\Pi_s)_{s\geq 0}\) if the fact that \(E\) holds for \((\Pi_s)_{s\geq 0}\) implies that it holds for all \((\Pi'_s)_{s\geq 0}\) for which \(\Pi_s'\supseteq \Pi_s\) for all \(s\geq 0\).
We need the following definitions.

\begin{mydef}\label{def:restricted}
	We say an event \(E\) is \emph{restricted} to a region \(X\subset\mathbb{Z}^d\) and a time interval \([t_0,t_1]\) if it is measurable with respect to the \(\sigma\)-field generated by all the particles that are inside \(X\) at time \(t_0\) and their positions from time \(t_0\) to \(t_1\).
\end{mydef}
\begin{mydef}\label{def:displacement}
	We say a particle has displacement inside \(X'\) during a time interval \([t_0,t_0+t_1]\), if the location of the particle at all times during \([t_0,t_0+t_1]\) is inside \(x+X'\), where \(x\) is the location of the particle at time \(t_0\).
\end{mydef}
 
For an increasing event \(E\) that is restricted to a region \(X\) and time interval \([0,t]\), we have the following definition.

\begin{mydef}\label{def:probassoc}
	\(\nu_E\) is called the \emph{probability associated} to an increasing event \(E\) that is restricted to \(X\)  and a time interval \([0, t]\) if, for an intensity measure \(\zeta\) and a region \(X'\in\mathbb{Z}^d\), \(\nu_E(\zeta,X,X',t)\) is the probability that \(E\) happens given that, at time \(0\), 
	the particles in \(X\) are a particle system distributed according to the Poisson random measure of intensity \(\zeta\)
	and their motions from \(0\) to \(t\) are independent continuous time random walks on the weighted graph \((G,\mu)\), where the particles are conditioned to have displacement inside \(X'\) during \([0,t]\).
\end{mydef}

For each \((i,\tau)\in\mathbb{Z}^{d+1}\), let \(E_{\mathrm{st}}(i,\tau)\) be an increasing event restricted to the super cube \(i\) and the super interval \(\tau\). We will assume that \(E_{\textrm{st}}(i,\tau)\) is invariant under space-time translations. We say that a cell \((i,\tau)\) is \emph{good} if \(E_{\mathrm{st}}(i,\tau)\) holds and \emph{bad} otherwise.

{\bf The base-height index.}
We will need a different way to index space-time cells, which we refer to as the \emph{base-height index}. In the base-height index, we pick one of the \(d\) spatial dimensions and denote it as \emph{height}, using index \(h\in\mathbb{Z}\), while the other \(d\) space-time dimensions form the \emph{base}, which will be indexed by \(b\in\mathbb{Z}^d\). Then, a \emph{base-height cell} will be indexed by \((b,h)\in\mathbb{Z}^{d+1}\).
We will use the base-height index in order to define the two-sided Lipschitz surface so that it, as the name implies, satisfies the Lipschitz property. More precisely, we will define the two-sided Lipschitz surface to be a collection of space-time cells such that when considering the height of each cell as a mapping of its base, this mapping is Lipschitz continuous.

Analogously to space-time, we define the \emph{base-height super cell} \((b,h)\) to be the space-time super cell \((i,\tau)\), for which the base-height cell \((b,h)\) corresponds to the space-time cell \((i,\tau)\). Similarly, we define \(E_{\mathrm{bh}}(b,h)\), the increasing event restricted to the super cell \((b,h)\), to be the same as the event \(E_{\mathrm{st}}(i,\tau)\) for the space-time cell \((i,\tau)\) that corresponds to the base-height cell \((b,h)\).

{\bf Two-sided Lipschitz surface.}
Let a function \(F:\mathbb{Z}^d\rightarrow \mathbb{Z}\) be called a \emph{Lipschitz function}
	 if \(|F(x)-F(y)|\leq 1\) whenever \(\|x-y\|_1 = 1\).

\begin{mydef}\label{def:lip_surf}
	A \emph{two-sided Lipschitz surface} \(F\) is a set of base-height cells \((b,h)\in\mathbb{Z}^{d+1}\) such that for all \(b\in\mathbb{Z}^d\) there are exactly two (possibly equal) integer values \(F_+(b)\geq 0\) and \(F_-(b)\leq0\) for which \((b,F_+(b)),(b,F_-(b))\in F\) and, moreover, \(F_+\) and \(F_-\) are Lipschitz functions.
\end{mydef}

\begin{figure}[!h]
  \begin{center}
  	\includegraphics[width=0.8\linewidth]{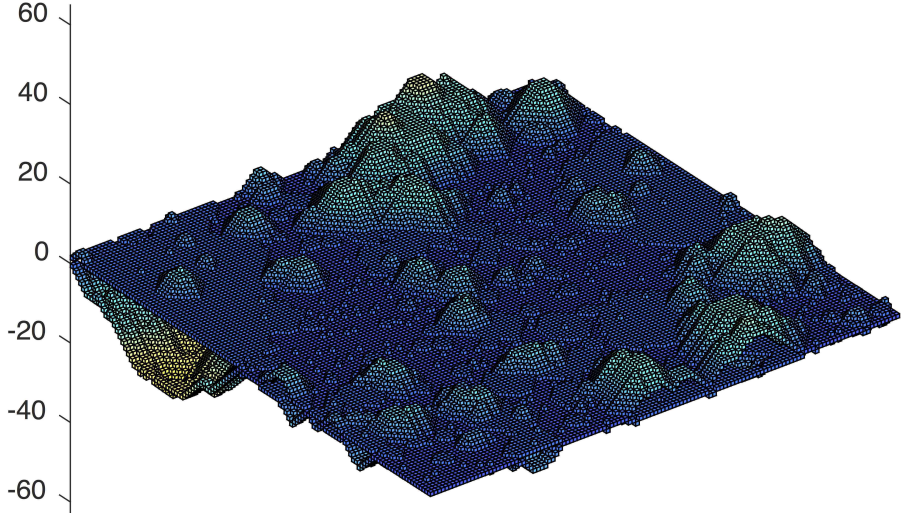}
  \end{center}
  \caption{A two-sided Lipschitz surface for the case of \(\mathbb{Z}^3\).}\label{fig:surface}
\end{figure}

An illustration of \(F\) for \(d=2\) is given in \Cref{fig:surface}. 
We say a space-time cell \((i,\tau)\) belongs to \(F\) if the corresponding base-height cell \((b,h)\) belongs to \(F\). We say a two-sided Lipschitz surface \(F\) \emph{exists}, if for all $b\in\mathbb{Z}^d$, we have $F_+(b)<\infty$ and $F_-(b)>-\infty$. For any positive integer $D$, we say a two-sided Lipschitz surface \emph{surrounds} a cell \((b',h')\) at distance \(D\) if any path \((b',h')=(b_0,h_0),(b_1,h_1),\dots,(b_n,h_n)\) for which \(\|(b_i,h_i)-(b_{i-1},h_{i-1})\|_1=1\) for all \(i\in\{1,\dots n\}\) and \(\|(b_n,h_n)-(b_0,h_0)\|_1>D\), intersects with \(F\).

{\bf Results.}
For any \(z\in\mathbb{Z}_+\), let \(Q_z=[-z/2,z/2]^d\). The following theorem establishes the existence of the Lipschitz surface.

\begin{thrm}\label{thrm:surface_event_simple}
	Let \((G,\mu)\) be a uniformly elliptic conductance graph on the lattice \(\mathbb{Z}^d\) for \(d\geq 2\). There exist positive constants \(c_0\), \(c_1\) and \(c_2\) such that the following holds. Tessellate \(G\) in space-time cells and super cells as described above for some \(\ell,\beta,\eta>0\) such that the ratio \(\beta/\ell^2<c_0\).
	Let \(E_{\mathrm{st}}(i,\tau)\) be an increasing event, restricted to the space-time super cell \((i,\tau)\).
	Fix \(\epsilon\in(0,1)\) and fix \(w\) such that
	\[
		w\geq\sqrt{\frac{\eta\beta}{c_2\ell^2}\log\left(\frac{8c_1}{\epsilon}\right)}.
	\]
	Then, there exists a positive number \(\alpha_0\) that depends on \(\epsilon\), \(\eta\), \(w\) and the ratio \(\beta/\ell^2\) so that if
	\begin{equation}\label{eq:main_theorem}
	\min\left\{C_{M}^{-1}\epsilon^2\lambda_0\ell^d,\log\left(\frac{1}{1-\nu_{E_{\mathrm{st}}}((1-\epsilon)\lambda,Q_{(2\eta+1)\ell},Q_{w\ell},\beta)}\right)\right\}\geq\alpha_0,
	\end{equation}\label{eq:minalpha}
	a two-sided Lipschitz surface \(F\) where \(E_{\mathrm{st}}(i,\tau)\) holds for all \((i,\tau)\in F\) almost surely exists.
	\end{thrm}
	
We now briefly explain the main conditions for the establishment of the above theorem. We usually fix $\beta/\ell^2$ to be an arbitrary, but small constant. The value of $\eta$ defines the super cubes, which just model how much 
overlap we need between the cells of the tessellation (usually to allow information to propagate from one cell to its neighbors). Once these two parameters are fixed, we need to satisfy (\ref{eq:main_theorem}).
First we need $C_{M}^{-1}\epsilon^2\lambda_0\ell^d\geq \alpha_0$. After fixing $\epsilon$, this can be satisfied either by setting $\ell$ large enough (which makes the cells of the tessellation large), or by 
assuming that the density of particles $\lambda_0$ is large enough. 
Then we still need to make $\nu_{E_{\mathrm{st}}}((1-\epsilon)\lambda,Q_{(2\eta+1)\ell},Q_{w\ell},\beta)\geq 1-\exp(-\alpha_0)$. 
Usually \(E_{\textrm{st}}\) is a local event that becomes more and more likely by setting \(\ell\) larger and larger; so 
having \(\ell\) large enough suffices to satisfy this condition as well. 
The value of $\epsilon>0$ is introduced so that in $\nu_{E_{\mathrm{st}}}$ we can consider a Poisson
point process of particles of intensity measure $(1-\epsilon)\lambda$, slightly smaller than the actual intensity of particles. 
This slack is needed to restrict our attention to the particles that ``behave well''. 
Then the lower bound on \(w\) is to guarantee that, as particles move in \(Q_{(2\eta+1)\ell}\) for time \(\beta\), with high probability they do not leave \(Q_{(2\eta+1)\ell+w\ell}\), 
allowing a better control of dependences between neighboring cells of the tessellation.
The proof of \Cref{thrm:surface_event_simple} is given in \Cref{section:proof_of_surface}. With some additional work, which we do in \Cref{section:dobule_diagonal}, we can establish the following property of \(F\).

\begin{thrm}\label{thrm:shell_tail}
	Assume the conditions of \Cref{thrm:surface_event_simple} are satisfied. There exist positive constants \(c\) and \(C\) such that, for any sufficiently large \(r>0\), we have
	\begin{align*}
		\mathbb{P}\left[\begin{array}{c}F\textrm{ does not surround}\\\textrm{ the origin at distance }r\end{array}\right]&
		\leq\left\{
		\begin{array}{ll}
			\sum_{s\geq r}s^d\exp\{-C\lambda_0\frac{\ell s}{(\log\ell s)^c}\},&\textrm{for }d=2\\
			\sum_{s\geq r}s^d\exp\{-C\lambda_0\ell s\},&\textrm{for }d\geq 3.
		\end{array}\right.&
	\end{align*}
\end{thrm}

The way \Cref{thrm:shell_tail} is proved also gives that the parts of the two-sided Lipschitz surface where the two sides \(F_+\) and \(F_-\) intersect not only almost surely separate the origin from infinity within the ``zero-height hyperplane'' \(\mathbb{L}=\mathbb{Z}^{d}\times\{0\}\), but they even percolate within \(\mathbb{L}\). We say that the two-sided Lipschitz surface percolates within \(\mathbb{L}\) if the set \(\mathbb{L}\setminus F\) contains only finite connected components.

\begin{thrm}\label{corol:percolate_simple}
	Assume the conditions of \Cref{thrm:surface_event_simple} are satisfied.
	If in addition we have that \(\ell\) is sufficiently large and \(\mathbb{P}[E_{\mathrm{st}}(0,0)]\) is sufficiently large, then the zero-height cluster \(F\cap\mathbb{L}\) of the two-sided Lipschitz surface \(F\) percolates within \(\mathbb{L}\) almost surely.
\end{thrm}
\vspace{0.11cm}
\begin{remark}\label{rem:height1}
In the definition of the base-height index, we fixed height to correspond to one of the spatial dimensions. 
This is the natural setting for the application of this Lipschitz surface technique to all problems we have in mind, for example the ones in \cite{Gracar2016}. 
However, in the definition of the surface we could have let height correspond to the \emph{time} dimension. 
Then, Theorems \ref{thrm:surface_event_simple}, \ref{thrm:shell_tail} and \ref{corol:percolate_simple} hold for \(d\geq 3\), but they no longer hold for \(d=2\). 
See \Cref{rem:height2} in \Cref{section:support_connected_paths} for details. 
\end{remark}

The remainder of this paper is structured as follows. In \Cref{section:paths} we give a construction of the two-sided Lipschitz surface for site percolation. \Cref{section:multiscale} introduces multiple scales of the tessellation and \Cref{section:support_connected_paths} generalizes the paths defined in the construction from \Cref{section:paths} to this multi-scale framework. \Cref{section:proof_of_prop} ties together the results from the previous sections, which is then applied in \Cref{section:proof_of_surface} to prove \Cref{thrm:surface_event_simple}. \Cref{section:dobule_diagonal} extends the results to a larger class of paths, which let us control areas where the two sides of the Lipschitz surface have non-zero height, in order to prove Theorems \ref{thrm:shell_tail} and \ref{corol:percolate_simple}.

\section{Two-sided Lipschitz surface in percolation}\label{section:paths}
In this section we show how to construct the Lipschitz surface \(F\) given a realization of the events \(E_{\textrm{bh}}(b,h)\), \((b,h)\in\mathbb{Z}^{d+1}\), from \Cref{section:statement}. For this, we regard \((E_{\textrm{bh}}(b,h))_{(b,h)\in\mathbb{Z}^{d+1}}\) as a site percolation process on \(\mathbb{Z}^{d+1}\) so that a site \((b,h)\in\mathbb{Z}^{d+1}\) is considered to be \emph{open} iff \(E_{\textrm{bh}}(b,h)\) holds and \emph{closed} otherwise. 
We assume that the \(E_{\textrm{bh}}(b,h)\) are translation invariant. The concept of Lipschitz percolation for independent Bernoulli percolation was introduced and studied in \cite{Dirr2010,Grimmett2012}. We modify their approach as we need several additional properties from the surface, such as the surface being two sided (i.e., composed of two sheets), the surface being close enough to the zero-height hyperplane \(\mathbb{L}=\mathbb{Z}^{d}\times\{0\}\), and the two sides of the surface intersecting in several points in~\(\mathbb{L}\).

The construction of \(F\) is based on the definition of a special type of paths, which we call \(d\)-paths. The definition of \(d\)-paths is based on a few rules. The first is that \(d\)-paths only start from \emph{closed} sites at height \(0\) (i.e., closed sites of \(\mathbb{L}\)).
For \(x\in\mathbb{Z}\), define the set \(\operatorname{Sign}(x)\) as \(\{+1\}\) if \(x>0\), \(\{-1\}\) if \(x<0\), and \(\{-1,+1\}\) if \(x=0\). 
A \(d\)-path from a closed site \(u\in\mathbb{L}\) to a not necessarily closed site \(v\in\mathbb{Z}^{d+1}\) is any finite sequence of distinct sites \(u=(b_0,0),(b_1,h_1),\dots,(b_k,h_k)=v\) of \(\mathbb{Z}^{d+1}\) such that for each \(i=1,2,\dots,k\) we have that either (\ref{for:heightdiff}) or (\ref{for:diagdiff}) below hold:

\begin{equation}\label{for:heightdiff}
	b_i=b_{i-1},\;h_i-h_{i-1}\in\operatorname{Sign}(h_{i-1})\textrm{ and }(b_i,h_i)\textrm{ is a closed site},
\end{equation}
or
\begin{equation}\label{for:diagdiff}
	\|b_i-b_{i-1}\|_1=1,\;h_{i-1}-h_{i}\in\operatorname{Sign}(h_{i-1})\textrm{ and }h_{i-1}\neq 0.
\end{equation}

We say the \(i\)-th move of a \(d\)-path is \emph{vertical} if it is like (\ref{for:heightdiff}), 
otherwise we say the \(i\)-th move is \emph{diagonal}. Note that in a vertical move, the path moves away from \(\mathbb{L}\), while in a diagonal move it moves towards \(\mathbb{L}\). Moreover, unlike a vertical move, a diagonal move is not required to go into a closed site and cannot be performed from a site of \(\mathbb{L}\).

In order to avoid issues of parity, we define for \((b,h)\in\mathbb{Z}^{d+1}\) the set of all sites that have the same base as \((b,h)\), but are further away from \(\mathbb{L}\).
\[
	\widehat{(b,h)}  :=\left\{(b,h')\in \mathbb{Z}^{d+1}:\;\tfrac{h'}{h}\geq 1\right\}.
\]

For \(u\in\mathbb{L}\) and \(v\in\mathbb{Z}^{d+1}\), we denote by \(u\rightarrowtail_{d}v\) the event that there is a \(d\)-path from \(u\) to at least one site of \(\hat v\). We say \(v\) is \emph{reachable} from \(u\) when this event holds\footnote{It would be natural to allow \(\|b_i-b_{i-1}\|_1\leq1\) in (\ref{for:diagdiff}). However, similar to \cite{Grimmett2012}, this is equivalent to the case when \(\|b_i-b_{i-1}\|_1=1\) if a \(d\)-path can be extended by vertical steps \emph{towards} \(\mathbb{L}\).}.

We now define several sets of sites and some corresponding values, which will let us construct the desired two-sided Lipschitz surface.

\begin{mydef}\label{def:hill}\label{def:mountain}
The \emph{hill} around \(u\in\mathbb{L}\) is the set of all sites that are reachable from \(u\),
	\[H_u:=\{v\in\mathbb{Z}^{d+1}:\,u\rightarrowtail_{d}v\}.\]
The \emph{mountain} around \(v\in\mathbb{L}\) is the union of all hills that contain \(v\)
	\[
		M_v=\bigcup_{u:v\in H_u}H_u.
	\]
\end{mydef}

\begin{figure}[!hbt]
\centering
\begin{subfigure}{.5\textwidth}
  \centering
  \includegraphics[width=1\linewidth]{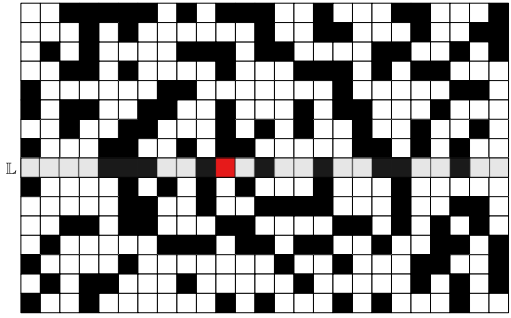}
  \caption{Site percolation with the closed site \\\(u\in\mathbb{L}\) marked in red.}
  \label{fig:sub1}
\end{subfigure}%
\begin{subfigure}{.5\textwidth}
  \centering
  \includegraphics[width=1\linewidth]{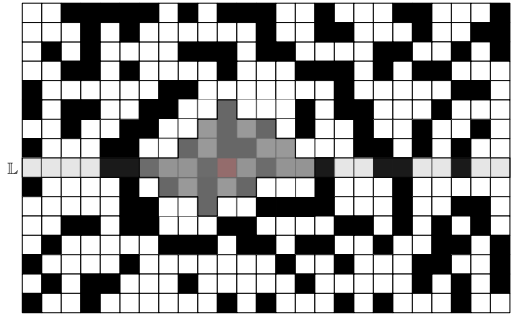}
  \caption{The hill \(H_u\) of all sites that can be reached from \(u\) overlaid in gray.}
  \label{fig:sub2}
\end{subfigure}
\begin{subfigure}{.5\textwidth}
  \centering
  \includegraphics[width=1\linewidth]{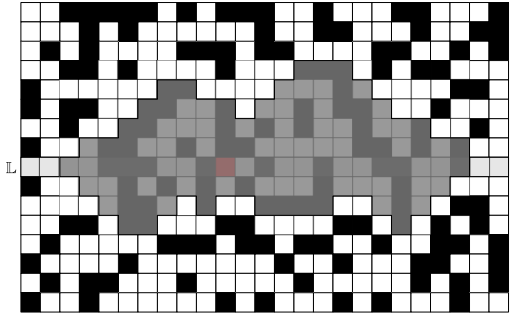}
  \caption{The mountain \(M_u\) of all hills that contain\\ \(u\) overlaid in gray.}
  \label{fig:sub3}
\end{subfigure}%
\begin{subfigure}{.5\textwidth}
  \centering
  \includegraphics[width=1\linewidth]{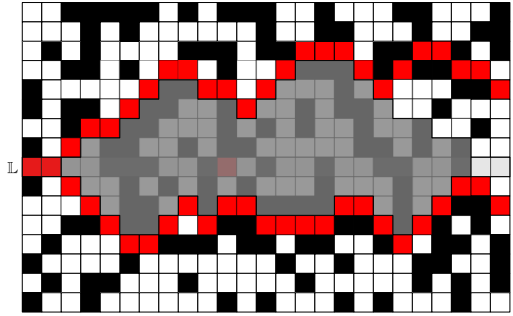}
  \caption{A possible realization of the Lipschitz surface \(F\) in red (cf. \Cref{def:lip_surf_main}).}
  \label{fig:sub4}
\end{subfigure}
\caption{Examples of \(H_u\) and \(M_u\) of a chosen site \(u\) for site percolation on \(\mathbb{Z}^2\). Open sites are white and closed sites are black.}
\label{fig:hills_and_mountains}
\end{figure}

Note that the sets \(H_u\) and \(M_v\) can be empty; in particular, \(H_u=\emptyset\) if \(u\) is an open site. We define
the positive and negative \emph{depths} of a set \(S\subset\mathbb{Z}^{d+1}\) at site \(u=(b,h)\in\mathbb{Z}^{d+1}\) as
	\[l_u^+(S)=\sup\{k:(b,h+k)\in S\},\]
and
	\[l_u^-(S)=\sup\{k:(b,h-k)\in S\}.\]

Define also the \emph{radius} of a set \(S\subset\mathbb{Z}^{d+1}\) around \(u\) as
	\[\operatorname{rad}_u(S)=\sup\{\|v-u\|_1:\,v\in S\}.\]
We are now ready to define our two-sided Lipschitz surface \(F\); see \Cref{fig:hills_and_mountains} for an illustration of \(H_u\), \(M_u\) and \(F\), and \Cref{fig:surface} for an example of \(F\) in three dimensions.

\begin{mydef}\label{def:F+andF-}
	For \(u\in \mathbb{L}\) define
		\[F_+(u)=\left\{\begin{array}{ll}
		1+l_u^+(M_u)&\textrm{if }M_u\neq\emptyset\\
		0&\textrm{if }M_u=\emptyset,
	\end{array}
	\right.\]
	and
		\[F_-(u)=\left\{\begin{array}{ll}
		-1-l_u^-(M_u)&\textrm{if }M_u\neq\emptyset\\
		0&\textrm{if }M_u=\emptyset.
	\end{array}
	\right.\]
\end{mydef}

\begin{mydef}\label{def:lip_surf_main}
	The two-sided Lipschitz surface \(F\) is defined as the set of sites
	\[
		\bigcup_{b\in\mathbb{Z}^d}(b,F_-(b))\cup(b,F_+(b)).
	\]
\end{mydef}

Note that the Lipschitz surface ``envelops'' the union of mountains \(\bigcup_{u\in\mathbb{L}}M_u\).
By definition, if \(l_u^{\pm}(M_u)\) is infinite for some \(u\), then it is infinite for all \(u\) (because of the diagonal moves of \(d\)-paths). Thus it is sufficient to show that \(l_0^{\pm}(M_0)\) is finite almost surely in order to guarantee the existence of \(F\). The theorem below establishes that \(F\) is finite almost surely; its proof follows along the lines of \cite[Theorem 1]{Grimmett2012}.

\begin{thrm}\label{thrm:surface}
	For any \(d\geq 1\), if \((E_{\textrm{bh}}(b,h))_{(b,h)\in\mathbb{Z}^{d+1}}\) is translation invariant and 
	\begin{equation}\label{for:radius_bound}
		\sum_{r\geq 1}r^d\mathbb{P}\left[\operatorname{rad}_0(H_0)>r\right]<\infty,
	\end{equation}
	then there exist almost surely a two sided Lipschitz surface \(F\) as in \Cref{def:lip_surf_main}. Moreover, the functions \(F_+\) and \(F_-\) from \Cref{def:F+andF-} satisfy
	\begin{enumerate}
	\item For each \(u=(b,0)\in\mathbb{L}\), the sites \((b,F_+(u))\) and \((b,F_-(u))\) are open.
	\item For any \(u,u'\in\mathbb{L}\) with \(\|u-u'\|_1=1\), we have \(|F_+(u)-F_+(u')|\leq1\) and \(|F_-(u)-F_-(u')|\leq1\).
	\end{enumerate}
\end{thrm}

\begin{proof}
	We start by showing item 1. First, suppose that \(M_0\neq\emptyset\), and assume the opposite, i.e. that the site \((b,F_+(u))\) is closed. By the definition of the function \(l_u^+\), the site \((b,l_u^+(M_u))\) belongs to \(M_u\). Then, since \(F_+(u)=1+l_u^+(M_u)\) and \(M_u\neq\emptyset\), we can extend the \(d\)-path reaching the site \((b,l_u^+(M_u))\) with a vertical move into the closed site \((b,F_+(u))\). This gives that \((b,F_+(u))\in M_u\), which is in contradiction with the construction of \(F_+\). When \(M_u=\emptyset\), we have \(H_u=\emptyset\) and the site \((b,F_+(u))=(b,0)\) is open by definition. The proof for \((b,F_-(u))\) is similar.
	
	Next, we establish item 2. Let \(u,u'\in\mathbb{L}\) with \(\|u-u'\|_1=1\). To show that \(|F_+(u)-F_+(u')|\leq1\), it is enough to show that \(F_+(u')-F_+(u)\geq -1\) since the roles of \(u\) and \(u'\) are symmetric. Assume the converse, that is, that \(F_+(u)\geq F_+(u')+2\). Write \(u=(b,0)\) and \(u'=(b',0)\). We have by \Cref{def:F+andF-} that \((b,F_+(u)-1)\in M_u\), so the site \((b,F_+(u)-1)\) can be reached by some \(d\)-path from \(\mathbb{L}\). Extending this path by a diagonal move, we have that the site \((b',F_+(u)-2)\in M_u\). Since \((b',F_+(u)-2)\in\reallywidehat{(b',F_+(u'))}\) by our assumption, we obtain that \((b',F_+(u'))\in M_u\), contradicting the construction of \(F_+\).
	The proof for \(F_-\) is similar.

	Finally, we prove the almost sure existence of \(F_+\), that is, that \(l_0^+(M_0)\) is almost surely finite. Because of the diagonal moves we have that \(l_0(M_0)\leq\operatorname{rad}_0(M_0)\), so we only need to show that \(\operatorname{rad}_0(M_0)<\infty\). By translation invariance we have
	\begin{align*}
	\mathbb{P}[\operatorname{rad}_0(M_0)\geq r]&\leq \sum_{v\in \mathbb{L}}\mathbb{P}[0\in H_v,\,\operatorname{rad}_v(H_v)\geq r-\|v\|_1]\\
	&=\sum_{v\in \mathbb{L}}\mathbb{P}[v\in H_0,\,\operatorname{rad}_0(H_0)\geq r-\|v\|_1]
	\end{align*}

The last sum can be split into two sums depending on whether or not \(\|v\|_1\leq r/2\). In the first case, the sum is no larger than \(cr^d\mathbb{P}[\operatorname{rad}_0(H_0)\geq r/2]\) for some constant \(c\), and by (\ref{for:radius_bound}) this term goes to 0 as \(r\) increases. Since \(\{v\in H_{0}\}\subseteq\{\operatorname{rad}_0(H_0)\geq\|v\|_1\}\), we can bound the sum for which \(\|v\|_1> r/2\) by
	\begin{align*}
		\sum_{\substack{v\in \mathbb{L}\\\|v\|_1\geq r/2}}\mathbb{P}[v\in H_0]&\leq\sum_{s\geq r/2}Cs^{d}\mathbb{P}[\operatorname{rad}_0(H_0)\geq s],
	\end{align*}
	where \(C>0\) is a constant that depends only on \(d\). By (\ref{for:radius_bound}) this term also goes to 0 as \(r\) increases, which concludes the proof.
\end{proof}

\section{Multi-scale setup}\label{section:multiscale}

In light of \Cref{thrm:surface}, the key in establishing the existence of the Lipschitz surface is to control the radius of \(H_0\). To do this, we look at all paths starting from \(0\) and the probability that they are a \(d\)-path. The challenge is that the event that a given cell \((b,h)\) is bad is not independent of other space-time cells. To solve this problem we resort to a multi-scale approach.
After defining the multi-scale tessellation, we will also state a result regarding \emph{local mixing} of particles, which we will use to link cells from one scale to the next.

\subsection{Tessellation}\label{section:tessel}
We start by tessellating space at multiple scales. Let \(m>0\) be a sufficiently large integer
and let \(\epsilon\in(0,1)\). 
For each scale \(k\geq 1\) we will tessellate the graph \(G=(\mathbb{Z}^d,E)\) into cubes of length \(\ell_k\) such that
\[
	\ell_1=\ell\quad\textrm{and}\quad\ell_k=mk^a\ell_{k-1}=m^{k-1}(k!)^a\ell,
\]
where \(a\) is a large integer we will set later. Set also \(\ell_0=\ell/m\).

We index the cubes by integer vectors \(i\in\mathbb{Z}^d\) and denote them by \(S_k(i)\). Then, for \(i=(i_1,i_2,\dots,i_d)\) we have
\[
	S_k(i)=\prod_{j=1}^d[i_j\ell_k,(i_j+1)\ell_k].
\]

This makes \(S_k(i)\) the union of \((mk^a)^d\) cubes of scale \(k-1\). Next, we introduce the following hierarchy. For \(k,j\geq0\) and \(i\in\mathbb{Z}^d\) we define
\[
	\pi_k^{(j)}(i)=i'\quad\textrm{iff}\quad S_k(i)\subseteq S_{k+j}(i').
\]
We say \((k+1,i')\) is the \emph{parent} of \((k,i)\) if \(\pi_k^{(1)}(i)=i'\) and in this case also say \((k,i)\) is a \emph{child} of \((k+1,i')\). We define the set of descendants of \((k,i)\) as \((k,i)\) and the union of all the descendants of the children of \((k,i)\) or as only \((k,i)\) in the case \((k,i)\) has no children.

Let \(w\) be a ``sufficiently large'', but otherwise arbitrary positive value; We will later require \(w\) to satisfy the inequality from \Cref{thrm:surface_event_simple}. For now, we can think of \(w\) as a large constant.
We introduce a new variable \(n\) that satisfies
\begin{equation}\label{for:eta}
	n^d=\frac{m}{7\eta}\quad\textrm{and}\quad n \geq \frac{1}{2} + \frac{w}{2\eta},
\end{equation}
where we impose the requirement on \(m\) to be large enough to yield \(n>1\) and to satisfy the inequality in (\ref{for:eta}). We also assume \(m\) is specified in such a way that \(n\) is an integer. Recall that \(\eta\) is the parameter introduced in the definition of super cells in the tessellation of \Cref{section:statement}, and that \(\eta\geq 1\) is an integer.

We define some larger cubes based on \(S_k(i)\). For \(k\geq 0\) define the \emph{base} and the \emph{area of influence} of \(S_k(i)\) as, respectively,
\[
	S_k^{\textrm{base}}(i)=\bigcup_{i':\|i-i'\|_{\infty}\leq\eta mn(k+1)^a}S_k(i')\quad\textrm{and}\quad S_k^{\textrm{inf}}(i)=\bigcup_{i':\|i-i'\|_{\infty}\leq 2\eta mn(k+1)^a} S_k(i').
\]
For \(k\geq 1\) we also define the \emph{extended} cube
\[
	S_k^{\textrm{ext}}(i)=\bigcup_{i':\pi_{k-1}^{(1)}(i')=i}S_{k-1}^{\textrm{base}}(i').
\]
Observe that \(S_k^{\textrm{ext}}(i)\) is the union of the bases of the children of \((k,i)\), which are the \((k-1)\)-cubes contained in \(S_k(i)\). We can see that \(S_k(i)\subset S_k^{\textrm{base}}(i)\subset S_k^{\textrm{inf}}(i)\) and
\begin{equation}\label{eq:Shierarch}
	S_{k+1}^{\textrm{ext}}(\pi_k^{(1)}(i))=\bigcup_{i':\pi_k^{(1)}(i')=\pi_k^{(1)}(i)}S_k^{\textrm{base}}(i')\supset S_k^{\textrm{base}}(i).
\end{equation}

\begin{figure}[!ht]
\centering
\includegraphics{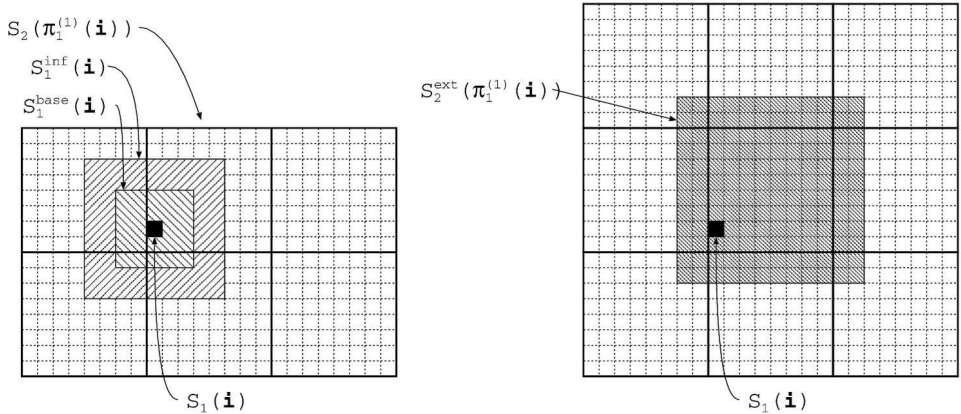}
\caption{Illustration of the tessellation of~$\mathbb{R}^d$. Different
scales are represented by the thickness of the lines; for example,\vspace*{1pt}
$S_2(\pi_1^{(1)}(i))$
is the square with thick borders that contains $S_1(i)$, which is
the black square.
Note that $S_1(i)$ is at the same position in both the left and the right pictures above, illustrating that $S^\mathrm{base}_1(i) \subset S^{\mathrm{ext}}_{2}(\pi_1^{(1)}(i))$ as given in~(\protect\ref{eq:Shierarch}).}\label{fig:spacescale}
\end{figure} 

\vspace{0.11cm}
\begin{remark}\label{rem:supercube_in_extended_cube}
	An important property derived from these definitions is that an extended cube of scale 1 has side length \(\ell+2\eta mn\ell_0=(1+2\eta n)\ell\). Therefore, for any \(i\in\mathbb{Z}^d\), the extended cube \(S_1^{\textrm{ext}}(i)\) contains the super cube \(i\) defined in the tessellation of \Cref{section:statement}. By the inequality in (\ref{for:eta}), we also have that the extended cube \(S_1^{\textrm{ext}}(i)\) has enough ``slack'' that this remains true even if we extend the super cube \(i\) by an additional factor of \(1+w\) in all directions.
\end{remark}

Now, we define the multi-scale tessellation of time. Let
\[
	\epsilon_1=\epsilon\quad\textrm{and}\quad \epsilon_k=\epsilon_{k-1}-\frac{\epsilon}{k^2}\quad\textrm{for all } k\geq 2.
\]
Define also \(\epsilon_0=2\epsilon\) for consistency. Let
\begin{equation}\label{for:beta}
	\beta_k=C_{\textrm{mix}}\frac{\ell_{k-1}^2}{(\epsilon_{k-1}-\epsilon_k)^{4/\Theta}}=C_{\textrm{mix}}\frac{\ell_{k-1}^2k^{8/\Theta}}{\epsilon^{4/\Theta}}\quad\textrm{for all } k\geq 1,
\end{equation}
where \(C_{\textrm{mix}}\geq 2^{4/\Theta}c_0\), and \(c_0\), \(\Theta\) are constants that will be given existence by \Cref{thrm:mixing} below. To simplify the notation, we assume that \(\frac{1}{\Theta}\) is an integer; otherwise we could work with \(\left\lceil\frac{1}{\Theta}\right\rceil\) instead. For \(k=1\) we set \(\beta=\beta_1=C_{\textrm{mix}}\frac{(\ell/m)^2}{\epsilon^{4/\Theta}}\). Given \(\beta/\ell^2\) and \(\epsilon\), \(m\) can be set sufficiently large so that \(C_{\textrm{mix}}\geq 2^{4/\Theta}c_0\). Observe that
\begin{equation}\label{for:beta_ratio}
	\frac{\beta_{k+1}}{\beta_k}=\frac{\ell_k^2(k+1)^{8/\Theta}}{\ell_{k-1}^2k^{8/\Theta}}=m^2k^{2a-8/\Theta}(k+1)^{8/\Theta}\quad\textrm{for all }k\geq1.
\end{equation}
Now, for scale \(k\geq 1\), we tessellate time into intervals of length \(\beta_k\). We index the time intervals by \(\tau\in\mathbb{Z}\) and denote them by \(T_k(\tau)\), where
\[
	T_k(\tau)=[\tau\beta_k,(\tau+1)\beta_k).
\]
We allow time to be negative and note that \(\beta_{k+1}/\beta_k\) is always an integer by (\ref{for:beta_ratio}) if \(a\) is chosen larger than \(4/\Theta\), which gives that a time interval of scale \(k\) is contained in a time interval of scale \(k+1\). We therefore assume from now on that \(a\) is an integer and sufficiently large for 
\begin{equation}\label{eq:theta_a}
	2a-8/\Theta>1
\end{equation}
to hold. 

Let \((k,\tau)\) refer to the time interval \(T_k(\tau)\). We also introduce a hierarchy over time, but which is different than the one defined for the cubes. For all \(k\) and \(\tau\) let \(\gamma_k^{(0)}(\tau)=\tau\), and for \(j\geq 1\), define
\[
	\gamma_k^{(j)}(\tau)=\tau'\quad\textrm{if}\quad \gamma_k^{(j-1)}(\tau)\beta_{k+j-1}\in T_{k+j}(\tau'+1).
\]

\begin{figure}[!h]
\centering
\includegraphics{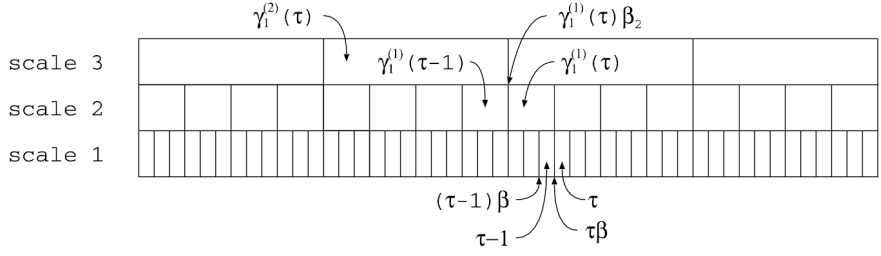}

\caption{Time scale. The horizontal axis represents time and the
vertical axis represents the scale.
Note that $\gamma_1^{(1)}(\tau)=\gamma_1^{(1)}(\tau+1)=\gamma
_1^{(1)}(\tau+2)$.}
\label{fig:timescale}
\end{figure}

For the time tessellation, if \(\tau'=\gamma_k^{(1)}(\tau)\), then the interval at scale \(k+1\) that contains \(T_k(\tau)\) is \(T_{k+1}(\tau'+1)\). For any \(j'\leq j\), we have \(\gamma_k^{(j)}=\gamma_{k+j'}^{(j-j')}(\gamma_k^{(j')})\). Thus, for \(\tau,\tau'\in\mathbb{Z}\) and \(k\geq 1\) we say that \((k+1,\tau')\) is the \emph{parent} of \((k,\tau)\), if \(\gamma_k^{(1)}(\tau)=\tau'\); in this case we also say that \((k,\tau)\) is a \emph{child} of \((k+1,\tau')\). We also define the set of descendants of \((k,\tau)\) as \((k,\tau)\) and the union of the descendants of the children of \((k,\tau)\) or only \((k,\tau)\) in the case \((k,\tau)\) has no children.

Now, for any \(i\in\mathbb{Z}^d\), \(k\geq 1\), \(\tau\in \mathbb{Z}\), we define the space-time parallelogram
\[
	R_k(i,\tau)=S_k(i)\times T_k(\tau),
\]
and note that these parallelograms are a tessellation of space and time. For \(k=1\) this is the same \(R_1\) defined in the tessellation of \Cref{section:statement}.

We extend \(\pi\) and \(\gamma\) to a hierarchy of space and time. Then, letting \((k,i,\tau)\) refer to the space-time cell \(S_k(i)\times T_k(\tau)\), we define the \emph{descendants} of \((k,i,\tau)\) as the cells \((k',i',\tau')\) so that \((k',i')\) is a descendant of \((k,i)\) and \((k',\tau')\) is a descendant of \((k,\tau)\). We  also say \((k,i,\tau)\) is an \emph{ancestor} of \((k',i',\tau')\) if \((k',i',\tau')\) is a descendant of \((k,i,\tau)\).

\subsection{A fractal percolation process}\label{section:fractal}

We now define the percolation process we will analyze. For the remainder of the paper, let \(E(i,\tau):=\mathbbm{1}_{E_{\textrm{st}}(i,\tau)}
\) be the indicator random variable of the increasing event \(E_{\textrm{st}}(i,\tau)\).
 For \(k\geq 1\), define \(S_k(i)\) to be \emph{\(k\)-dense} at some time \(t\) if all \((\frac{\ell_k}{\ell_{k-1}})^d=(mk^a)^d\) cubes \(S_{k-1}(i')\subset S_k(i)\) contain at least \((1-\epsilon_k)\lambda_0\sum_{y\in S_{k-1}(i')}\mu_y\) particles at time \(t\). For a cell \((k,i,\tau)\) let \(D_k(i,\tau)\) be the indicator random variable such that
\[
	D_k(i,\tau)=1\quad\textrm{iff}\quad S_k(i)\textrm{ is }k\textrm{-dense at time }\tau\beta_k.
\]
We also define a more restrictive indicator random variable:

\begin{center}
	\begin{minipage}{0.90\linewidth}
		\(D_k^{\textrm{ext}}(i,\tau)=1\) iff, at time \(\tau\beta_{k}\), all cubes \(S_{k-1}(i')\) of scale \(k-1\) contained in \(S_k^{\textrm{ext}}(i)\) have at least \((1-\epsilon_k)\lambda_0\sum_{y\in S_{k-1}(i')}\mu_y\) particles whose displacement throughout \([\tau\beta_k,(\tau+2)\beta_k]\) is in \(Q_{\eta mnk^a\ell_{k-1}}\).
	\end{minipage}
\end{center}
Recall the definition of the displacement of a particle from \Cref{def:displacement}.
Then \(D_k^{\textrm{ext}}(i,\tau)\leq D_k(i,\tau)\) for all cells \((k,i,\tau)\).
\vspace{0.11cm}
\begin{remark}
	An important property of this definition is that, when \(D_k^{\textrm{ext}}(i,\tau)=1\), if \((k-1,i',\tau')\) is a child of \((k,i,\tau)\), then we know that there are enough particles in \(S_{k-1}^{\textrm{base}}(i')\) at time \(\tau\beta_k\) and these particles never leave the cube \(S_{k-1}^{\textrm{inf}}(i')\) during the interval \([\tau\beta_k,\tau'\beta_{k-1}]\). This will let us apply \Cref{thrm:mixing} to show that if \(D_k^{\textrm{ext}}(i,\tau)=1\), then \(D_{k-1}^{\textrm{ext}}(i',\tau')\) is likely to be~1.
\end{remark}

Define
\begin{center}
	\begin{minipage}{0.90\linewidth}
		\(D_k^{\textrm{base}}(i,\tau)=1\) iff, at time \(\gamma_k^{(1)}(\tau)\beta_{k+1}\), all cubes \(S_k(i')\) of scale \(k\) inside \(S_k^{\textrm{base}}(i)\) contain at least \((1-\epsilon_{k+1})\lambda_0\sum_{y\in S_k(i')}\mu_y\) particles whose displacement throughout \([\gamma_k^{(1)}(\tau)\beta_{k+1},\tau\beta_{k}]\) is in \(Q_{\eta mn(k+1)^a\ell_k}\).
	\end{minipage}
\end{center}
Note that if \(D_{k+1}^{\textrm{ext}}(\pi_k^{(1)}(i),\gamma_k^{(1)}(\tau))=1\) then \(D_k^{\textrm{base}}(i,\tau)=1\). This gives that
\begin{equation}\label{eq:dbasegeqext}
	D_k^{\textrm{base}}(i,\tau)\geq D_{k+1}^{\textrm{ext}}(\pi_k^{(1)}(i),\gamma_k^{(1)}(\tau)),\quad\textrm{for all }(k,i,\tau).
\end{equation}

We next fix a scale \(\kappa\) as being the largest scale we will consider, and define
\[
	A_{\kappa}(i,\tau)=D_{\kappa}^{\textrm{ext}}(i,\tau).
\]
For \(k\) satisfying \(2\leq k\leq \kappa-1\), we set
\[
	A_{k}(i,\tau)=\max\left\{D_k^{\textrm{ext}}(i,\tau),1-D_k^{\textrm{base}}(i,\tau)\right\}.
\]
For scale 1 we set
\[
	A_{1}(i,\tau)=\max\left\{E(i,\tau),1-D_1^{\textrm{base}}(i,\tau)\right\}.
\]
Finally, define
\begin{equation}\label{for:A}
	A(i,\tau)=\prod_{k=1}^\kappa A_k(\pi_1^{(k-1)}(i),\gamma_1^{(k-1)}(\tau)).
\end{equation}

Intuitively, a cell \((k,i,\tau)\) will be ``well behaved'' if \(A_k(i,\tau)=1\). 
More precisely, it  follows from (\ref{eq:dbasegeqext}) that if \(A_{k+1}(i',\tau')=1\) and \((k,i,\tau)\) is a descendent of \((k+1,i',\tau')\), then \(A_{k}(i,\tau)=0\) if and only if \(D_{k}^{\textrm{ext}}(i,\tau)=0\) (or \(E(i,\tau)=0\) if \(k=1\)). On the other hand, \(A_{k+1}(i',\tau')=0\) implies that \(D_{k+1}^{\textrm{ext}}(i',\tau')=0\) and by (\ref{eq:dbasegeqext}) we have that \(D_{k}^{\textrm{base}}(i,\tau)\geq 0\), so that \(A_{k}(i,\tau)=1\) if either \(D_{k}^{\textrm{base}}(i,\tau)= 0\) or \(D_{k}^{\textrm{ext}}(i,\tau)=1\) (or \(E(i,\tau)=1\) if \(k=1\)). Therefore, \(A_k(i,\tau)\) can be seen as the indicator of the event that the particles are ``well behaved'' in the cell \((k,i,\tau)\), given that they were well behaved in the ancestor cell of \((k,i,\tau)\).
Finally, whenever \(A_k(i,\tau)=0\), it follows from (\ref{for:A}) that all descendants \((1,i',\tau')\) of \((k,i,\tau)\) at scale \(1\) have \(A(i',\tau')=0\).

\subsection{\texorpdfstring{\(D\)}{D}-paths and bad clusters}\label{section:DpathsAndBadClusters}
Consider two distinct cells \((i,\tau),(i',\tau')\) of scale 1. We say that \((i,\tau)\) is \emph{adjacent} to \((i',\tau')\) if \(\|i-i'\|_\infty\leq 1\) and \(|\tau-\tau'|\leq 1\).
Also, we say that \((i,\tau)\) is \emph{diagonally connected} to \((i',\tau')\) if there exists a sequence of cells \((i,\tau)=(b_0,h_0),(b_1,h_1),\dots,(b_n,h_n)=(\hat i,\hat\tau)\), where the indices \((b_j,h_j)\) refer to the base-height index, such that all the following hold:
\begin{itemize}
	\item for all \(j\in\{1,\dots,n\}\), \(\|b_j-b_{j-1}\|_1=1\) and \(h_{j-1}-h_{j}\in\operatorname{Sign}(h_{j-1})\),
	\item \(h_ih_j\geq0\) for all \(i,j\in\{0,\dots,n\}\), 
	\item \((\hat i,\hat\tau)\) is adjacent to \((i',\tau')\) or \((\hat i,\hat\tau)=(i',\tau')\).
\end{itemize}
The definition of diagonally connected is in line with the definition of \(d\)-paths from \Cref{section:paths}, where paths can move diagonally towards \(\mathbb{L}\) regardless of the status (open or closed) of the cells. We then define a \emph{\(D\)-path} as a sequence of scale 1 cells where each cell is either adjacent or diagonally connected to the next cell in the sequence.

Recall also that a cell \((i,\tau)\) of scale \(1\) is denoted \emph{bad} if \(E_{\textrm{st}}(i,\tau)\) does not hold. 
Given a cell \((i,\tau)\) of scale 1, we define the \emph{bad cluster} \(K(i,\tau)\) as the set of cells \((i',\tau')\) of scale 1 that are bad and to which there exists a \(D\)-path from \((i,\tau)\) where all cells in the \(D\)-path are bad.
We say that a cell \((i,\tau)\) of scale \(1\) has a \emph{bad ancestry} if \(A(i,\tau)=0\) and in this case we define the \emph{cluster of bad ancestries} as
\begin{align*}
	K'(i,\tau)=\{(i',\tau')\in\mathbb{Z}^{d+1}:A(i',\tau')=0\textrm{ and }\exists\textrm{ a }D\textrm{-path from }\\
	(i,\tau)\textrm{ to }(i',\tau')\textrm{ where each cell}\\
	\textrm{of the path has a bad ancestry}\}.
\end{align*}

\begin{lemma}\label{lemma:EgeqA}
	For each cell \((i,\tau)\) of scale 1, we have that \(E(i,\tau)\geq A(i,\tau)\). This implies that \(K(i,\tau)\subseteq K'(i,\tau)\).
\end{lemma}

\begin{proof}
	Fix \((i,\tau)\in\mathbb{Z}^{d+1}\). Then, for \(k=1\), define \(X_1=E(i,\tau)\) and, for \(k\geq 2\), define \(X_k=D_k^{\textrm{ext}}(\pi_1^{(k-1)}(i),\gamma_1^{(k-1)}(\tau))\). Let \(Y_k=D_k^{\textrm{base}}(\pi_1^{(k-1)}(i),\gamma_1^{(k-1)}(\tau))\). Therefore, by the definition of \(A\) in (\ref{for:A}), we have
	\[
		A(i,\tau)=\left(\prod_{k=1}^{\kappa-1}\max\{X_k,1-Y_k\}\right)X_{\kappa}.
	\]
	We now have \(Y_k\geq X_{k+1}\) for all \(k\). Therefore, for any \(k\leq\kappa-1\), we have
	\[
		\max\{X_k,1-Y_k\}X_{k+1}\leq\max\{X_k,1-X_{k+1}\}X_{k+1}=X_kX_{k+1}.
	\]
	Applying this repeatedly, we have
	\[
		A(i,\tau)\leq\left(\prod_{k=1}^{\kappa-2}\max\{X_k,1-Y_k\}\right)X_{\kappa-1}X_{\kappa}\leq\prod_{k=1}^\kappa X_k\leq X_1=E(i,\tau).
	\]
\end{proof}

\subsection{Local mixing}\label{section:mixing}
Let \(G=(\mathbb{Z}^d,E)\) be the \(d\)-dimensional square lattice equipped with conductances \((\mu_{x,y})_{(x,y)\in E}\) satisfying (\ref{eq:mu_bounds_new}).
The next theorem shows that if particles are dense enough inside a large cube \(Q_K=[-K/2,K/2]^d\), then after particles move for some time, their distribution inside \(Q_K\) (but away from \(Q_K\)'s boundary) dominates an independent Poisson point process.

\begin{thrm}[{\cite[Theorem~4.1]{Gracar2016}}]\label{thrm:mixing}
	Let \(\mu_{x,y}\) satisfy (\ref{eq:mu_bounds_new}) for some constant \(C_M\) and \(c>0\) be an arbitrary constant. There exist positive constants \(c_0\), \(c_1\), \(C\) and \(\Theta\) such that the following holds.
	Fix \(K>\ell>0\) and \(\epsilon\in(0,1)\). Consider the cube \(Q_K\) tessellated into subcubes \((T_i)_{i}\) of side length \(\ell\).
	Suppose that at time \(0\) there is a collection of particles in \(Q_K\) with
	each subcube \(T_i\) containing at least \(\sum_{y\in T_i}\beta\mu_y>c\) particles for some \(\beta>0\) and that \(\ell\) is sufficiently large for this to be possible.
	Let \(\Delta\geq c_0\ell^2\epsilon^{-4/\Theta}\).
	Fix \(K'>0\) such that \(K-K'\geq c_1\sqrt{\Delta\log\Delta}\).
	For each \(j\), denote by \(Y_j\) the location of the \(j\)-th particle of the collection at time \(\Delta\), conditioned on having displacement in \(Q_{K-K'}\) during \([0,\Delta]\).
	Then there exists a coupling \(\mathbb{Q}\) of an independent Poisson point process \(\psi\) with intensity measure \(\zeta(y)=\beta(1-\epsilon)\mu_y\), \(y\in Q_{K'}\), and \((Y_j)_{j}\) such that \(\psi\) is a subset of \((Y_j)_{j}\) with probability at least
	\[
		1-\sum_{y\in  Q_{K'}}\exp\left\{-C\beta\mu_y\epsilon^2\Delta^{d/2}\right\}.
	\]
\end{thrm}

\subsection{High-level overview}
Here, we explain the intuition behind the definitions from Sections \ref{section:tessel} to \ref{section:mixing} and give a high-level overview of how \Cref{thrm:mixing} is applied.

The main idea is an adaptation of fractal percolation, so we begin by presenting this more intuitive idea first. Take the \(d\)-dimensional unit cube and partition it into \(r^d\) subcubes of side length \(\frac{1}{r}\), where \(r\in\mathbb{Z}\). We refer to the cubes of this first tessellation as \(1\)-cubes, and let each of them independently be \emph{open} with probability \(p\in(0,1)\) and \emph{closed} otherwise. We now repeat this tessellating process for each open \(1\)-cube, splitting it into \(r^d\) subcubes of side length \(\frac{1}{r^2}\) which we call \(2\)-cubes. We again independently declare each of the \(2\)-cubes open with probability \(p\). The \(1\)-cubes that are closed are not partitioned again, and the entire region spanned by these cubes is considered to be closed (see \Cref{fig:fractalpercolation}). We repeat this procedure until we obtain \(z\)-cubes of side length \(\frac{1}{r^z}\).

\begin{figure}[!ht]
\centering
\includegraphics{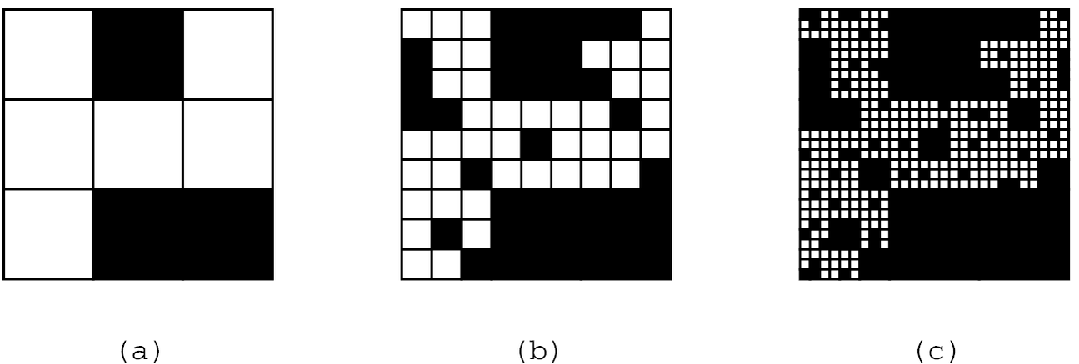}
\caption{Illustration of a fractal percolation process with $r=3$
and its $1$-cubes~\textup{(a)}, 2-cubes~\textup{(b)} and
3-cubes~\textup{(c)}. Black squares represent
closed cubes and white squares represent open cubes.}
\label{fig:fractalpercolation}
\end{figure}

We now present the intuition behind our definitions and the connection with fractal percolation. Begin at scale \(\kappa\). We tessellate space and time into very large cells. These are the cells indexed by the tuples \((\kappa,i,\tau)\) and each cell represents a cube in space and a time interval. Then, for each cell \((i,\tau)\) at scale \(\kappa\), we check whether the cell contains sufficiently many particles at the beginning of its time interval, i.e. we check whether \(A_{\kappa}(i,\tau)=D_{\kappa}^{\textrm{ext}}(i,\tau)=1\). If \(A_{\kappa}(i,\tau)=1\), we do a finer tessellation of the cell in both space and time. In terms of fractal percolation, this corresponds to the event that a large cube is open and then is subdivided into smaller cubes. On the other hand, if \(A_{\kappa}(i,\tau)=0\), we skip that cell and tessellate it no further, similarly to what happens to cubes that are closed in a fractal percolation process. We iterate this procedure until we obtain cells of volume \(\beta\ell^d\) (i.e. cells of scale 1). The main reason for employing this idea instead of analyzing the events \(D_k(i,\tau)\) directly is that the \(D_k(i,\tau)\) are highly dependent.

In the analysis, we start with the variables \(A_{k}(i,\tau)\) of the scale \(k=\kappa\), where the cells are so large that we can easily obtain \(A_{\kappa}=1\) for all \((i,\tau)\). Then we move from scale \(k+1\) to \(k\). Let \((i,\tau)\) be a cell of scale \(\kappa\). In order to analyze \(A_k(i,\tau)\), we need to observe \(A_{k+1}(i',\tau')\) such that \(\pi_k^{(1)}(i)=i'\) and \(\gamma_k^{(1)}(\tau)=\tau'\), i.e. \((k+1,i',\tau')\) is the parent of \((k,i,\tau)\) with respect to the hierarchies \(\pi\) and \(\gamma\). If \(A_{k+1}(i',\tau')=0\), then we do not need to observe \(A_k(i,\tau)\) since we will not do the finer tessellation of \(R_{k+1}(i',\tau')\) that produces the cell \((k,i,\tau)\). In this case, we will consider all descendants at scale \(1\) of the cell \((k+1,i',\tau')\) as ``bad'', and hence we will not need to observe any other descendant of \((k+1,i',\tau')\) such as \((k,i,\tau)\). On the other hand, if \(A_{k+1}(i',\tau')=1\), we know that there is a sufficiently large density of particles in the region \(S_{k}^{\textrm{base}}(i)\subset S_{k+1}^{\textrm{ext}}(i')\) that surrounds \(S_k(i)\) at time \(\tau'\beta_{k+1}\). Then, by allowing these particles to move from \(\tau'\beta_{k+1}\) to \(\tau\beta_k\), we obtain by \Cref{thrm:mixing} that many of these particles move inside \(S_k(i)\), giving that the probability that \(A_k(i,\tau)=0\), which corresponds to the event \(D_k^{\textrm{base}}(i,\tau)=1\) and \(D_k^{\textrm{ext}}(i,\tau)=0\), is small. We then apply this reasoning for all \((k,i,\tau)\). The key fact is that a dense cell at scale \(k\) makes the children of this cell likely to be dense as well.

We now give the intuition behind the different types of cubes. Let \((k,i,\tau)\) be a space-time cell of scale \(k\) and assume \((k+1,i',\tau')\) is the parent of \((k,i,\tau)\). We consider the extended cube \(S_{k+1}^{\textrm{ext}}(i')\) instead of just \(S_{k+1}(i',\tau')\) to assure that, when \(D_{k+1}^{\textrm{ext}}(i',\tau')=1\), then there is a large density of particles around \(S_k(i)\) at time \(\tau'\beta_{k+1}\) even if \(S_{k}(i)\) lies near the boundary of \(S_{k+1}(i')\); this happens since \(\{D_{k+1}^{\textrm{ext}}(i',\tau')=1\}\) guarantees that there are sufficiently many particles in \(S_k^{\textrm{base}}(i)\subset S_{k+1}^{\textrm{ext}}(i')\). We then let the particles move for time \(\tau\beta_k-\tau'\beta_{k+1}\geq\beta_{k+1}\), thereby allowing them to mix in \(S_k^{\textrm{base}}(i)\) and move inside \(S_k(i)\). While these particles move in the interval \([\tau'\beta_{k+1},\tau\beta_k)\), they never leave the \emph{area of influence} \(S_k^{\textrm{inf}}(i)\). This allows us to argue that cells that are sufficiently far apart in space are ``roughly independent'' since we only observe particles that stay inside the are of influence of their cells.

Now we give a brief sketch of the proof. We want to give an upper bound for the probability that \(K(0,0)\) is not contained in the region \([-t,t]^d\times{}[0,t]\). When that is the case, then there exists a \emph{very large} D-path of bad cells of scale \(1\). A natural strategy is to consider a fixed \(D\)-path from the cell \((0,0)\) to a cell outside of the region \([-t,t]^d\times{}[0,t]\) and show that the probability that all cells in this are bad is exponentially small, and then take the union bound over all such paths. However, this strategy seems challenging due to the dependencies among the events that the cells of a given path are bad and the fact that there is a large number of ways for two sequential cells of a \(D\)-path to be diagonally connected. We use two ideas to solve this problem: paths of cells of varying scales and well separated cells.

We start with cells of scale \(\kappa\), which are so large that we can show that, with very large probability, \(A_{\kappa}(i,\tau)=1\) for the cells \((i,\tau)\) of scale \(\kappa\) that are relevant for the existence of a \(D\)-path within \([-t,t]^d\times [0,t]\). Therefore, if a cell \((i,\tau)\) of scale \(1\) has \(A_1(i,\tau)=0\), we know that there exists an ancestor \((k',i',\tau')\) of \((1,i,\tau)\) such that \((k',i',\tau')\) is bad but its parents is good (i.e. \(A_{k'}(i',\tau')=0\)). With this, we have that if a \(D\)-path of bad cells of scale 1 exists, then there is a \(D\)-path of bad cells of varying scales. This \(D\)-path must contain sufficiently many cells because it must connect the cell \((0,0)\) to a cell outside of \([-t,t]^d\times{}[0,t]\). We take any fixed \(D\)-path of cells of varying scale and show that, given that this path contains sufficiently many cells, we can obtain a subset of the cells of the path so that these cells are ``well separated'' in space and time. We then use the fact that the \(A_k(i,\tau)\) are ``roughly independent'' for well separated cells which implies that the probability that all cells in this subset are bad is very small. Then, by applying the union bound with a careful counting argument  over all sets of well separated cells that can be obtained from a \(D\)-path of cells of varying scales, we establish \Cref{thrm:surface_event_simple}. In order to better define and count paths involving cells of multiple scales, we will introduce the notions of the support of a cell and the extended support of a cell.

\subsection{The support of a cell}
We define the \emph{time of influence} \(T_k^{\textrm{inf}}(\tau)\) of \((k,\tau)\) as
\[
	T_1^{\textrm{inf}}(\tau)=[\gamma_1^{(1)}(\tau)\beta_2,(\tau+\max\{\eta,2\})\beta_1]~~\textrm{and}~~ T_k^{\textrm{inf}}(\tau)=[\gamma_k^{(1)}(\tau)\beta_{k+1},(\tau+2)\beta_k]\textrm{ for }k\geq2,
\]
and set the \emph{region of influence} as
\[
	R_k^{\textrm{inf}}(i,\tau)=S_k^{\textrm{inf}}(i)\times T_k^{\textrm{inf}}(\tau).
\]

We assume \(m\) is sufficiently large with respect to \(\eta\) so that \(\max\{\eta,2\}\beta\leq\beta_2=m^22^{8/\Theta}\beta\), which gives that
\begin{equation}\label{for:TkSize}
	T_k^{\textrm{inf}}(\tau)\subseteq T_{k+1}(\gamma_k^{(1)}(\tau))\cup T_{k+1}(\gamma_k^{(1)}(\tau)+1)\cup T_{k+1}(\gamma_k^{(1)}(\tau)+2)
\end{equation}
We define the \emph{time support} as
\[
	T_k^{\textrm{sup}}(\tau)=\bigcup_{i=0}^8T_{k+1}(\gamma_k^{(1)}(\tau)-3+i)
\]
and the \emph{spatial support} as
\[
	S_k^{\textrm{sup}}(i)=\bigcup_{i':\|i'-\pi_k^{(1)}(i)\|_{\infty}\leq m}S_{k+1}(i'),
\]
and, for any cell \((k,i,\tau)\), we define
\[
	R_k^{\textrm{sup}}(i,\tau)=S_k^{\textrm{sup}}(i)\times T_k^{\textrm{sup}}(\tau)
\]

\begin{lemma}\label{lemma:R_inf_not_R_sup}
	For any sufficiently large \(m\) the following is true. For any cells \((k,i,\tau)\), \((k',i',\tau')\), with \(k\geq k'\), if \(R_{k'}^{\textrm{inf}}(i',\tau')\not\subseteq R_k^{\textrm{sup}}(i,\tau)\) then \(R_{k'}^{\textrm{inf}}(i',\tau')\cap R_k^{\textrm{inf}}(i,\tau)=\emptyset\).
\end{lemma}

\begin{figure}[!hbt]
		\centering
  		\includegraphics[width=0.7 \linewidth]{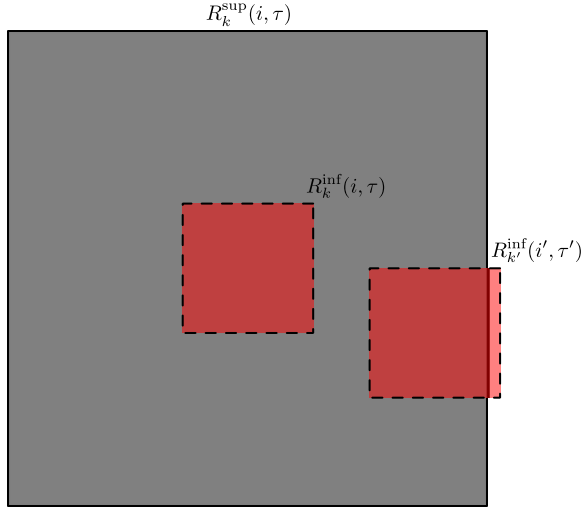}
  		\caption{The scaling of \(R_k^{\textrm{sup}}(i,\tau)\) relative to \(R_k^{\textrm{inf}}(i,\tau)\) is such that \Cref{lemma:R_inf_not_R_sup} holds for all \(k\) and \(k'\leq k\) when \(m\) is large enough. The figure shows the case when \(k=k'\); when \(k'<k\) the proof of the lemma becomes easier.}
  		\label{fig:Cubes2D}
\end{figure} 

\begin{proof}
	Note that, if \(R_{k'}^{\textrm{inf}}(i',\tau')\not\subseteq R_k^{\textrm{sup}}(i,\tau)\), then either \(T_{k'}^{\textrm{inf}}(\tau')\not\subseteq T_k^{\textrm{sup}}(\tau)\) or \(S_{k'}^{\textrm{inf}}(i')\not\subseteq S_k^{\textrm{sup}}(i)\). We start with the case that \(T_{k'}^{\textrm{inf}}(\tau')\not\subseteq T_k^{\textrm{sup}}(\tau)\) and show that this implies
	\[
		T_{k'}^{\textrm{inf}}(\tau')\cap T_{k}^{\textrm{inf}}(\tau)=\emptyset,
	\]
	which gives that \(R_{k'}^{\textrm{inf}}(i',\tau')\cap R_{k}^{\textrm{inf}}(i,\tau)=\emptyset\).

	Note that the interval \(T_{k'}^{\textrm{inf}}(\tau')\) has length at most \(3\beta_{k'+1}\) by (\ref{for:TkSize}). Then, since \(T_{k'}^{\textrm{inf}}(\tau')\not\subseteq T_k^{\textrm{sup}}(\tau)\),
	\begin{equation}\label{for:TkSizeIntersect}
		T_{k'}^{\textrm{inf}}(\tau')\cap [(\gamma_k^{(1)}(\tau)-3)\beta_{k+1}+3\beta_{k'+1},(\gamma_k^{(1)}(\tau)+6)\beta_{k+1}-3\beta_{k'+1}]=\emptyset.
	\end{equation}
	Using that \(\beta_{k'}\leq\beta_k\), we get
	\begin{align*}
		&[(\gamma_k^{(1)}(\tau)-3)\beta_{k+1}+3\beta_{k'+1},(\gamma_k^{(1)}(\tau)+6)\beta_{k+1}-3\beta_{k'+1}]\\
		&\qquad\supseteq[\gamma_k^{(1)}(\tau)\beta_{k+1},(\gamma_k^{(1)}(\tau)+3)\beta_{k+1}]\\
		&\qquad=T_{k+1}(\gamma_k^{(1)}(\tau))\cup T_{k+1}(\gamma_k^{(1)}(\tau)+1)\cup T_{k+1}(\gamma_k^{(1)}(\tau)+2)\\
		&\qquad\supseteq T_k^{\textrm{inf}}(\tau),
	\end{align*}
	where the last step follows from (\ref{for:TkSize}). This, together with (\ref{for:TkSizeIntersect}), implies that \(T_{k'}^{\textrm{inf}}(\tau')\cap T_{k}^{\textrm{inf}}(\tau)=\emptyset\).

	For the spatial component, consider the case \(S_{k'}^{\textrm{inf}}(i')\not\subseteq S_k^{\textrm{sup}}(i)\), for which we want to show that
	\[
		S_{k'}^{\textrm{inf}}(i')\cap S_{k}^{\textrm{inf}}(i)=\emptyset.
	\]
	Let \(x_1,x_2,\dots,x_d\) be defined so that \(S_k(i)=\prod_{j=1}^d[x_j,x_j+\ell_k]\). Then, we can write
	\begin{equation}\label{for:S_k_inf}
		S_k^{\textrm{inf}}(i)=\prod_{j=1}^d[x_j-2\eta mn(k+1)^a\ell_k,x_j+\ell_k+2\eta mn(k+1)^a\ell_k].
	\end{equation}
	Next, let \(y_1,y_2,\dots,y_d\) be defined so that \(S_k^{\textrm{sup}}(i)=\prod_{j=1}^d[y_j,y_j+(2m+1)\ell_{k+1}]\). Since \(S_{k'}^{\textrm{inf}}(i')\) is a cube of side length \((1+4\eta mn(k'+1)^a)\ell_{k'}\leq (1+4\eta mn(k+1)^a)\ell_k\) and \(S_{k'}^{\textrm{inf}}(i')\) is not contained in \(S_k^{\textrm{sup}}(i)\), we have that
	\begin{equation}\label{for:S_k'_inf}
		S_{k'}^{\textrm{inf}}(i')\cap\prod_{j=1}^d[y_j+(1+4\eta mn(k+1)^a)\ell_k,y_j+(2m+1)\ell_{k+1}-(1+4\eta mn(k+1)^a\ell_k]=\emptyset.
	\end{equation}

	Now, we use the fact that \(m\ell_{k+1}\leq x_j-y_j\leq (m+1)\ell_{k+1}-\ell_k\) for all \(j=1,2,\dots,d\). This and (\ref{for:S_k_inf}) give
	\begin{equation}\label{for:S_k_inf2}
		S_k^{\textrm{inf}}(i)\subseteq\prod_{j=1}^d[y_j+m\ell_{k+1}-2\eta mn(k+1)^a\ell_k,y_j+(m+1)\ell_{k+1}+2\eta mn(k+1)^a\ell_k].
	\end{equation}
	Now, using the relation between \(m\) and \(n\) in (\ref{for:eta}), we have that
	\begin{equation}\label{for:m_lk_bound}
		m\ell_{k+1}=m^2(k+1)^a\ell_k=7\eta mn^d(k+1)^a\ell_k\geq(1+6\eta mn(k+1)^a)\ell_{k}.
	\end{equation}
	Using this result in (\ref{for:S_k'_inf}) we get that \(S_{k'}^{\textrm{inf}}(i')\) does not intersect
	\begin{equation}\label{for:S_k'_inf_end}
		\prod_{j=1}^d[y_j+(1+4\eta mn(k+1)^a)\ell_k,y_j+(m+1)\ell_{k+1}+2\eta mn(k+1)^a\ell_k].
	\end{equation}
	Similarly, plugging (\ref{for:m_lk_bound}) into (\ref{for:S_k_inf2}) we see that \(S_k^{\textrm{inf}}(i)\) is contained in the space-time region given by (\ref{for:S_k'_inf_end}). These two facts establish the lemma.
\end{proof}

Another important property concerns the fact that the support of a cell contains all its descendants.

\begin{lemma}\label{lemma:R_in_R_sup}
	Assume \(m\geq 3\). For any cell \((k,i,\tau)\), if \((k',i',\tau')\) is a descendant of \((k,i,\tau)\) then
	\[
		R_{k'}(i',\tau')\subseteq R_k^{\textrm{sup}}(i,\tau).
	\]
	Moreover, \(R_k^{\textrm{sup}}(i,\tau)\) contains all the neighbors of \((k',i',\tau')\).
\end{lemma}

\begin{proof}
	Fix \(i'',\tau''\) such that \((k',i'',\tau'')\) is adjacent to \((k',i',\tau')\) and assume that the ancestor of \((k',i'',\tau'')\) of scale \(k\) is not \((k,i,\tau)\), otherwise the second part of the lemma follows from the first part. We prove this lemma first for space and then for time. For space, since \((k',i',\tau')\) is a descendant of \((k,i,\tau)\) we have that \(S_{k'}(i')\subseteq S_k(i)\subseteq S_k^{\textrm{sup}}(i)\). Also, \((k',i'')\) is adjacent to \((k',i')\) which implies that the ancestor of \((k',i'')\) of scale \(k\) is adjacent to \((k,i)\). Since \(S_k^{\textrm{sup}}(i)\) contains all cells of scale \(k\) that are adjacent to \((k,i)\), it also contains \(S_{k'}(i'')\).

	We now prove the lemma for the time dimension. The first part corresponds to showing that \(T_{k'}(\tau')\subseteq T_k^{\textrm{sup}}(\tau)\). Recall that \(T_{k'}(\tau')=[\tau'\beta_{k'},(\tau'+1)\beta_{k'}]\), which is contained in \([\tau\beta_k,(\tau'+1)\beta_{k'}]\) since \((k',i',\tau')\) is a descendant of \((k,i,\tau)\). Now, note that
	\[
		\tau\beta_k=\gamma_{k'}^{(k-k')}(\tau')\beta_k\geq \gamma_{k'}^{(k-k'-1)}(\tau')\beta_{k-1}-2\beta_k\geq \tau'\beta_{k'}-2\sum_{i=k'+1}^k\beta_i.
	\]
	Then, since \(k'\geq 1\), we can use the bound
	\[
		\sum_{i=2}^k\beta_i=C_{\textrm{mix}}\sum_{i=2}^k\frac{\ell_{i-1}^2i^{8/\Theta}}{\epsilon^{4/\Theta}}=C_{\textrm{mix}}\epsilon^{-4/\Theta}\sum_{i=2}^k\ell_{i-1}^2(i^{4/\Theta})^2\leq C_{\textrm{mix}}\epsilon^{-4/\Theta}2(k^{4/\Theta})^2\ell_{k-1}^2=2\beta_k,
	\]
	where the last inequality can be proven by induction on \(k\). Then, we have that
	\begin{equation}\label{for:tau_beta}
		\tau\beta_k\geq\tau'\beta_{k'}-4\beta_k.
	\end{equation}
	Since \(k>k'\geq1\), we have \(k> 1\) and
	\[
		4\beta_k+\beta_{k'}\leq 5\beta_k=5\frac{\beta_{k+1}}{m^2 k^{2a-8/\Theta}(k+1)^{8/\Theta}}\leq \beta_{k+1}.
	\]
	This combined with (\ref{for:tau_beta}) gives
	\[
		T_{k'}(\tau')\subseteq [\tau\beta_k,\tau\beta_k+4\beta_k+\beta_{k'}]\subseteq[\tau\beta_k,\tau\beta_k+\beta_{k+1}]\subseteq T_k^{\textrm{sup}}(\tau).
	\]
	This proves the first part of the lemma. To prove the second part, we use the fact that \((k',\tau'')\) is adjacent to \((k',\tau')\) and the result above, which gives
	\[
		T_{k'}(\tau'')\subseteq[\tau\beta_k-\beta_{k'},\tau\beta_k+\beta_{k+1}+\beta_{k'}]\subseteq T_k^{\textrm{sup}}(\tau).\qedhere
	\]
\end{proof}

\section{Multi-scale analysis of \texorpdfstring{\(D\)}{D}-paths}\label{section:support_connected_paths}

In order to prove our theorems we need to control the existence of \(D\)-paths of scale 1 whose cells have a bad ancestry (cf. \Cref{lemma:EgeqA}). We will do this via a multi-scale analysis of such paths. In \Cref{section:multiscale} we defined the multi-scale tessellation we need. Here we will use this framework to consider a multi-scale version of \(D\)-paths.

We start by defining the \emph{extended support} of a cell. Given a cell \((k,i,\tau)\), define
\[
	T_k^{\textrm{2sup}}(\tau)=\bigcup_{i=0}^{26}T_{k+1}(\gamma_k^{(1)}(\tau)-12+i)
\]
and
\[
	S_k^{\textrm{2sup}}(i)=\bigcup_{i':\|i'-\pi_k^{(1)}(i)\|_{\infty}\leq 3m+1}S_{k+1}(i').
\]
Then, as before, we set \(R_k^{\textrm{2sup}}(i,\tau)=S_k^{\textrm{2sup}}(i)\times T_k^{\textrm{2sup}}(\tau)\).
\vspace{0.11cm}
\begin{remark}\label{remark:contain}
The extended support is defined in a way so that if the supports of two cells intersect, the smaller of the supports is completely contained in the extended support of the bigger of the cells.
\end{remark}

We now extend the definition of a bad cell to multiple scales. 
We say that a cell \((k,i,\tau)\) is \emph{multi-scale bad} if \(A_k(i,\tau)=0\). 
Note that for \(k=1\), this definition is stricter than that of a bad cell, i.e. since \(E(i,\tau)\leq A_1(i,\tau)\) it follows that if a cell of scale 1 is multi-scale bad, it is also bad but not the other way around. Intuitively, a super cell of scale 1 is bad whenever the increasing event \(E_{\textrm{st}}(i,\tau)\) does not hold whereas it is multi-scale bad when the increasing event does not hold and \(D_1^{\textrm{base}}(i,\tau)=1\). 

Recall that two cells \((k,i,\tau)\) and \((k,i',\tau')\) of the same scale are said to be adjacent if \(\|i-i'\|_{\infty}\leq 1\) and \(|\tau-\tau'|\leq 1\).
	Let \((k_1,i_1,\tau_1),(k_2,i_2,\tau_2)\) be two cells with \(k_1>k_2\). We say \((k_1,i_1,\tau_1)\) and \((k_2,i_2,\tau_2)\) are \emph{adjacent} if \((k_1,i_1,\tau_1)\) is adjacent to \((k_1,\pi_{k_2}^{(k_1-k_2)}(i_2),\gamma_{k_2}^{(k_1-k_2)}(\tau_2))\).
	We say \((k,i,\tau)\) is \emph{diagonally connected} to \((k',i',\tau')\) if there exists a cell \((1,\hat i,\hat \tau)\) that is a descendant of \((k,i,\tau)\) and a cell \((1,i'',\tau'')\) that is a descendant of \((k',i',\tau')\), such that \((1,\hat i,\hat \tau)\) is diagonally connected to \((1,i'',\tau'')\).

We extend the definition of \(D\)-paths to cells of arbitrary scale by referring to a \(D\)-path as a sequence of distinct cells for which any two consecutive cells in the sequence are either adjacent or the first of the two cells is diagonally connected to the second. 
For any two cells \((k_1,i_1,\tau_1)\) and \((k_2,i_2,\tau_2)\) we say that they are \emph{well separated} if \(R_{k_1}(i_1,\tau_1)\not\subseteq R_{k_2
}^{\textrm{sup}}(i_2,\tau_2)\) and \(R_{k_2}(i_2,\tau_2)\not\subseteq R_{k_1}^{\textrm{sup}}(i_1,\tau_1)\).
In order to ensure the cells we will look at are well separated but still not too far apart, we say that any two cells \((k_1,i_1,\tau_1)\) and \((k_2,i_2,\tau_2)\) are \emph{support adjacent} if \(R_{k_1}^{\textrm{2sup}}(i_1,\tau_1)\cap R_{k_2}^{\textrm{2sup}}(i_2,\tau_2)\neq\emptyset\).
We say a cell \((k_1,i_1,\tau_1)\) is \emph{support connected with diagonals} to \((k_2,i_2,\tau_2)\) if there exists a scale 1 cell contained in \(R_{k_1}^{\textrm{2sup}}(i_1,\tau_1)\) and a scale 1 cell contained in \(R_{k_2}^{\textrm{2sup}}(i_2,\tau_2)\), such that the former is diagonally connected to the latter.

Finally, define a sequence of cells \(P=((k_1,i_1,\tau_1),(k_2,i_2,\tau_2),\dots,(k_z,i_z,\tau_z))\) to be a \emph{support connected \(D\)-path} if the cells in \(P\) are mutually well separated and, for each \(j=1,2,\dots,z-1\), \((k_j,i_j,\tau_j)\) is support adjacent or support connected with diagonals to \((k_{j+1},i_{j+1},\tau_{j+1})\). 

For any \(t\), define \(\Omega_t\) to be the set of all \(D\)-paths of cells of scale \(1\) so that the first cell of the path is \((0,0)\) and the last cell of the path is the only cell not contained in the space-time region \([-t,t]^d\times[-t,t]\). Also, define \(\Omega_{\kappa,t}^{\textrm{sup}}\) as the set of all support connected \(D\)-paths of cells of scale at most \(\kappa\) so that the extended support of the first cell of the path contains \(R_1(0,0)\) and the last cell of the path is the only cell whose extended support is not contained in \([-t,t]^d\times[-t,t]\). Then the lemma below shows that we can focus on support connected \(D\)-paths instead of \(D\)-paths with bad ancestry.

\begin{lemma}\label{lemma:support_connected_paths}
	We have that
	\begin{align*}
		&\mathbb{P}\left[\exists P\in\Omega_t\textrm{ s.t. all cells of } P\textrm{ have a bad ancestry}\right]\\
		&\leq\mathbb{P}\left[\exists P\in\Omega_{\kappa,t}^{sup}\textrm{ s.t. all cells of }P\textrm{ are multi-scale bad}\right].
	\end{align*}
\end{lemma}

\begin{proof}
	We complete the proof in two stages. First, we show that if there exists a \(D\)-path \(P\in\Omega_t\), such that each cell of \(P\) has bad ancestry, then there exists a \(D\)-path of multi-scale bad cells of arbitrary scales up to \(\kappa\). Next, we show that, given the existence of such a path of multi-scale bad cells of arbitrary scales up to \(\kappa\), there exists a support connected \(D\)-path of \(\Omega_{\kappa,t}^{\textrm{sup}}\) such that all cells of the path are multi-scale bad.
	
\begin{figure}[!hbt]
\centering
\begin{subfigure}{1\textwidth}
  \centering
  \includegraphics[width=1\linewidth]{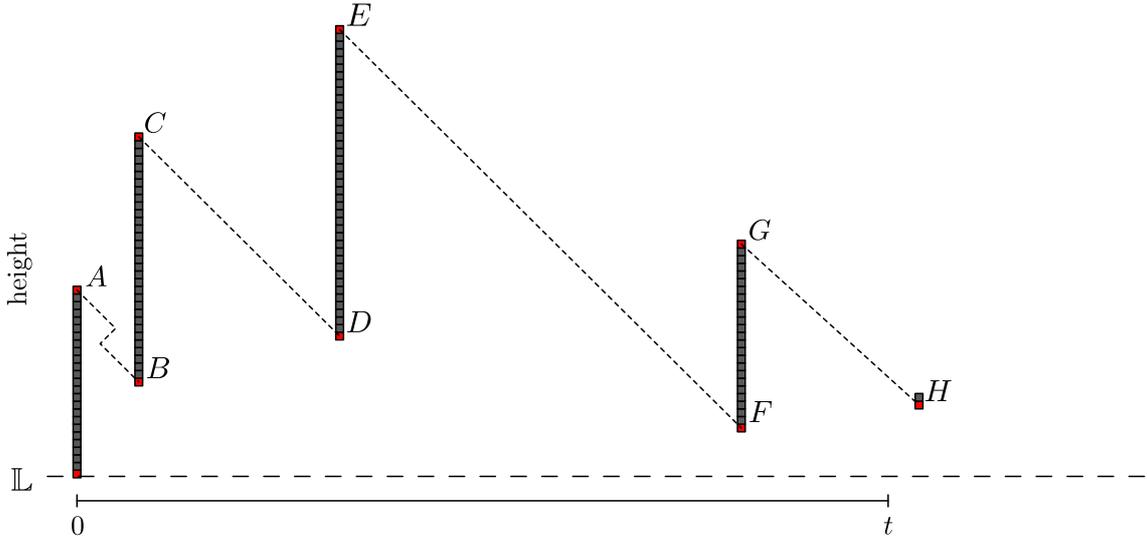}
  \caption{An example of a \(D\)-path where the cells of scale 1 of the path are gray. The dashed lines represent the cells that make a cell of the path diagonally connected to the next cell, i.e. cell \(A\) connected to cell \(B\), cell \(C\) to cell \(D\), etc.}
  \label{fig:Dsub1}
\end{subfigure}
\begin{subfigure}{.48\textwidth}
  \centering
  \includegraphics[width=1\linewidth]{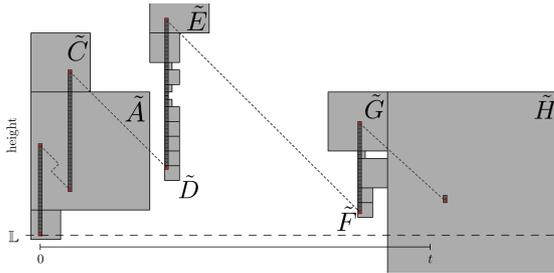}
  \caption{The corresponding \(D\)-path of multi-scale bad cells overlaid in light gray. Note that only the ancestors \(\tilde C,\tilde D\) of the pair \(C,D\) and the ancestors \(\tilde E,\tilde F\) of the pair \(E,F\) remain diagonally connected. The other diagonally connected pairs either share the same ancestor (Cells \(A\) and \(B\) both share the same ancestor \(\tilde A\)) or are adjacent (the pair \(G,H\) is diagonally connected, but the ancestors \(\tilde G,\tilde H\) are adjacent).}
  \label{fig:Dsub2}
\end{subfigure}\hspace{5mm}%
\begin{subfigure}{.48\textwidth}
  \centering
  \includegraphics[width=1\linewidth]{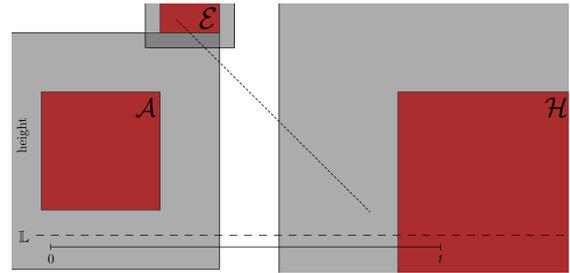}
  \caption{After performing the procedure from Step 2, the remaining three cells form a support connected \(D\)-path, where all three cells are well separated, and \(\mathcal{A}\) is support adjacent to \(\mathcal{E}\) whereas \(\mathcal{E}\) is support connected with diagonals to the \(\mathcal{H}\). Here, \(\mathcal{A}\) is the cell to which we associated \(\tilde A\) in the procedure of Step 2. Similarly, \(\tilde E\) was associated to \(\mathcal{E}\) and \(\tilde H\) was associated to \(\mathcal{H}\). All other cells from (b) were associated to \(\mathcal{A}\), \(\mathcal{E}\) or \(\mathcal{H}\) as well.}
  \label{fig:Dsub3}
\end{subfigure}%

\caption{An example of the procedure that establishes \Cref{lemma:support_connected_paths}. Starting with a \(D\)-path for which all cells have bad ancestries (a), we produce a \(D\)-path of multi-scale bad cells of arbitrary scales (b) in Step 1. Then, in Step 2 we produce a support connected \(D\)-path (c).}
\label{fig:dpath2separated}
\end{figure} 

	\vspace{5pt}\noindent\emph{Step 1:} Let \(\Omega_{\kappa,t}\) be the set of all \(D\)-paths of cells of scale at most \(\kappa\) such that the first cell of the path is an ancestor of \((0,0)\) and the last cell of the path is the only cell whose support is not contained in \([-t,t]^d\times[-t,t]\).
	We now establish that
	\begin{align*}
	&\mathbb{P}[\exists P\in\Omega_t\textrm{ s.t. all cells of } P\textrm{ have a bad ancestry}]\\
	&\leq\mathbb{P}[\exists P\in\Omega_{\kappa,t}\textrm{ s.t. all cells of }P\textrm{ are multi-scale bad}].
	\end{align*}
	Let \(P=((1,i_1,\tau_1),(1,i_2,\tau_2),\dots(1,i_z,\tau_z))\in\Omega_t\) be a \(D\)-path of cells with bad ancestries; therefore \((i_1,\tau_1)=(0,0)\) and \(R_1(i_z,\tau_z)\) is not contained in \([-t,t]^d\times[-t,t]\). For each \(j\), since \(A(i_j,\tau_j)=0\), we know by definition of \(A\) in (\ref{for:A}) that there exists a \(k_j'\) such that, if we set \(i'_j=\pi_1^{(k'_j-1)}(i_j)\) and \(\tau_j'=\gamma_1^{(k_j'-1)}(\tau_j)\), we obtain \(A_{k'_j}(i'_j,\tau'_j)=0\). 
	Now construct \(J\subseteq\{1,2,\dots,z\}\), starting with \(J=\{1,2,\dots,z\}\) and removing elements from \(J\) iteratively as we go from scale \(k=\kappa\) down to scale \(k=1\) using the following rule: 
	if there exists \(j\in J\) such that \(k'_j=k\) and \(A_{k'_j}(i'_j,\tau'_j)=0\), then remove from \(J\) all descendants of \((k'_j,i'_j,\tau'_j)\), except for the first one; i.e. keep in \(J\) only the smallest \(j'\) for which \(i'_j=\pi_1^{(k'_j-1)}(i_{j'})\) and \(\tau_j'=\gamma_1^{(k_j'-1)}(\tau_{j'})\).
	Put differently, \(J\) contains only distinct elements of the set \(\{(k_j',i_j',\tau_j'):j=1,2,\dots z\}\) which have no ancestor within \(J\). With this set we define
	\[
		\tilde P=\{(k'_j,i'_j,\tau'_j):\:j\in J\},
	\]
	and show that \(\tilde P\) is a \(D\)-path. This gives us the existence of a \(D\)-path of multi-scale bad cells of arbitrary scales starting from an ancestor of \((1,i_1,\tau_1)\) and such that the last cell \((k',i',\tau')\in\tilde P\) is an ancestor of \((1,i_z,\tau_z)\), which is not contained in \([-t,t]^d\times[-t,t]\).
	 \Cref{lemma:R_in_R_sup} then gives us that \(R_1(i_z,\tau_z)\) is contained in \(R_{k'}^{\textrm{sup}}(i',\tau')\) so that the union of the supports of the cells in \(\tilde P\) is not contained in \([-t,t]^d\times[-t,t]\).
	 We note that it is possible that the support of some other cell of \(\tilde P\) is also not  contained in \([-t,t]^d\times[-t,t]\). In this case, we modify \(\tilde P\) and remove from it all \(j'\in J\) for which \(j'>j\), where \((k_j,i_j,\tau_j)\) is the first cell of \(\tilde P\) for which \(R_{k_j}^{\textrm{sup}}(i_j,\tau_j)\) is not contained in \([-t,t]^d\times[-t,t]\). Furthermore, it is possible that \(\tilde P\) might contain loops. This does not cause any issues; in fact, the procedure in step 2 will remove any loops should they exist 
	
	Now it remains to verify that \(\tilde P\) satisfies the adjacency properties of a \(D\)-path. By construction, each cell of \(P\) has exactly one ancestor in \(\tilde P\). If we take two adjacent cells \((1,i_j,\tau_j)\), \((1,i_{j+1},\tau_{j+1})\) of \(P\), they either have the same ancestor in \(\tilde P\) or their ancestors are adjacent. This follows from the fact that two non-adjacent cells cannot have descendants of scale 1 that are adjacent. Now assume that \((1,i_j,\tau_j)\in P\) is diagonally connected to \((1,i_{j+1},\tau_{j+1})\in P\). In this case, the two cells either have the same ancestor in \(\tilde P\), have ancestors that are adjacent or the ancestor of \((1,i_j,\tau_j)\) is diagonally connected to the ancestor of \((1,i_{j+1},\tau_{j+1})\).

	\vspace{5pt}\noindent\emph{Step 2:} Here we establish that
	\begin{align*}
	&\mathbb{P}[\exists P\in\Omega_{\kappa,t}\textrm{ s.t. all cells of }P\textrm{ are multi-scale bad}]\\
	&\leq\mathbb{P}[\exists P\in\Omega_{\kappa,t}^{\textrm{sup}}\textrm{ s.t. all cells of }P\textrm{ are multi-scale bad}].
	\end{align*}
	Let \(P=((k_1,i_1,\tau_1),(k_2,i_2,\tau_2),\dots,(k_z,i_z,\tau_z))\in\Omega_{\kappa,t}\) be a \(D\)-path of multi-scale bad cells. We will now show the existence of a support connected \(D\)-path \(P'\) of multi-scale bad cells.	
	First, we order the cells of \(P\) in the following way. If two cells have the same scale, we order them by taking in the same order as they have in \(P\). For two cells of different scales, we say the cell with the larger scale comes before the other cell. This gives us a total order of the cells of \(P\). Next, let \(L\) be the list of cells of \(P\) following this order. We construct \(P'\) step-by-step, by adding the first element of \(L\) to \(P'\) and removing some elements from \(L\), repeating this until \(L\) is empty. While doing this, we associate each cell of \(P\) to a cell of \(P'\), which we will later use to show that using the ordering inherited from \(P\), \(P'\) is a support connected \(D\)-path. Assuming \((k',i',\tau')\) is the current first element of \(L\), the steps taken to construct \(P'\) are as follows:
	\begin{enumerate}
		\item Add \((k',i',\tau')\) to \(P'\) and remove it from \(L\). Associate \((k',i',\tau')\) in \(P\) with itself in \(P'\).
		\item Remove from \(L\) all cells \((k'',i'',\tau'')\) that are not well separated from \((k',i',\tau')\) and associate them to \((k',i',\tau')\).
	\end{enumerate}
	We repeat these steps until \(L\) is empty. We highlight that by construction \(P'\) contains only mutually well separated cells. Let \((k^*,i^*,\tau^*)\) be the cell that \((k_1,i_1,\tau_1)\) is associated to. Note that the extended support of this cell contains \(R_1(0,0)\), because \(R_{k_1}^{\textrm{sup}}(i_1,\tau_1)\) contains \(R_1(0,0)\) and is itself contained in the extended support of \((k^*,i^*,\tau^*)\). We also obtain that

	\[
		\bigcup_{(k',i',\tau')\in P'}R_{k'}^{\textrm{2sup}}(i',\tau')\not\subseteq[-t,t]^d\times[-t,t].
	\]

	This can be argued similarly as above, noting that the support of a cell in \(P\) was not contained in \([-t,t]^d\times[-t,t]\), so the extended support of the cell it is associated to cannot be contained either. Let the cell it is associated to be \((k',i',\tau')\). 
	
	Now it remains to show that there exists a subset of cells \(P''\subseteq P'\) which is a support connected \(D\)-path with diagonals and contains both \((k^*,i^*,\tau^*)\) and \((k',i',\tau')\).
	To see this, we will add some cells from \(P'\) to \(P''\), starting with \((k',i',\tau')\). Let \((k_j,i_j,\tau_j)\) be the first cell of \(P\) that is associated to \((k',i',\tau')\). If \(j=1\) then \((k^*,i^*,\tau^*)=(k',i',\tau')\) and \(P''\) is just this cell. Otherwise, let \((k'',i'',\tau'')\) be the cell \((k_{j-1},i_{j-1},\tau_{j-1})\) is associated to and add it to \(P''\).
	We will show later that 
	\begin{equation}\label{eq:supcon}
		(k'',i'',\tau'')\textrm{ is support adjacent or support connected with diagonals to }(k',i',\tau').
	\end{equation}
	
	Now we iterate the procedure above; that is, we take the first cell \((k_\iota,i_\iota,\tau_\iota)\) of \(P\) that is associated to \((k'',i'',\tau'')\) and either finish the construction of \(P''\) if \(\iota=1\) or continue by taking the cell that \((k_{\iota-1},i_{\iota-1},\tau_{\iota-1})\) is associated to and adding it to \(P''\).
	Note that \(\iota<j\), which guarantees that this procedure will eventually add \((k^*,i^*,\tau^*)\) to \(P''\), thus completing the construction. It is possible that the extended support of some cell of \(P''\) other than \((k',i',\tau')\) to not be contained in \([-t,t]^d\times[-t,t]\). As in step 1, we simply remove from \(P''\) all cells \((k_j,i_j,\tau_j)\) that come after the first such cell.
	
	It remains to show (\ref{eq:supcon}) holds.
	Assume for the following that \(k_{j-1}\geq k_{j}\), the converse can be argued the same way, and
	recall that in \(P\), two consecutive cells \((k_{j-1},i_{j-1},\tau_{j-1})\) and \((k_{j},i_{j},\tau_{j})\) are either adjacent or the first cell is diagonally connected to the second cell.
	Since \(k_{j-1}\geq k_{j}\), we have a cell \((\hat k, \hat i,\hat \tau)\) at scale \(\hat k = k_{j-1}\) that is an ancestor of \((k_{j},i_{j},\tau_{j})\), to which \((k_{j-1},i_{j-1},\tau_{j-1})\) is either adjacent or diagonally connected. In the first case, we have
	by \Cref{lemma:R_in_R_sup} that \(R^{\textrm{sup}}_{\hat k}(\hat i,\hat \tau)\) contains both \(R_{k_{j-1}}(i_{j-1},\tau_{j-1})\) and \(R_{k_{j}}(i_{j},\tau_{j})\). 
	Since \(\hat k= k_{j-1}\) and \((k_{j-1},i_{j-1},\tau_{j-1})\) is associated to \((k'',i'',\tau'')\), we have that \(\hat k\leq k''\). 
	Then, by \Cref{remark:contain}, we have that \(R^{\textrm{sup}}_{\hat k}(\hat i,\hat \tau)\subseteq R_{k''}^{\textrm{2sup}}(i'',\tau'')\), which gives that \(R_{k''}^{\textrm{2sup}}(i'',\tau'')\) intersects \(R_{k'}^{\textrm{2sup}}({i'},{\tau'})\).
	Alternatively, \((k_{j-1},i_{j-1},\tau_{j-1})\) is diagonally connected to \((k_{j},i_{j},\tau_{j})\). 
	This gives that \((k'',i'',\tau'')\) is support connected with diagonals to \((k',i',\tau')\) with the same diagonal steps that make \((k_{j-1},i_{j-1},\tau_{j-1})\) diagonally connected to \((k_j,i_j,\tau_j)\).
\end{proof}

The next lemma is a technical result bounding the probability that a random walk on a weighted graph remains inside a cube.

\begin{lemma}\label{lemma:exit_times}
	Let \(\Delta>0\) and, for any \(z>0\), define \(F_{\Delta}(z)\) to be the event that a random walk on \((G,\mu)\) starting from the origin stays inside \(Q_z\) throughout the time interval \([0,\Delta]\). Then, on a uniformly elliptic graph, there exist constants \(c\), \(c_1\) and \(c_2\) such that if \(\Delta>cz\), we have
	\[
		\mathbb{P}[F_{\Delta}(z)]\geq 1-c_1\exp\{-c_2 z^2/\Delta\}.
	\]
\end{lemma}

\begin{proof}
	The result is a reformulation of the exit time result for random walks on weighted graphs from \cite{Barlow2004} and \cite{Hambly2009} by taking a ball with radius \(z/2\) that is contained in \(Q_z\) and using that the weights \(\mu_{x,y}\) are uniformly elliptic.
\end{proof}

We now give a lemma that will be used to control the dependencies involving well separated cells. Let \(\mathcal{F}_k(i,\tau)\) be the \(\sigma\)-field generated by all \(A_{k'}(i',\tau')\) for which \(T_{k'}^{\textrm{inf}}(\tau')\) does not intersect \([\gamma_k^{(1)}(\tau)\beta_{k+1},\infty)\) or both \(\tau'\beta_{k'}\leq \tau\beta_k\) and \(S_k^{\textrm{inf}}(i)\cap S_{k'}^{\textrm{inf}}(i')=\emptyset\). Furthermore, recall the value \(w\) from (\ref{for:eta}), which has until now been assume to be an arbitrary positive value. We define the following two quantities:
\begin{align}
	\psi_1&=\min\left\{\epsilon^2\lambda_0\ell^d C_{M}^{-1},\log\left(\frac{1}{1-\nu_{E_{\textrm{st}}}((1-\epsilon)\lambda,Q_{(2\eta+1)\ell},Q_{w\ell},\eta\beta)}\right)\right\}\nonumber\\
	\psi_k&=\frac{\epsilon^2\lambda_0\ell_{k-1}^d}{(k+1)^4}=\frac{\epsilon^2\lambda_0 \ell^d m^{d(k-2)}((k-1)!)^{ad}}{(k+1)^4}\quad\textrm{for }k\geq 2.\label{eq:def_psi}
\end{align} 
We now give a short intuitive explanation behind \(\psi_1\) and \(\psi_k\). Note that \(\psi_k\) is increasing in \(k\); this can easily be verified by observing that in the right-most expression for \(\psi_k\), \(ad>4\) for all \(d\geq 2\). Intuitively, one can think of \(\psi_k\) as the ``weight'' of a space-time super cell of scale \(k\). Furthermore, across all \(k\geq 2\), we can increase how much the super cells ``weigh'' by increasing the size of the tessellation (by making \(\ell\) larger) or by increasing the density of particles (by increasing \(\lambda_0\)). This holds also for super cells of scale \(1\); note however that in order to make the weight of a super cell of scale 1 large, we also need to ensure the second term of the minimum in (\ref{eq:minalpha}) is made large (say larger than some value \(\alpha\)). That is, we need to make \(\mathbb{P}[E(i,\tau)=0]\leq e^{-\alpha}\) given that there is at least a Poisson point process with intensity \((1-\epsilon)\lambda\) of particles inside of the cube \(Q_{(2\eta+1)\ell}\) and this particles have displacement in \(Q_{w\ell}\) during a time interval of length \(\eta\beta\).

\begin{lemma}\label{lemma:failure_prob}
	Let \(w\geq\sqrt{\frac{\eta\beta}{c_2\ell^2}\log\left(\frac{8c_1}{\epsilon}\right)}\) and 
	\[\alpha=\min\left\{\epsilon^2\lambda_0\ell^d C_{M}^{-1},\log\left(\frac{1}{1-\nu_{E_{\textrm{st}}}((1-\epsilon)\lambda,Q_{(2\eta+1)\ell},Q_{w\ell},\eta\beta}\right)\right\}.\] 
	If \(m\) is sufficiently large with respect to \(d\), \(\beta/\ell^2\), \(\epsilon\), \(w\) and \(C_M\), then there are positive constants \(c=c(C_M)\geq 1\) and \(\alpha_0\) so that, for all \(\alpha\geq \alpha_0\), all cells \((k,i,\tau)\) and any \(F\in\mathcal{F}_k(i,\tau)\), we have
	\begin{enumerate}
		\item \(\mathbb{P}[A_k(i,\tau)=0]\leq\exp\{-c\psi_k\}\), for all \(k=1,2,\dots,\kappa\)
		\item \(\mathbb{P}[A_k(i,\tau)=0\;|\; F]\leq\exp\{-c\psi_k\}\), for all \(k=1,2,\dots,\kappa-1\).
	\end{enumerate}
\end{lemma}

\begin{proof}
	Note that \(A_k\) is defined differently for \(k=1\) and \(2\leq k\leq\kappa-1\). We will first prove the result for \(k\geq 2\) and establish part 2 of the lemma. Since
	\[
		\mathbb{P}[A_k(i,\tau)=0\;|\;F]=\mathbb{P}[\{D_k^{\textrm{ext}}(i,\tau)=0\}\cap\{D_k^{\textrm{base}}(i,\tau)=1\}\;|\;F],
	\]
	if \(F\cap\{D_k^{\textrm{base}}(i,\tau)=1\}=\emptyset\), then the lemma holds. We now assume \(F\cap\{D_k^{\textrm{base}}(i,\tau)=1\}\neq\emptyset\) and write
	\[
		\mathbb{P}[A_k(i,\tau)=0\;|\;F]\leq\mathbb{P}[\{D_k^{\textrm{ext}}(i,\tau)=0\}\;|\;F\cap\{D_k^{\textrm{base}}(i,\tau)=1\}].
	\]
	Recall that \(\{D_k^{\textrm{base}}(i,\tau)=1\}\) gives that all cubes \(S_k(i')\) of scale \(k\) contained in \(S_k^{\textrm{base}}(i)\) have at least \((1-\epsilon_{k+1})\lambda_0\sum_{y\in S_k(i')}\mu_y\) particles at time \(\gamma_k^{(1)}(\tau)\beta_{k+1}\) and the displacement of these particles throughout \([\gamma_k^{(1)}(\tau)\beta_{k+1},\tau\beta_k]\) is in \(Q_{\eta mn(k+1)^a\ell_k}\). Remember that \(F\) reveals only information about the location of these particles before time \(\gamma_k^{(1)}(\tau)\beta_{k+1}\) since these particles never leave the cube \(S_k^{\textrm{inf}}(i)\) during the whole \([\gamma_k^{(1)}(\tau)\beta_{k+1},\tau\beta_k]\).

	We now apply \Cref{thrm:mixing} and denote the variables appearing in the statement of that theorem with a bar. We apply the theorem with
	\begin{align*}
		\bar K&=(1+2\eta mn(k+1)^a)\ell_k,\\
		\bar\ell&=\ell_k,\\
		\bar\beta&=(1-\epsilon_{k+1})\lambda_0,\\
		\bar\Delta&=\tau\beta_k-\gamma_k^{(1)}(\tau)\beta_{k+1}\in[\beta_{k+1},2\beta_{k+1}],\\
		\bar K'&\qquad\textrm{such that }\bar K-\bar K'=\eta mn(k+1)^a\ell_k, \textrm{ and}\\
		\bar \epsilon&\qquad\textrm{such that }(1-\bar\epsilon)(1-\epsilon_{k+1})=\left(1-\tfrac{(\epsilon_{k+1}+\epsilon_k)}{2}\right).
	\end{align*}
	This gives that \(\bar\epsilon\geq\frac{\epsilon_k-\epsilon_{k+1}}{2}=\frac{\epsilon}{2(k+1)^2}\). Using these values and the fact that \(m\) is large enough, we have that
	\[
		\bar K'=\ell_k+\eta mn(k+1)^a\ell_k\geq \ell_k+2\eta mnk^a\ell_{k-1},
	\]
	which is the side length of \(S_k^{\textrm{ext}}(i)\). We also have \(\bar \Delta\geq\beta_{k+1}\geq c_0\frac{\bar\ell^2}{\bar\epsilon^{4/\Theta}}\) since \(C_{\textrm{mix}}\geq 2^{4/\Theta}c_0\) in the definition of \(\beta_{k+1}\). We still have to check whether \(\bar K-\bar K'\geq C\sqrt{\bar\Delta\log\bar\Delta}\), which is equivalent to checking that
	\[
		\eta mn(k+1)^a\ell_k\geq \tilde C\sqrt{\beta_{k+1}\log\beta_{k+1}}
	\]
	for some constant \(\tilde C\). Using the definitions of \(\ell_k\) and \(\beta_{k+1}\), this inequality can be rewritten as
	\[
		\eta mn(k+1)^a \ell_k\geq \tilde C\sqrt{C_\textrm{mix}}\frac{1}{\epsilon^{2/\Theta}}(k+1)^{4/\Theta}\ell_k\sqrt{\log\beta_{k+1}}.
	\]
	Now, using the value of \(\beta_{k+1}\) and \(\ell_k\) we obtain that there exists a constant \(C\) independent of \(k\) and \(m\), but depending on \(\epsilon\) such that
	\[
		\frac{\tilde C\sqrt{C_{\textrm{mix}}}}{\epsilon^{2/\Theta}}\sqrt{\log\beta_{k+1}}\leq C\sqrt{\log k+k\log m+a\log k!+\log\ell}.
	\]
	Therefore, it remains to check that 
	\[
		\eta mn(k+1)^{a-4/\Theta}\geq C\sqrt{\log k+k\log m + a\log k!+\log\ell}.
	\]
	Since \(a-4/\Theta>\frac{1}{2}\) by (\ref{eq:theta_a}) and \(\log k!\leq k\log k\), \((k+1)^{a-4/\Theta}\) is larger than the right-hand side above for all large enough \(k\). Then, since \(\epsilon\) is fixed, setting \(m\) large enough makes the above inequality true for all \(k\geq 1\).

	Hence, we obtain a coupling between the particles that end up in \(S_k^{\textrm{ext}}(i)\) and an independent Poisson point process \(\Xi\) with intensity measure \(\zeta (y)=(1-\bar \epsilon)(1-\epsilon_{k+1})\lambda_0\mu_y=(1-\frac{\epsilon_k}{2}-\frac{\epsilon_{k+1}}{2})\lambda_0\mu_y\) that succeeds with probability at least
	\begin{align}
		&1-\sum_{y\in S_k^{\textrm{ext}}(i)}\exp\{-\bar C(1-\epsilon_{k+1})\lambda_0\mu_y\bar\epsilon^2\bar\Delta^{d/2}\}\nonumber\\
		&\geq1-\sum_{y\in S_k^{\textrm{ext}}(i)}\exp\left\{-\bar C C_{\textrm{mix}}^{d/2}\lambda_0\mu_y\frac{\epsilon^2}{4(k+1)^4}\ell_k^d\right\}\nonumber\\
		&\geq1-(\ell_k+2\eta mnk^a\ell_{k-1})^d\exp\left\{-\bar C C_{\textrm{mix}}^{d/2}\lambda_0 C_M^{-1}\frac{\epsilon^2}{4(k+1)^4}\ell_{k-1}^d\right\}\nonumber\\
		&\geq 1-\tfrac{1}{2}\exp\{-c\psi_k\}\label{eq:cm}
	\end{align}
	where \(c\) is constant independent of \(\ell\), \(k\) and \(\epsilon\), and we used that \(\bar\Delta\geq\beta_{k+1}\geq C_{\textrm{mix}}\ell^2_k\geq C_{\textrm{mix}}\ell_{k-1}^2\). The last inequality holds for large \(k\) by setting \(m\) large, since \(\ell_{k-1}=m^{k-2}((k-1)!)^a\ell\). Similarly, for small \(k\geq 2\) the inequality holds since \(C_M^{-1}\epsilon^2\lambda_0\ell^d\geq\alpha\) is assumed large enough.

	Now, for the case \(k\geq 2\), define a Poisson point process \(\Xi'\) consisting of those particles of \(\Xi\) whose displacement throughout \([\tau\beta_k,(\tau+2)\beta_k]\) is in \(Q_{\eta mnk^a\ell_{k-1}}\). For each particle of \(\Xi\), this condition is satisfied with probability \(\mathbb{P}[F_{2\beta_k}(\eta mnk^a\ell_{k-1})]\), independently over the particles of \(\Xi\). Using \Cref{lemma:exit_times} and the thinning property of Poisson processes, we have that \(\Xi'\) is a Poisson point process with intensity measure
	\[
		\zeta'(y)= (1-\bar\epsilon)(1-\epsilon_{k+1})\mathbb{P}[F_{2\beta_k}(\eta mnk^a\ell_{k-1})]\lambda_0\mu_y,
	\]
	which is greater than
	\begin{align*}
		&\left(1-\frac{\epsilon_k}{2}-\frac{\epsilon_{k+1}}{2}\right)\left(1-c_1\exp\left\{-c_2\frac{(\eta mnk^a\ell_{k-1})^2}{2\beta_k}\right\}\right)\lambda_0\mu_y\\
		&\geq\left(1-\frac{\epsilon_k}{2}-\frac{\epsilon_{k+1}}{2}\right)\left(1-c_1\exp\left\{-c_2\frac{(\eta mnk^a)^2\epsilon^{4/\Theta}}{2C_{\textrm{mix}}k^{8/\Theta}}\right\}\right)\lambda_0\mu_y\\
		&\geq\left(1-\frac{\epsilon_k}{2}-\frac{\epsilon_{k+1}}{2}\right)\left(1-c_1\exp\left\{-c_2\frac{(\eta n)^2k}{2(\beta/\ell^2)}\right\}\right)\lambda_0\mu_y,
	\end{align*}
	where the first inequality follows from the definition of \(\beta_k\) and the second inequality follows from the condition \(2a-8/\Theta>1\) in (\ref{eq:theta_a}) and from \(C_{\textrm{mix}}=\frac{\beta\epsilon^{4/\Theta}m^2}{\ell^2}\), which is obtained by setting \(\beta_1=\beta\) in (\ref{for:beta}). Setting \(m\), and thus \(n\), sufficiently large with respect to \(\beta\), \(\epsilon\), \(\eta\) and the constants \(c_1\) and \(c_2\), we obtain that
	\begin{align*}
		\zeta'(y)&\geq\left(1-\frac{\epsilon_k}{2}-\frac{\epsilon_{k+1}}{2}\right)\left(1-\frac{(\epsilon_k-\epsilon_{k+1})}{4}\right)\lambda_0\mu_y\\
		&\geq\left(1-\frac{3\epsilon_k}{4}-\frac{\epsilon_{k+1}}{4}\right)\lambda_0\mu_y.
	\end{align*}
	Conditioning on the coupling above, we obtain that \(D_k^{\textrm{ext}}(i,\tau)=1\) with probability at least
	\begin{align}
		& 1-\sum_{i':S_{k-1}(i')\subseteq S_k^{\textrm{ext}}(i)}\exp\left\{-\frac{1}{2}\left(\frac{\epsilon_k-\epsilon_{k+1}}{4}\right)^2\left(1-\frac{3\epsilon_k}{4}-\frac{\epsilon_{k+1}}{4}\right)\lambda_0\sum_{y\in S_{k-1}(i')}\mu_y\right\}\nonumber\\
		&\geq 1-\sum_{i':S_{k-1}(i')\subseteq S_k^{\textrm{ext}}(i)}\exp\left\{-\frac{1}{2}\left(\frac{\epsilon^2}{16(k+1)^4}\right)\left(1-\frac{3\epsilon_1}{4}-\frac{\epsilon_{2}}{4}\right)\lambda_0\sum_{y\in S_{k-1}(i')}\mu_y\right\}\nonumber\\
		&\geq 1-(mk^a+2\eta mnk^a)^d\exp\left\{-\frac{1}{2}\left(\frac{\epsilon^2}{16(k+1)^4}\right)\left(1-\frac{15\epsilon}{16}\right)\lambda_0 C_M^{-1}\ell_{k-1}^d\right\}\label{for:P_D_F}\\
		&\geq 1-\tfrac{1}{2}\exp\{-c\psi_k\}\nonumber
	\end{align}
	for some constant \(c\), where in the fist step we applied Chernoff's bound from \Cref{lem:chernoff} with \(\delta=\frac{\epsilon_k-\epsilon_{k+1}}{4}\), using that \((1-\delta)\left(1-\frac{3\epsilon_k}{4}-\frac{\epsilon_{k+1}}{4}\right)\geq1-\epsilon_k\). In the second step, we used that \(\epsilon_k\) is decreasing with \(k\). The last inequality holds using the same argument as the one following (\ref{eq:cm}).
	This and (\ref{eq:cm}) establishes part 2 for \(k\geq 2\).
	
	For part 2 with \(k=1\) we again use the Poisson point process \(\Xi\) of intensity measure
	\[
		\zeta(y)\geq\left(1-\frac{\epsilon_k}{2}-\frac{\epsilon_{k+1}}{2}\right)\lambda_0\mu_y=\left(1-\frac{7\epsilon}{8}\right)\lambda_0\mu_y
	\]
	over \(S_1^{\textrm{ext}}(i)\) as defined above. We also use the fact that \(E(i,\tau)\) is an event restricted to the super cell \(i\) and \(S_1^{\textrm{ext}}(i)\) contains the super cell \(i\) (see \Cref{rem:supercube_in_extended_cube}). Recall that, for the event \(E(i,\tau)\), we only consider the particles of \(\Xi\) whose displacement from time \(\tau\beta\) to \((\tau+\eta)\beta\) is inside \(Q_{w\ell}\). Let the event that this happens for a given particle of \(\Xi\) be denoted by \(F_{\eta\beta}(w\ell)\). Then, we apply \Cref{lemma:exit_times} with \(\Delta=\eta\beta\) and \(z=w\ell\) to obtain
	\[
		\mathbb{P}[F_{\eta\beta}(w\ell)]\geq 1-c_1\exp\left\{-c_2\frac{(w\ell)^2}{\eta\beta}\right\}.
	\]
	Using the fact that \(w^2\ell^2\geq\frac{1}{c_2}\eta\beta\log(8c_1\epsilon^{-1})\), we have that \(\mathbb{P}[F_{\eta\beta}(w\ell)]\geq1-\frac{\epsilon}{8}\). Therefore, using thinning, we have that the particles of \(\Xi\) for which \(F_{\eta\beta}(w\ell)\) hold consist of a Poisson point process with intensity at least \(\left(1-\frac{7\epsilon}{8}\right)\left(1-\frac{\epsilon}{8}\right)\lambda_0\mu_y\geq\left(1-\epsilon\right)\lambda_0\mu_y\). Since \(E(i,\tau)\) is increasing, we have that
	\begin{equation*}
		\mathbb{P}[E(i,\tau)=0\;|\;F\cap\{D_k^{\textrm{base}}(i,\tau)=1\}]\leq 1-\nu_{E_{\textrm{st}}}((1-\epsilon)\lambda,Q_{(2\eta+1)\ell},Q_{w\ell},\eta\beta)\leq e^{-\alpha}.
	\end{equation*}
	
	A similar argument as above can be used to establish part 1 with \(k<\kappa\). For \(k=\kappa\), the argument is simpler as we do not need to carry out the coupling procedure.
\end{proof}

Later, in \Cref{section:proof_of_prop}, we will use \Cref{lemma:failure_prob} to bound the probability that a path \(P\in\Omega_{\kappa,t}^{\textrm{sup}}\) of multi-scale bad cells exists. 
We will use a uniform bound to control the probability that at least one of the space-time cell of scale \(\kappa\) is multi-scale bad. In the converse case, where all scale \(\kappa\) cells are multi-scale good (i.e. not multi-scale bad), we will need to consider paths in \(\Omega_{\kappa-1,t}^{\textrm{sup}}\) and count how many such paths exist. To that end, we now show some bounds that hold for  paths in \(\Omega_{\kappa-1,t}^{\textrm{sup}}\).

\begin{lemma}\label{lemma:failure_prob_total}
	Assume the conditions of \Cref{lemma:failure_prob} are satisfied and let \(P\in\Omega_{\kappa-1,t}^{\textrm{sup}}\) be the support connected \(D\)-path \(((k_1,i_i,\tau_1),\dots,(k_z,i_z,\tau_z))\). Then, with \(\psi_k\) as defined in (\ref{eq:def_psi}) there exists a constant \(c_3\), such that we have
	\[
		\mathbb{P}\left[\textstyle \bigcap_{j=1}^z\{A_{k_j}(i_j,\tau_j)=0\}\right]\leq\exp\left\{-c_3\sum_{j=1}^z\psi_{k_j}\right\}.
	\]
\end{lemma}

\begin{proof}
	We derive the probability that all cells of \(P\) are multi-scale bad. Consider the following order of cells of \(P\). First, take an arbitrary order of \(\mathbb{Z}^d\). Then, we say that \((k_j,i_j,\tau_j)\) precedes \((k_{j'},i_{j'},\tau_{j'})\) in the order if \(\tau_j\beta_{k_j}<\tau_{j'}\beta_{k_{j'}}\) or if both \(\tau_j\beta_{k_j}=\tau_{j'}\beta_{k_{j'}}\) and \(i_j\) precedes \(i_{j'}\) in the order of \(\mathbb{Z}^d\). Then, for any \(j\), we let \(J_j\) be a subset of \(\{1,2,\dots,z\}\) containing all \(j'\) for which \((k_{j'},i_{j'},\tau_{j'})\) precedes \((k_j,i_j,\tau_j)\) in the order. Using this order, we write
	\[
		\mathbb{P}\left[\textstyle\bigcap_{j=1}^z\{A_{k_j}(i_j,\tau_j)=0\}\right]\leq \prod_{j=1}^z\mathbb{P}\left[A_{k_j}(i_j,\tau_j)=0\;|\;\textstyle\bigcap_{j'\in J_j}\{A_{k_{j'}}(i_{k_j'},\tau_{k_j'})=0\}\right].
	\]
	Note that, for each \(j'\in J_j\), we have that \((k_j,i_j,\tau_j)\) and \((k_{j'},i_{j'},\tau_{j'})\) are well separated. Using the definition of well separated cells, we have that \(R_{k_{j'}}^{\textrm{inf}}(i_{j'},\tau_{j'})\not\subseteq R_{k_j}^{\textrm{sup}}(i_j,\tau_j)\) and \(R_{k_{j}}^{\textrm{inf}}(i_{j},\tau_{j})\not\subseteq R_{k_{j'}}^{\textrm{sup}}(i_{j'},\tau_{j'})\). Hence, we obtain by \Cref{lemma:R_inf_not_R_sup} that \(R_{k_{j'}}^{\textrm{inf}}(i_{j'},\tau_{j'})\cap R_{k_{j}}^{\textrm{inf}}(i_{j},\tau_{j})=\emptyset\). By the ordering above, we also have \(\tau_j\beta_{k_j}\geq \tau_{j'}\beta_{k_{j'}}\), which gives that the event \(\bigcap_{j'\in J_j}\{A_{k_{j'}}(i_{k_j'},\tau_{k_j'})=0\}\) is measurable with respect to \(\mathcal{F}_{k_j}(i_j,\tau_j)\). Then, we apply \Cref{lemma:failure_prob} to obtain a positive constant \(c_3\) such that
	\[
		\mathbb{P}\left[\textstyle\bigcap_{j=1}^z\{A_{k_j}(i_j,\tau_j)=0\}\right]\leq\exp\left\{-c_3\sum_{j=1}^z\psi_{k_j}\right\}.
	\]
	\end{proof}

\begin{lemma}\label{lemma:number_of_paths}
	Let \(z\) be a positive integer and \(k_1,k_2,\dots,k_z\geq1\) be fixed. Then, if \(\alpha\) and \(m\) are sufficiently large, the total number of support connected \(D\)-paths containing \(z\) cells of scales \(k_1,k_2,\dots,k_z\) is at most \(\exp\left\{\tfrac{c_3}{2}\sum_{j=1}^z\psi_{k_j}\right\}\), where \(c_3\) is the same constant as in \Cref{lemma:failure_prob_total}, \(\psi\) is defined in (\ref{eq:def_psi}) and \(\alpha\) is defined in \Cref{lemma:failure_prob}.
\end{lemma}

\begin{proof}
	Recall that for two consecutive cells of a support connected \(D\)-path, they are either support adjacent or the first cell is support connected with diagonals to the second; see the beginning of this section for details. For the remainder of this proof, when a cell \((k,i,\tau)\) of a support connected \(D\)-path \(P\) is support connected with diagonals to the next cell \((k',i',\tau')\) of \(P\), we will refer to the scale \(1\) cells forming the diagonal connection between a cell contained in \(R_k^{\textrm{2sup}}(i,\tau)\) and a cell contained in \(R_{k'}^{\textrm{2sup}}(i',\tau')\) as the \emph{diagonal steps}. Note also that by the definition of \(D\)-paths from \Cref{section:DpathsAndBadClusters}, the first cell of the diagonal steps is diagonally connected to the last cell of the diagonal steps. 
	
	We will prove the result in three steps. 
	We will first show an upper bound for the number of support connected \(D\)-paths with no diagonal steps.
	Next, we will prove a bound for the number of support connected \(D\)-paths where the first and last cell of each sequence of diagonal steps is fixed and show that this bound is directly linked to the bound from the first step. 
	Finally, we will prove an upper bound for the number of all possible arrangements of the first and last cell of the diagonal steps for each diagonal, which will then, when combined with the bound from step two, prove the lemma.
	
	We begin with the first step, by considering the number of possible support connected \(D\)-paths when each cell of the \(D\)-path is support adjacent to the next cell, that is, there are no diagonal steps in the \(D\)-path.
	For any \(k,k'\geq 1\), define
	\[
		\Phi_{k,k'}=\max_{(i_1,\tau_1)\in\mathbb{Z}^{d+1}}|\{(i_2,\tau_2)\in\mathbb{Z}^{d+1}:(k,i_1,\tau_1)\textrm{ is support adjacent to }(k',i_2,\tau_2)\}|,
	\]
	that is, \(\Phi_{k,k'}\) is the maximum number of cells of scale \(k'\) that are support adjacent to a given cell of scale \(k\).
	Let \(\chi_k\) be the number of cells of scale \(k\) whose extended support contains \(R_1(0,0)\). 
	This gives that the total number of different \(D\)-paths of \(z\) cells of scales \(k_1,\dots,k_z\) with no diagonal steps can be bound above by
	\[
		\chi_{k_1}\prod_{j=2}^z\Phi_{k_{j-1},k_j}.
	\]

	 Now we derive a bound for \(\chi_k\). At scale \(k\), the number of cells that have the same extended support is \(\left(\frac{\ell_{k+1}}{\ell_k}\right)^d\frac{\beta_{k+1}}{\beta_k}=m^{d+2}k^{2a-8/\Theta}(k+1)^{8/\Theta+ad}\). Furthermore, the extended support of a cell of scale \(k\) contains exactly \(27(2(3m+1)+1)^d\) different cells of scale \(k+1\). Thus, the number of different extended supports for a cell of scale \(k\) that contains \(R_1(0,0)\) is bounded above by
	\[
		\chi_k\leq 27(2(3m+1)+1)^d m^{d+2}k^{2a-8/\Theta}(k+1)^{8/\Theta+ad}\leq\exp\left\{\frac{c_3}{16}\psi_k\right\},
	\]
	where the last inequality holds since \(m\) and \(\alpha\) are large enough.
	To derive a bound for \(\Phi_{k,k'}\), fix a cell \((k,i_1,\tau_1)\) of scale \(k\). Now, a cell of scale \(k'\) can only be support adjacent to \((k,i_1,\tau_1)\) if it is inside the region
	\begin{equation}\label{eq:double2sup}
		\bigcup_{x\in R_k^{\textrm{2sup}}(i_1,\tau_1)}\left(x+[-(3m+2)\ell_{k'+1},(3m+2)\ell_{k'+1}]^d\times[-14\beta_{k'+1},14\beta_{k'+1}]\right).
	\end{equation}
	For \(k\geq k'\), let \(\overline{\Phi_{k,k'}}\) be the number of cells of scale \(k'\) that lie in the region above. We then have that \(\Phi_{k,k'}\leq\overline{\Phi_{k,k'}}\)
	and
	\begin{align*}
		\overline{\Phi_{k,k'}}&=\left(\frac{(6m+3)\ell_{k+1}+2(3m+2)\ell_{k'+1}}{\ell_{k'}}\right)^d\left(\frac{27\beta_{k+1}+28\beta_{k'+1}}{\beta_{k'}}\right)\\
		&\leq\left((6m+3)m^{k-k'+1}\prod_{i=k'+1}^{k+1}i^a+2(3m+2)m(k'+1)^a\right)^d\\
		&\times \left(27m^{2(k-k'+1)}\prod_{i=k'}^{k}i^{2a-8/\Theta}(i+1)^{8/\Theta}+28m^2k'^{2a-8/\Theta}(k'+1)^{8/\Theta}\right)\\
		&\leq c_4m^{(k-k'+2)d}k^{kad}m^{2(k-k'+1)}k^{2ka}\leq c_4 m^{(d+2)(k-k'+2)}k^{(ad+2a)k},
	\end{align*}
	for some universal positive constant \(c_4\). Note that for any constant \(c>0\), since \(m\) and \(\alpha\) are large enough, it holds that \(c\overline{\Phi_{k,k'}}\leq c\overline{\Phi_{k,1}}\leq\exp\left\{\frac{c_3}{16}\psi_k\right\}\). For \(k<k'\) we set \(\overline{\Phi_{k,k'}}=2^{d+1}\overline{\Phi_{k',k}}\), which gives using (\ref{eq:double2sup})
	that \(\Phi_{k,k'}\leq2^{d+1}\overline{\Phi_{k',k}}\leq \exp\left\{\frac{c_3}{16}\psi_{k'}\right\}\).

	Observe now that 
	\begin{equation}\label{for:phi_phi}
		\prod_{j=2}^z\overline{\Phi_{k_{j-1},k_j}}\leq\prod_{j=2}^z\left(\overline{\Phi_{k_{j-1},k_{j}}}\mathbbm{1}_{k_{j-1}\geq k_{j}}+\mathbbm{1}_{k_{j-1}<k_{j}}\right)\left(\overline{\Phi_{k_{j},k_{j-1}}}\mathbbm{1}_{k_j\geq k_{j-1}}+\mathbbm{1}_{k_j< k_{j-1}}\right).
	\end{equation}
	If write \(k_0=k_{z+1}=\infty\), the right hand side of (\ref{for:phi_phi}) can be written as
	\[
		\prod_{j=1}^z\left(\overline{\Phi_{k_{j},k_{j-1}}}\mathbbm{1}_{k_{j}\geq k_{j-1}}+\mathbbm{1}_{k_{j}<k_{j-1}}\right)\left(\overline{\Phi_{k_{j},k_{j+1}}}\mathbbm{1}_{k_j\geq k_{j+1}}+\mathbbm{1}_{k_j< k_{j+1}}\right).
	\]
	
	Then, applying the bounds above for \(\Phi\) and \(\chi\), we obtain 
	\[
		\chi_{k_1}\prod_{j=2}^z\Phi_{k_{j-1},k_j}\leq\chi_{k_1}\overline{\Phi_{k_1,1}}\prod_{j=2}^z\left(2^{d+1}\overline{\Phi_{k_{j},1}}\right)^2\leq\exp\left\{\frac{c_3}{8}\sum_{j=1}^z\psi_{k_j}\right\}.
	\]
	
	
	We now proceed to the second step.
	By definition, a cell \((k,i,\tau)\) can only be support connected with diagonals to \((k',i',\tau')\) if there exists a cell 
	\((1,i'',\tau'')\) for which \(R_1(i'',\tau'')\subseteq R_{k}^{\textrm{2sup}}(i,\tau)\) that is diagonally connected to a cell \((1,\hat i,\hat\tau)\) for which \(R_1(\hat i,\hat\tau)\subseteq R_{k'}^{\textrm{2sup}}(i',\tau')\). Define \((\hat i-i'',\hat\tau-\tau'')\in\mathbb{Z}^{d+1}\) to be the \emph{relative position} of the cell \((1,\hat i,\hat\tau)\) with respect to the cell \((1,i'',\tau'')\). For convenience we will write when \((k,i,\tau)\) is adjacent to \((k',i',\tau')\) that the relative position of \((1,\hat i,\hat\tau)\) with respect to \((1,i'',\tau'')\) is \((0,0)\). We will show a bound for the number of such relative positions in a \(D\)-path in step three, so we now proceed to show a bound for the number of \(D\)-paths that have fixed relative positions of \((1,\hat i,\hat\tau)\) with respect to \((1,i'',\tau'')\) for all consecutive pairs of cells in the path.
	
	Let \((k,i,\tau)\) be a cell of the support connected \(D\)-path that is support adjacent or support connected with diagonals to the cell \((k',i',\tau')\) and let \((\hat i-i'',\hat\tau-\tau'')\in\mathbb{Z}^{d+1}\) be the relative position, as above.
	Then, for a fixed relative position \((\hat i-i'',\hat\tau-\tau'')\), define 
	\begin{align*}
		\Phi^*_{k,k'}&=\max_{(i_1,\tau_1)\in\mathbb{Z}^{d+1}}|\{(i_2,\tau_2)\in\mathbb{Z}^{d+1}:(k,i_1,\tau_1)\textrm{ is support adjacent or }\\
		&\qquad\qquad\qquad\quad\textrm{support connected with diagonals to }(k',i_2,\tau_2)\\
		&\qquad\qquad\qquad\quad\textrm{with fixed relative position }(\hat i-i'',\hat\tau-\tau'')\}|.
	\end{align*}

	Then the number of \(D\)-paths containing \(z\) cells of scales \(k_1,k_2,\dots,k_z\) where consecutive cells are support adjacent or support connected with diagonals with fixed relative positions is smaller than 
	\begin{equation}\label{for:phistar}
		\chi_{k_1}\prod_{j=2}^z\Phi^*_{k_{j-1},k_j}.
	\end{equation}
	
	Since a cell of a support connected \(D\)-path can either be support adjacent or support connected with diagonals to the next cell, we will consider the two cases individually.
	Consider first the case when the two cells are support adjacent, i.e. there are no diagonal steps between the extended supports of \((k,i,\tau)\) and \((k',i',\tau')\). By step 1 of this proof, we have that in this case
	\[
		\Phi^*_{k,k'}\leq\overline{\Phi_{k,k'}}.
	\]
	
	Let now the relative position of \((1,\hat i,\hat\tau)\) with respect to \((1,i'',\tau'')\) be different from \((0,0)\).
	Then, since the relative position is fixed, \(\Phi^*_{k,k'}\) can be bound by the product of the number of cells of scale 1 contained in the extended support of a cell of scale \(k\) and the number of cells of scale 1 that are contained in the extended support of a cell of scale \(k'\).
	Using the bounds from step 1, this gives that 
	\[
		\Phi^*_{k,k'}\leq\overline{\Phi_{k,1}}\cdot\overline{\Phi_{k',1}}.
	\]
	We have therefore for any fixed relative position of \((1,\hat i,\hat\tau)\) with respect to \((1,i'',\tau'')\) that 
	\begin{equation}\label{for:DpathsWithFixedDiagonalsSingle}
		\Phi^*_{k,k'}\leq\overline{\Phi_{k,k'}}\mathbbm{1}_{(\hat i,\hat\tau)=(i'',\tau'')}+\overline{\Phi_{k,1}}\cdot\overline{\Phi_{k',1}}\mathbbm{1}_{(\hat i,\hat\tau)\neq(i'',\tau'')}.
	\end{equation}
	By using the bounds from step 1 and (\ref{for:DpathsWithFixedDiagonalsSingle}), we get that
	\begin{equation}\label{for:DpathsWithFixedDiagonals}
		\chi_{k_1}\prod_{j=2}^z\Phi^*_{k_{j-1},k_j}\leq \exp\left\{\frac{3c_3}{8}\sum_{j=1}^z\psi_{k_j}\right\}.
	\end{equation}
	

	We now move on to the third step and show a bound for the number of different relative positions that are possible in a
	support connected \(D\)-path of cells of scales \(k_1,k_2,\dots,k_z\). We will show that this number is smaller than
	\begin{equation}\label{for:DiagonalStepsOnly}
		\exp\left\{\frac{c_3}{8}\sum_{j=1}^z\psi_{k_j}\right\},
	\end{equation}
	which combined with (\ref{for:DpathsWithFixedDiagonals}) proves the lemma.
	
	Consider two consecutive cells of the \(D\)-path and let \((1,i,\tau)\) be a cell contained in the extended support of the first cell that is diagonally connected to a cell \((1,i',\tau')\) that is contained in the extended support of the second cell.
	Recall from \Cref{section:statement} the definition of the base-height index and from \Cref{section:fractal} the properties of the sequence of cells that make \((1,i,\tau)\) diagonally connected to \((1,i',\tau')\).
	Denote by \(x\) the height difference between the two cells, i.e. \(x:= |h-h'|\) in the base-height index,
	and define \(A(x), x\in \mathbb{Z}\), to be the number of different cells of scale \(1\) that \((1,i,\tau)\) is diagonally connected to with height difference \(x\). More precisely,
	\begin{align*}
		A(x)&=\max_{(b_1,h_1)\in\mathbb{Z}^{d+1}}|\{(b_2,h_2)\in\mathbb{Z}^{d+1}:|h_2-h_1|=x\textrm{ and}\\
		&\qquad\qquad\qquad\textrm{ \((b_1,h_1)\) is diagonally connected to \((b_2,h_2)\)}\}|.
	\end{align*}
	
	Let \(H_k\) be the side length of the cube \(S_k^{\textrm{2sup}}(i)\) divided by the side length of the cube \(S_1(i')\), that is, let \(H_k=(3m+1)m^k((k+1)!)^a\).
	Recall from
	\Cref{section:fractal} that using the base-height index, for any two cells \((b_i,h_i),(b_j,h_j)\) of the diagonal, \(h_ih_j\geq 0\) and that
	\(h_{j-1}-h_{j}\in\operatorname{Sign}(h_{j-1})\) for any two consecutive cells of the diagonal.
	Therefore, given the \(z\) cells of scales \(k_1,k_2,\dots,k_z\), the maximum number of scale \(1\) diagonal steps contained in all diagonal connections between the cells of the path is at most
	\[
		H:=\sum_{i=1}^{z-1} H_{k_i}.
	\]
	Letting \(x_i\), for \(i\in\{1,2,\dots,z-1\}\) be the height difference between the \(i\)-th and \((i+1)\)-th cell of the path, with \(x_i=0\) if the cells are support adjacent, we have that the number of possible configurations of the diagonal steps is at most

	\begin{equation}\label{for:sum_vol}
		\sum_{y=0}^H\sum_{\substack{x_1,x_2,\dots,x_{z-1}:\\x_1+\dots+x_{z-1}=y}}A(x_1+1)A(x_2+1)\cdots A(x_{z-1}+1).
	\end{equation}
	
	See \Cref{fig:Dpaths} for an illustration of one such configuration. The \(+1\) terms in (\ref{for:sum_vol}) account for the fact that each diagonal either ends in a multi-scale bad cell or is adjacent to one, so by increasing the height difference by 1, we account for both possibilities at once.
	
	\begin{figure}[hbt]
  		\includegraphics[width=1 \linewidth]{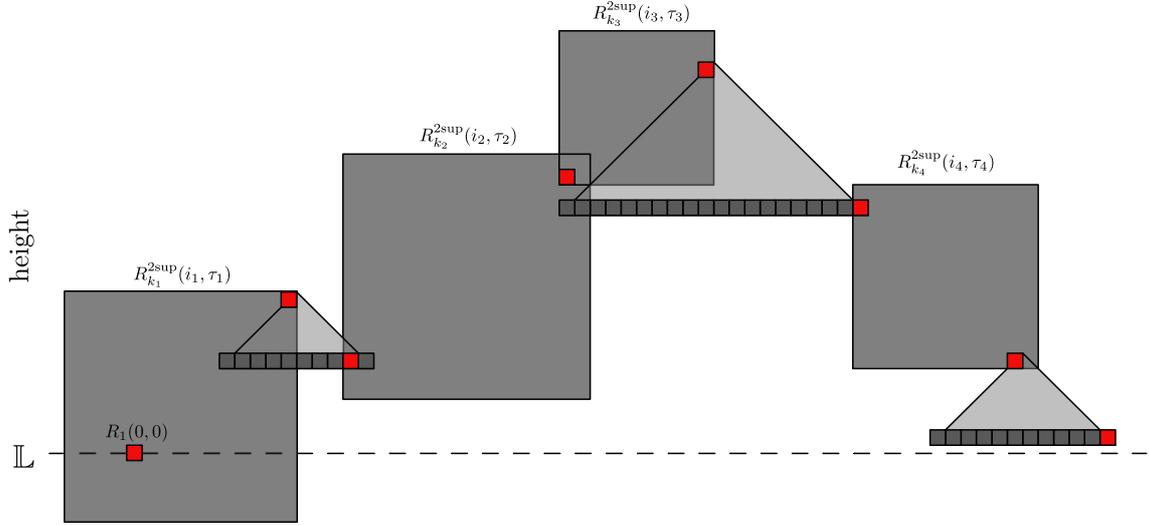}
  		\caption{Example with \(z=4\) where \(x_2=0\) and \(x_1+x_2+x_3+x_4\leq H\). The red cells are the scale \(1\) cells used as the (fixed) beginnings and ends of the diagonals within their respective extended supports. The dark cells with the matching red cell at the bottom of the triangles represent the area containing \(A(x_i)\) cells of scale 1.}
  		\label{fig:Dpaths}
	\end{figure}

	Due to the properties of the diagonal steps, we have that \(A(x)\) is the volume of a \(d\)-dimensional ball of radius \(x\), so \(A(x)\leq c_dx^d\) where \(c_d\) is some constant that depends on \(d\) only. It follows, e.g. by the method of Lagrange multipliers, that for \(x_k\geq 0, k\in \{1,\dots,{z-1}\}\) and \(x_1+\dots+x_{z-1}=y\), we have
	\[
		A(x_1+1)A(x_2+1)\cdots A(x_{z-1}+1)\leq \left[A\left(\tfrac{y}{z-1}+1\right)\right]^{z-1}.
	\]

	Next, using the above bound and 
	\[
		\sum_{\substack{x_1,x_2,\dots,x_{z-1}:\\x_1+\dots+x_{z-1}=y}}1={z+y-2\choose z-2},
	\]
	we have that the sum in (\ref{for:sum_vol}) is smaller than
	\[
		\sum_{y=0}^H{z+y-1\choose z-1}A\left(\tfrac{y}{z-1}+1\right)^{z-1}\leq{z+H\choose z}A\left(\tfrac{H}{z-1}+1\right)^{z-1},
	\]
	where the binomial inequality used can easily be proven by induction on \(H\) (using Pascal's rule).

	Then, for some positive constants \(C\) and \(C_2\) and using that \(\frac{H}{z}\) is large for large \(\alpha\), we have that
	\begin{align}
		{z+H\choose z}A\left(\tfrac{H}{z-1}+1\right)^{z-1}&\leq C\frac{(z+H)^{z}}{(z)!}\left(\frac{H}{z-1}+1\right)^{(z-1)d}\nonumber\\
		&\leq C \frac{(z+H)^{z}}{(z/3)^{z}}\left(\frac{2H}{z}\right)^{zd}\nonumber\\
		&\leq C(3+3H/z)^{z} \left(\frac{2H}{z}\right)^{zd}\nonumber\\
		&\leq C\left(C_2\frac{H}{z}\right)^{2zd}.\nonumber
	\end{align}
	In order to complete the proof, it remains to show that \(C\left(C_2\frac{H}{z}\right)^{2zd}\leq\exp\left\{\frac{c_3}{8}\sum_{j=1}^z\psi_{k_j}\right\}\), which is equivalent to showing that
	\begin{equation}\label{for:1overalpha}
		\tilde C\log\left(\frac{H}{z}\right)\leq\frac{1}{z}\sum_{j=1}^z\psi_{k_j},
	\end{equation}
	where \(\tilde C\) is some constant.
	For small \(k\), setting \(\alpha\) (and thus \(\ell\)) large enough gives that \(H_k\leq \psi_k\), and similarly setting \(m\) large gives \(H_k\leq\psi_k\) for all large \(k\).
		Combined, this gives (\ref{for:1overalpha}).
\end{proof}
\vspace{0.11cm}
\begin{remark}\label{rem:height2}
	As mentioned in \Cref{section:statement} (see \Cref{rem:height1}), one can set time to be height in the base-height index. 
	In that case all results up to and including \Cref{lemma:failure_prob_total} go through unchanged. 
	However, an important issue arises in \Cref{lemma:number_of_paths}.
	In the proof of \Cref{lemma:number_of_paths}, the height \(H_k\) of the extended support of a cell becomes the length of the interval \(T_k^{\textrm{2sup}}(\tau)\). Then, if \(d\geq 3\), the proof goes through unchanged since it still holds that \(H_k\leq\psi_k\) for all \(k\geq 1\) by setting \(m\) and \(\alpha\) large enough. For \(d=2\) however the lemma no longer holds, since it can happen that the number of different arrangements of diagonal steps and the \(z\) cells of a path is larger than \(\exp\{\sum_{i=1}^z\psi_{k_i}\}\).
	To see this, consider the following example. 
	Let \(k_1\) be large and let \(k_i=2\) for all \(i\in\{2,\dots,z\}\). 
	Let \(z\) be the largest integer for which it holds that \(\psi_{k_1}\geq 4\sum_{i=2}^z\psi_{k_i}=4(z-1)\psi_2\). Note that this gives that
	\begin{equation}\label{for:counterexample1}
		\frac{\psi_{k_1}}{4\psi_2}\leq z\leq 1+\frac{\psi_{k_1}}{4\psi_2}\leq \frac{\psi_{k_1}}{3\psi_2},
	\end{equation}
	where the last inequality holds for any large enough $k_1$.
	Furthermore note that since \(d=2\) we can write \(H_{k_1}=a_{k_1}\psi_{k_1}\), where \(a_{k_1}\) is a term that can be made arbitrarily large by increasing \(k_1\).
	Next, observe that the number of different arrangements of diagonal steps for the cells of scale 2 is at least \(H_{k_1}+z-2\choose z-2\). Therefore, we want to show that for any constant $c_1>0$, 
	we can set \(k_1\) large enough to have 
	\begin{equation}\label{for:counterexample2}
		{H_{k_1}+z-2\choose z-2}\geq \exp\left\{\sum_{i=1}^z c_1 \psi_{k_i}\right\}.
	\end{equation}
	Consider first the left hand side of (\ref{for:counterexample2}) and note that it is bigger than 
	\[
	   \frac{H_{k_1}^{z-2}}{(z-2)!}
	   \geq \left(\frac{H_{k_1}}{z-2}\right)^{z-2}
	   \geq \left(3a_{k_1}\psi_2\right)^{z-2},
	\] 
	where in the last inequality we used the upper bound on $z$ from (\ref{for:counterexample1}).
	For the right hand side of (\ref{for:counterexample2}), we have 
	\[
		\exp\left\{\sum_{i=1}^z c_1 \psi_{k_i}\right\} 
		= \exp\left\{c_1(\psi_{k_1}+(z-1)\psi_2)\right\} 
		\leq \exp\left\{4c_1 \psi_2 z + c_1\psi_2 (z-1)\right\},
	\]
	where the inequality follows from the upper bound on $\psi_{k_1}$ obtained from the leftmost inequality in (\ref{for:counterexample1}).
	Since $a_{k_1}$ grows with $k_1$, we obtain (\ref{for:counterexample2}) for large enough $k_1$. 
	\end{remark}

For any support connected \(D\)-path \(P=((k_1,i_1,\tau_1),(k_2,i_2,\tau_2),\dots,(k_z,i_z,\tau_z))\in\Omega_{\kappa-1,t}^{\textrm{sup}}\), we define the weight of \(P\) as \(\sum_{j=1}^{z}\psi_{k_j}\). The lemma below shows that, for any \(P\in \Omega_{\kappa-1,t}^{\textrm{sup}}\), if \(t\) is large enough, then the weight of \(P\) must be large.

\begin{lemma}\label{lemma:weight_of_path}
	Let \(t>0\) and let \(P=((k_1,i_1,\tau_1),(k_2,i_2,\tau_2),\dots,(k_z,i_z,\tau_z))\) be a path in \(\Omega_{\kappa-1,t}^{\textrm{sup}}\). If \(\alpha\) is sufficiently large and \(\kappa=\mathcal{O}(\log t)\), then there exist a positive constant \(c=c(C_M)\) and a value \(C\) independent of \(t\) such that
	\begin{align}
		\sum_{j=1}^z\psi_{k_j}\geq\begin{cases}
			C\frac{\sqrt{t}}{(\log t)^c},&\textrm{for }d=1,\\
			C\frac{t}{(\log t)^c},&\textrm{for }d=2,\\
			Ct,&\textrm{for }d\geq 3.
		\end{cases}
	\end{align}
\end{lemma}

\begin{proof}
	Let \(\Delta_{k}^{2\textrm{sup}}\) denote the diameter of the extended support of a cell of scale \(k\). Then, we have
	\[
		\Delta_{k}^{2\textrm{sup}}\leq (6m+3)\ell_{k+1}\sqrt{d}+27\beta_{k+1}=(6m+3)m(k+1)^a\ell_k\sqrt{d}+27C_{\textrm{mix}}\frac{\ell_k^2(k+1)^{8/\Theta}}{\epsilon^{4/\Theta}}.
	\]
	Then, the definition of \(C_{\textrm{mix}}\) gives us that there exists a constant \(c_6\) (that might depend on the ratio \(\beta/\ell^2\)) such that
	\[
		\Delta_{k}^{2\textrm{sup}}\leq(6m+3)m(k+1)^a\ell_k\sqrt{d}+27m^2\frac{\beta}{\ell^2}\ell_k^2(k+1)^{8/\Theta}\leq c_6 m^2(k+1)^{2a}\ell_k^2.
	\]
	Then, for \(k\geq 2\), we have for \(d=1\) that
	\begin{align*}
		\psi_k&=\frac{\epsilon^2\lambda_0\ell_{k-1}}{(k+1)^4}=\frac{\epsilon^2\lambda_0}{(k+1)^4}\left(\frac{\ell_k}{m(k^a)}\right)\\
			&\geq\frac{\epsilon^2\lambda_0}{m(k+1)^{a+4}}\left(\frac{\sqrt{c_6m^2(k+1)^{2a}\ell_k^2}}{\sqrt{c_6m^2(k+1)^{2a}}}\right)\\
			&\geq\frac{\epsilon^2\lambda_0}{\sqrt{c_6}m^2(k+1)^{3a+4}}\sqrt{\Delta_{k}^{2\textrm{sup}}}
	\end{align*}
	and for \(d\geq 2\) that
	\begin{align*}
		\psi_k&=
			\frac{\epsilon^2\lambda_0\ell_{k-1}^{d-2}}{(k+1)^4}\left(\frac{\ell_k}{mk^a}\right)^2\\
			&\geq\frac{\epsilon^2\lambda_0\ell_{k-1}^{d-2}}{m^2(k+1)^{2a+4}}\left(\frac{c_6m^2(k+1)^{2a}\ell_k^2}{c_6m^2(k+1)^{2a}}\right)\\
			&\geq\frac{\epsilon^2\lambda_0\ell_{k-1}^{d-2}}{c_6m^4(k+1)^{4a+4}}\Delta_{k}^{2\textrm{sup}}.
	\end{align*}
	Now, since \(\kappa=\mathcal{O}(\log t)\), there exists a constant \(c_7\) such that \((k+1)^b\leq c_7(\log t)^b\) for all \(k\leq \kappa\) and any \(b\geq 1\). We use this for dimensions 1 and 2. For dimension 3 and higher, we set \(c_7\) large enough to satisfy \(\frac{\ell_{k-1}^{d-2}}{(k+1)^{4a+4}}\geq\frac{\ell^{d-2}}{mc_7}\); this is possible since \(\ell_k\) is of order \((k!)^a\). This gives 
	\[
		\psi_k\geq\begin{cases}
			\frac{\epsilon^2\lambda_0}{\sqrt{c_6}c_7m^2}\frac{\sqrt{\Delta_{k}^{2\textrm{sup}}}}{(\log t)^{3a+4}},&\textrm{for }d=1\\
			\frac{\epsilon^2\lambda_0}{c_6c_7m^4}\frac{\Delta_{k}^{2\textrm{sup}}}{(\log t)^{4a+4}},&\textrm{for }d=2\\
			\frac{\epsilon^2\lambda_0\ell^{d-2}}{c_6c_7m^5}\Delta_{k}^{2\textrm{sup}},&\textrm{for }d\geq 3.
		\end{cases}
	\]
	For \(k=1\) we write \(\psi_1\geq c\sqrt{\Delta_{k}^{2\textrm{sup}}}\) for \(d=1\) and \(\psi_1\geq c\Delta_{k}^{2\textrm{sup}}\) for \(d\geq 2\), where \(c\) is some positive value that may depend on \(\epsilon\), \(m\), \(\lambda_0\), \(\ell\) and \(\nu_E\). Moreover, if a support connected \(D\)-path is such that \(\sum_{j=1}^z\Delta_{k_j}^{2\textrm{sup}}<t/2\), the extended support of all cells of the path must be contained in \([-t,t]^{d+1}\). This is true because if there are no diagonal steps in \(P\), then the extended supports are contained in \([-t/2,t/2]^{d+1}\), and if there are diagonal steps, they can only prolong the path by at most \(\sum_{j=1}^z\Delta_{k_j}^{\textrm{2sup}}\).
	Therefore, for \(P\in\Omega_{\kappa-1,t}^{\textrm{sup}}\) we have \(\sum_{j=1}^z\Delta_{k_j}^{2\textrm{sup}}\geq t/2\). This implies that there exists a positive \(C\) independent of \(t\), but depending on everything else such that
	\[
		\sum_{j=1}^z\psi_{k_j}\geq\begin{cases}
			 C\frac{\sqrt{t}}{(\log t)^{3a+4}},&\textrm{for }d=1\\
			C\frac{t}{(\log t)^{4a+4}},&\textrm{for }d=2\\
			Ct,&\textrm{for }d\geq 3.
		\end{cases}
	\]
\end{proof}

We now write \(\psi_k\), \(k\geq 2\) as a multiple of \(\psi_2\). This will be used to count the number of paths in \(\Omega_{\kappa,t}^{\textrm{sup}}\) later. For this,
set \(\tilde\psi_2=\psi_2=3^{-4}\epsilon^2\lambda_0\ell^d\), and for \(j\geq 3\), define
\[
	\tilde\psi_j=2\tilde\psi_2m^{(j-2)d}((j-1)!)^{ad-3}((j-2)!)^2(j-3)!.
\]

\begin{lemma}\label{lemma:psi_tilde}
For all \(j\geq 2\), it holds that \(\tilde \psi_j\leq \psi_j\leq 41\tilde\psi_j\).
\end{lemma}
\begin{proof}
	For \(j\geq 3\) we write
	\begin{align*}
		\psi_j&=\frac{\epsilon^2\lambda_0\ell^dm^{(j-2)d}((j-1)!)^{ad}}{(j+1)^4}=3^4\tilde\psi_2\frac{m^{(j-2)d}((j-1)!)^{ad}}{(j+1)^4}\\
		&=3^4\tilde\psi_2m^{(j-2)d}((j-1)!)^{ad-3}((j-2)!)^2(j-3)!\left(\frac{(j-1)^3(j-2)}{(j+1)^4}\right).
	\end{align*}
	This implies that \(\psi_j\leq \frac{3^4}{2}\tilde\psi_j\leq 41\tilde\psi_j\). The other direction follows from the fact that \(\frac{(j-1)^3(j-2)}{(j+1)^4}\geq 1/32\) for all \(j\geq 3\).
\end{proof}

\section{Size of bad clusters}\label{section:proof_of_prop}

For \(k\geq 1\), define \(\mathcal{S}^t_k\) to be the set of indices \(i\in\mathbb{Z}^d\) given by
\[
	\mathcal{S}^t_k=\left\{i\in\mathbb{Z}^d:\;S_k(i)\textrm{ intersects }[-t,t]^d\right\}.
\]
Similarly, we define \(\mathcal{T}^t_k\) as the set of indices \(\tau\) for time intervals that have a descendent at scale 1 intersecting \([0,t]\). Formally, let
\[
	\mathcal{T}^t_k=\left\{\tau\in\mathbb{Z}:\exists \tau'\textrm{ s.t. }\gamma_1^{(k-1)}(\tau')=\tau \textrm{ and }T_1(\tau')\cap[0,t]\neq\emptyset\right\}.
\]
Note that an interval in \(\mathcal{T}^t_k\) with \(k\geq 2\) may not intersect \([0,t]\). Using these definitions define
\[
	\mathcal{R}^t_k=\mathcal{S}^t_k\times\mathcal{T}^t_k.
\]
For the following proposition, recall from \Cref{section:DpathsAndBadClusters} the definitions of \(K(0,0)\) and \(K'(0,0)\).

\begin{prop}\label{prop:exp_tail}
	For each \((i,\tau)\in\mathbb{Z}^{d+1}\), let \(E_{\textrm{st}}(i,\tau)\) be an increasing event that is restricted to the super cube \(i\) and the super interval \(\tau\), and let \(\nu_{E_{\textrm{st}}}\) be the probability associated to \(E_{\textrm{st}}\) as defined in \Cref{def:probassoc}. Fix a constant \(\epsilon\in(0,1)\), and integer \(\eta\geq1\) and the ratio \(\beta/\ell^2>0\). Fix also \(w\) such that
	\[
		w\geq\sqrt{\frac{\eta\beta}{c_2\ell^2}\log\left(\frac{8c_1}{\epsilon}\right)},
	\]
	for some constants \(c_1\) and \(c_2\) which depend on the graph. Then, there exist constants \(c\) and \(C\), and positive numbers \(\alpha_0\) and \(t_0\) that depend on \(\epsilon\), \(\eta\), \(w\) and the ratio \(\beta/\ell^2\) such that if
	\[
	\alpha=\min\left\{C_{M}^{-1}\epsilon^2\lambda_0\ell^d,\log\left(\frac{1}{1-\nu_{E_{\textrm{st}}}((1-\epsilon)\lambda,Q_{(2\eta+1)\ell},Q_{w\ell},\eta\beta)}\right)\right\}\geq\alpha_0,
	\]
	we have for all \(t\geq t_0\) that
	\[
	\mathbb{P}\left[K(0,0)\not\subseteq\mathcal{R}^t_1\right]\leq\left\{
	\begin{array}{ll}
		\exp\left\{-C\lambda_0\frac{t}{(\log t)^c}\right\}&\textrm{for }d=2\\
		\exp\left\{-C\lambda_0 t\right\}&\textrm{for }d\geq 3.
	\end{array}
	\right.
	\]
\end{prop}

\begin{proof}
First, for any \(k\), note that the number of cells in \(\mathcal{R}^t_k\) satisfies
\begin{equation}\label{for:Rk_upper_bound}
	|\mathcal{R}^t_k|\leq\left(2\left\lceil\frac{t}{\ell_k}\right\rceil\right)^d\left\lceil 1+\frac{t}{\beta_k}\right\rceil.
\end{equation}
Also, using Lemmas \ref{lemma:EgeqA} and \ref{lemma:support_connected_paths}, we have
\begin{align*}
	\mathbb{P}[K(i,\tau)\not\subseteq\mathcal{R}^t_1]\leq\mathbb{P}[K'(i,\tau)\not\subseteq\mathcal{R}^t_1]&=\mathbb{P}[\exists P\in\Omega_t\textrm{ s.t. all cells of }P\textrm{ have bad ancestry}]\\
	&\leq \mathbb{P}[\exists P\in\Omega_{\kappa,t}^{\textrm{sup}}\textrm{ s.t. all cells of \(P\) are multi-scale bad}].
\end{align*}

We note that the random variable \(A_{\kappa}\) is defined differently than other scales. It follows from \Cref{lemma:failure_prob}, (\ref{for:Rk_upper_bound}) and the union bound over all cells in \(\mathcal{R}^t_{\kappa}\) that
\begin{equation}
	\mathbb{P}[(A_{\kappa}(i',\tau')=1\textrm{ for all }(i,\tau)\in\mathcal{R}^t_{\kappa}]\geq 1-|\mathcal{R}^t_{\kappa}|\exp\{-c\psi_{\kappa}\}\geq1-\exp\{-c_1t\},
\end{equation}
for some positive constant \(c_1\), where the last step follows by setting \(\kappa\) to be the smallest integer such that \(\psi_{\kappa}\geq t\), which using the Lambert W function and its asymptotics gives that \(\kappa=\Theta\left(\frac{\log t}{\log\log t}\right)\). Let us define \(H\) as the event that \(A_{\kappa}(i,\tau)=1\) for all \((i,\tau)\in\mathcal{R}^t_{\kappa}\). Then, we have
\begin{align*}
	\mathbb{P}&\left[\exists P\in\Omega_{\kappa,t}^{\textrm{sup}}\textrm{ s.t. all cells of }P\textrm{ are multi-scale bad}\right]\\
	&\leq\mathbb{P}\left[H\cap\{\exists P\in\Omega_{\kappa,t}^{\textrm{sup}}\textrm{ s.t. all cells of }P\textrm{ are multi-scale bad}\}\right]+\mathbb{P}\left[H^c\right]\\
	&\leq\mathbb{P}\left[\exists P\in\Omega_{\kappa-1,t}^{\textrm{sup}}\textrm{ s.t. all cells of }P\textrm{ are multi-scale bad}\right]+e^{-c_1t}.
\end{align*}

To get a bound for the term above, we fix a support connected \(D\)-path
\[
	P=((k_1,i_1,\tau_1),\dots,(k_z,i_z,\tau_z)),
\]
and use \Cref{lemma:failure_prob_total} to get
\[
	\mathbb{P}\left[\textstyle\bigcap_{j=1}^z\{A_{k_j}(i_j,\tau_j)=0\}\right]\leq\exp\left\{-c_3\sum_{j=1}^z\psi_{k_j}\right\}.
\]
We now take the union bound over all support connected \(D\)-paths with cells of scales \(k_1,k_2,\dots,k_z\) and using \Cref{lemma:number_of_paths}, we get that
\begin{align*}
	\mathbb{P}\left[\exists P\in\Omega_{\kappa-1,t}^{\textrm{sup}}\textrm{ s.t. }P\textrm{ has }z\textrm{ multi-scale bad cells of scales }k_1,k_2,\dots,k_z\right]\\\leq\exp\left\{-\frac{c_3}{2}\sum_{j=1}^z\psi_{k_j}\right\}.
\end{align*}

This bound depends on \(z\) and \(k_1,\dots,k_z\) only through \(\sum_{j=1}^z\psi_{k_j}\), which we call the weight of the path. Let \(W\) be the set of weights for which there exists at least one path in \(\Omega_{\kappa-1,t}^{\textrm{sup}}\) with such a weight. Then
\begin{equation}\label{for:M(w)}
	\mathbb{P}\left[\exists P\in\Omega_{\kappa-1}^{\textrm{sup}}\textrm{ s.t. all cells of }P\textrm{ are multi-scale bad}\right]\leq\sum_{w\in W}\exp\left\{-\frac{c_3}{2}w\right\}M(w),
\end{equation}
where \(M(w)\) is the number of possible ways to choose \(z\) and \(k_1,k_2,\dots,k_z\) such that \(\sum_{j=1}^z\psi_{k_j}=w\).

Let \(w=\sum_{j=1}^z\psi_{k_j}\) and let \(w_1=\psi_1|\{j:k_j=1\}|\). Let \(w_2=w-w_1\), so \(w_1\) is the weight given by cells of scale 1 and \(w_2\) the weight given by the other cells of the path. Note that by \Cref{lemma:psi_tilde}, \(w_2=\sum_{j:k_j\geq2}\psi_{k_j}\geq \sum_{j:k_j\geq2}\tilde\psi_{k_j}=h_2\psi_2\) for some non-negative integer \(h_2\). Likewise, \(w_2\leq 41h_2\psi_2\) and \(w_1=h_1\psi_1\) for some non-negative integer \(h_1\). Let \(w_0\) be the lower bound on the weight of the path given by \Cref{lemma:weight_of_path}, so for all \(w\in W\), we have \(w\geq w_0\). Since either \(w_1\) or \(w_2\) has to be larger than \(w_0/2\), we have that either \(h_1\geq\frac{w_o}{2\psi_1}\) or \(h_2\geq\frac{w_0}{2\cdot41\psi_2}\). Let \(M(h_1,h_2)\) be the number of ways to choose \(z\) and \(k_1,\dots,k_z\) such that there are \(h_1\) values \(j\) with \(k_j=1\) and \(\sum_{j:k_j\geq2}\tilde\psi_{k_j}=h_2\psi_2\). For any such choice, we have \(w=\sum_{j=1}^z\psi_{k_j}\geq h_1\psi_1+h_2\psi_2\). Then, the sum in the right-hand side of (\ref{for:M(w)}) can be bounded above by
\begin{align*}
	\sum_{h_1\geq\frac{w_0}{2\psi_1}}&\sum_{h_2=0}^\infty\exp\left\{-\frac{c_3}{2}(h_1\psi_1+h_2\psi_2)\right\}M(h_1,h_2)\\
	&+\sum_{h_1=0}^\infty\sum_{h_2\geq\frac{w_0}{82\psi_2}}\exp\left\{-\frac{c_3}{2}(h_1\psi_1+h_2\psi_2)\right\}M(h_1,h_2).
\end{align*}

We now proceed to bound \(M(h_1,h_2)\). Suppose we have \(h_1\) blocks of size \(\psi_1\) and \(h_2\) blocks of size \(\psi_2\). Consider an ordering of the blocks, such that permuting the blocks of the same size does not change the order. Then, for each block of size \(\psi_2\), we color it either black or white, while blocks of size \(\psi_1\) are not colored. For each choice of \(z\) and \(k_1,\dots,k_z\), we associate an order and coloring of the blocks as follows. if \(k_1=1\), then the first block is of size \(\psi_1\). Otherwise, the first \(\tilde\psi_{k_1}/\psi_2\) blocks are of size \(\psi_2\) and have black color. Then, if \(k_2=1\), the next block is of size \(\psi_1\), otherwise the next \(\tilde\psi_{k_2}/\psi_2\) blocks are of size \(\psi_2\) and have white color. We proceed in this way until \(k_z\), where whenever \(k_i\neq 1\) we use the color black if \(i\) is odd and the color white if \(i\) is even. Though there are orders and colorings that are not associated to any choice of \(z\) and \(k_1,\dots,k_z\), each such choice of \(z\) and \(k_1,\dots,k_z\) corresponds to a unique order and coloring of the blocks. Therefore, the number of ways to order and color the blocks gives an upper bound for \(M(h_1,h_2)\). Note that there are \(h_1+h_2\choose h_1\) ways to order the blocks and \(2^{h_2}\) ways to color the size-\(\psi_2\) blocks. Therefore
\begin{align*}
	\mathbb{P}&\left[\exists P\in\Omega_{\kappa-1,t}^{\textrm{sup}}\textrm{ s.t. all cells of }P\textrm{ are multi-scale bad}\right]\\
	&\leq \sum_{h_1\geq\frac{w_0}{2\psi_1}}\sum_{h_2=0}^\infty\exp\left\{-\frac{c_3}{2}(h_1\psi_1+h_2\psi_2)\right\}{h_1+h_2\choose h_1}2^{h_2}\\
	&\quad+\sum_{h_1=0}^\infty\sum_{h_2\geq\frac{w_0}{82\psi_2}}\exp\left\{-\frac{c_3}{2}(h_1\psi_1+h_2\psi_2)\right\}{h_1+h_2\choose h_1}2^{h_2}\\
	&\leq C\sum_{h1\geq\frac{w_0}{2\psi_1}}\sum_{h_2=0}^\infty\exp\left\{-\frac{c_3}{3}(h_1\psi_1+h_2\psi_2)\right\} + C\sum_{h_1=0}^\infty\sum_{h_2\geq\frac{w_0}{82\psi_2}}\exp\left\{-\frac{c_3}{3}(h_1\psi_1+h_2\psi_2)\right\}\\
	&\leq\exp\{-c w_0\},
\end{align*}
for some constants \(C\) and \(c\), where in the second inequality we use \Cref{lemma:binomial} and the fact that \(\alpha\) is sufficiently large to write \(\frac{c_3\psi_1}{2}-1\geq\frac{c_3\psi_1}{3}\), and similarly for \(\psi_2\). Since we defined \(w_0\) to be the lower bound on the weight of a path given by \Cref{lemma:weight_of_path}, the proof is complete.
\end{proof}

\section{Proof of \Cref{thrm:surface_event_simple}}\label{section:proof_of_surface}

\begin{proof}[Proof of \Cref{thrm:surface_event_simple}]
	By \Cref{thrm:surface}, it suffices to show that \[\sum_{r\geq 1}r^d\mathbb{P}\left[\operatorname{rad}_0(H_0)>r\right]<\infty.\] We begin by noting that after tessellating space and time, \(\mathcal{R}^t_1\) contains cells indexed only by \((i,\tau)\) for which \(\|i\|_{\infty}\leq\frac{t}{\ell}\) and \(|\tau|\leq\frac{t}{c\ell^2}\) for some positive constant \(c\). For fixed \(R>0\), if we set \(T>0\) such that
	\[
		\left(\frac{d}{\ell}+\frac{1}{c\ell^2}\right)T\leq R,
	\]
	then \(\mathcal{R}_1^T\) is contained in \(\{u\in\mathbb{Z}^{d+1}:\;\|u\|_1<R\}\). Let \(T(r)=\left(\frac{d}{\ell}+\frac{1}{c\ell^2}\right)^{-1}r\) and fix \(r_0\) such that \(T(r_0)>t_0\), where \(t_0\) comes from \Cref{prop:exp_tail}. Then we have that
	\begin{align*}
		\sum_{r\geq r_0}r^d\mathbb{P}[\operatorname{rad}_0(H_0)>r]&\leq\sum_{r\geq r_0}r^d\mathbb{P}\left[H_0\not\subseteq\mathcal{R}_1^{T(r)}\right] \\
		&\leq\sum_{r\geq r_0}r^d\mathbb{P}\left[K(0,0)\not\subseteq\mathcal{R}_1^{T(r)}\right],
	\end{align*}
	where we used in the second inequality that every \(d\)-path on the space-time tessellation is also a \(D\)-path of bad cells. We now apply \Cref{prop:exp_tail} with \(d\geq 3\) to bound \(\mathbb{P}\left[K(0,0)\not\subseteq\mathcal{R}_1^{T(r)}\right]\) for \(T(r)>t_0\) and get that
	\[
		\sum_{r\geq r_0}r^d\mathbb{P}[\operatorname{rad}_0(H_0)>r]\leq\sum_{r>r_0}r^d\exp\{-C\lambda_0 T(r)\}= \sum_{r\geq r_0}r^d\exp\left\{-C\left(\frac{d}{\ell}+\frac{1}{c\ell^2}\right)^{-1}\lambda_0 r\right\},
	\]
	for some positive constant \(C\), that does not depend on \(r\). Since this expression is finite, we have by \Cref{thrm:surface} that the Lipschitz surface exists and is a.s. finite.

	For \(d=2\) we similarly get that
	\[
		\sum_{r\geq r_0}r^d\mathbb{P}[\operatorname{rad}_0(H_0)>r]\leq\sum_{r\geq r_0}r^d\exp\left\{-C\lambda_0\frac{\ell r}{(\log\ell r)^{-c}}\right\}<\infty.
	\]
\end{proof}

The corollary below gives the probability that a base-height cell \((b,0)\in\mathbb{L}\) is not part of \(F\), i.e. \(F_+(b)\neq 0\) and \(F_-(b)\neq 0\), where \(F_+\) and \(F_-\) are the two Lipschitz functions as defined in \Cref{def:F+andF-}.

\begin{corol}\label{corol:event_surface}
	Assume the setting of \Cref{thrm:surface_event_simple}. There are positive constants \(C\), \(c\), \(C_3\) and \(r_0\) such that for any given \(b\in\mathbb{Z}^d\), we have
	\[
		\mathbb{P}[F_+(b)\cdot F_-(b)\neq0]<\left\{\begin{array}{ll}
			C r_0^d\mathbb{P}[E_{\textrm{st}}(0,0)^c]+\sum_{r\geq r_0}r^d\exp\{-C_3\lambda_0\frac{\ell r}{(\log\ell r)^c}\},&\textrm{for }d= 2\\
			C r_0^d\mathbb{P}[E_{\textrm{st}}(0,0)^c]+\sum_{r\geq r_0}r^d\exp\left\{-C_3\lambda_0	\ell r\right\},&\textrm{for }d\geq 3.
		\end{array}\right.
	\]
\end{corol}

\begin{proof}
	Recall first that by construction, \(F_+(b)=0\) if and only if \(F_-(b)=0\). Then,
	we have for a positive constant \(C\) that depends only on \(d\) that
	\begin{align*}
		\mathbb{P}[F_+(b)\neq 0]&\leq\sum_{(x,0)\in L}\mathbb{P}[(x,0)\rightarrowtail_{d}(b,0)]\\
		&=\sum_{\substack{(x,0)\in L\\\|x-b\|_1\leq r_0}}\mathbb{P}[(x,0)\rightarrowtail_{d}(b,0)]+\sum_{\substack{(x,0)\in L\\\|x-b\|_1> r_0}}\mathbb{P}[(x,0)\rightarrowtail_{d}(b,0)]\\
		&\leq \sum_{\substack{(x,0)\in L\\\|x-b\|_1\leq r_0}}\mathbb{P}[E_{\textrm{st}}^c(x,0)]+\sum_{\substack{(x,0)\in L\\\|x\|_1> r_0}}\mathbb{P}[(0,0)\rightarrowtail_{d}(x,0)]\\
		&\leq C r_0^d\mathbb{P}\left[E_{\textrm{st}}^c(0,0)\right]+\sum_{r>r_0}Cr^d\mathbb{P}[\operatorname{rad}_0(H_0)>r].
	\end{align*}
	The sum above can be bounded as in the proof of \Cref{thrm:surface_event_simple}.
\end{proof}
\vspace{0.11cm}
\begin{remark}
	We note that the sum in \Cref{corol:event_surface} is decreasing with \(\ell\) and can in fact be made arbitrarily small by making \(\ell\) large enough. This gives us that if the probability of the event \(E_{\textrm{st}}(i,\tau)\) is increasing in \(\ell\), the expression in \Cref{corol:event_surface} can also be made arbitrarily small.
\end{remark}

\section{Proof of \Cref{thrm:shell_tail}}\label{section:dobule_diagonal}
Recall from \Cref{section:paths} that a hill \(H_u\) is defined as all sites in \(\mathbb{Z}^{d+1}\) that can be reached by a \(d\)-path started from \(u\in\mathbb{L}\). Recall also the definition of a mountain \(M_u\) as a union of all hills that contain \(u\). By the construction of \(D\)-paths, every \(d\)-path on the space-time tessellation is also a \(D\)-path of bad cells. For this reason, as in \Cref{section:proof_of_surface}, we will use an extension \(D\)-paths when bounding probabilities of the existence of various hills and mountains in this section.

We begin by considering a broader range of diagonally connected paths. Intuitively, these are paths that can move within sequences of hills \(H_u\) for different \(u\in\mathbb{L}\).
Let \(u=(b,0)\in\mathbb{L}\) be a cell of the zero-height plane. By \Cref{def:F+andF-}, we know a mountain touches the Lipschitz surface at \((b,F_+(u))\) and \((b,F_-(u))\), but we cannot say anything more than that. If we want to say something about the positive and negative depth of the surface \(F\) across a larger area, we therefore need to consider a large number of different mountains. Since these mountains likely intersect and are composed of some of the same hills, we need a better way to control their dependences. 
To that end, we will consider paths with diagonals that can be thought of as concatenations of different \(D\)-paths, where some \(D\)-paths may be taken in reverse order. In order to define these, which we will refer to as \(DD\)-paths, we will need to define the concept of a \emph{double diagonal}, as well as slightly change the definition of two cells being diagonally connected.

As before, we say that distinct scale 1 cells \((i,\tau)\) and \((i',\tau')\) are \emph{adjacent} if \(\|i-i'\|_{\infty}\leq 1\) and \(|\tau-\tau'|\leq 1\). 
Also, we say that \((i,\tau)\) is \emph{diagonally connected} to \((i',\tau')\) if there exists a sequence of cells \((i,\tau)=(b_0,h_0),(b_1,h_1),\dots,(b_n,h_n)=(\hat i,\hat\tau)\), where the indices \((b_j,h_j)\) refer to the base-height index, such that all the following hold:
\begin{itemize}
	\item for all \(j\in\{1,\dots,n\}\), \(\|b_j-b_{j-1}\|_1=1\) and \(h_{j-1}-h_{j}\in\operatorname{Sign}(h_{j-1})\),
	\item \(h_ih_j\geq0\) for all \(i,j\in\{0,\dots,n\}\), 
	\item \((\hat i,\hat\tau)\) is adjacent to \((i',\tau')\) or \((\hat i,\hat\tau)=(i',\tau')\).
\end{itemize}
Moreover, if \((\hat i,\hat\tau)=(i',\tau')\) we say that \((i,\tau)\) and \((i',\tau')\) are \emph{diagonally linked}.
We say for two distinct cells \((i,\tau)\) and \((i',\tau')\) are \emph{single diagonally connected} if \((i,\tau)\) is diagonally connected to \((i',\tau')\) or if \((i',\tau')\) is diagonally connected to \((i,\tau)\).
Finally, we say two distinct cells \((i,\tau)\) and \((i',\tau')\) are \emph{double diagonally connected}, if there exists \((\hat i,\hat\tau)\) such that \((i,\tau)\) is diagonally connected to \((\hat i,\hat\tau)\), \((i',\tau')\) is diagonally connected to \((\hat i,\hat\tau)\), and \((\hat i,\hat\tau)\) is diagonally linked to \((i,\tau)\) or \((i',\tau')\).

Note that unlike the definition from \Cref{section:fractal} of a cell \((i,\tau)\) being diagonally connected to \((i',\tau')\), two cells being single or double diagonally connected is a symmetric relationship.

\begin{mydef}\label{def:DDpath}
We say a sequence of cells \((i_0,\tau_0),(i_1,\tau_1),\dots,(i_n,\tau_n)\) is a \(DD\)-path if for all \(j\in\{1,\dots,n\}\), we have that the cells \((i_{j-1},\tau_{j-1})\) and \((i_{j},\tau_j)\) are adjacent, single diagonally connected or double-diagonally connected.
\end{mydef}

Recall from \Cref{section:multiscale} the definition of cells of multiple scales. Now we will extend the definition of \(DD\)-paths to multiple scales, as we did in \Cref{section:support_connected_paths} for \(D\)-paths.
We say \((k,i,\tau)\) and \((k',i',\tau')\) are \emph{single diagonally connected} if there exists a cell \((1,\hat i,\hat \tau)\) that is a descendant of \((k,i,\tau)\) and a cell \((1,i'',\tau'')\) that is a descendant of \((k',i',\tau')\), such that \((1,\hat i,\hat \tau)\) and \((1,i'',\tau'')\) are single diagonally connected.
We say \((k,i,\tau)\) and \((k',i',\tau')\) are \emph{double diagonally connected} if there exists a cell \((1,\hat i,\hat \tau)\) that is a descendant of \((k,i,\tau)\) and a cell \((1,i'',\tau'')\) that is a descendant of \((k',i',\tau')\), such that \((1,\hat i,\hat \tau)\) and \((1,i'',\tau'')\) are double diagonally connected.

We refer to a \(DD\)-\emph{path} as a sequence of distinct cells of possibly different scales for which any two consecutive cells in the sequence are either adjacent, single diagonally connected or double diagonally connected to the second.

We say two cells \((k_1,i_1,\tau_1)\) and \((k_2,i_2,\tau_2)\) are \emph{support connected with single diagonals} if there exists a scale 1 cell contained in \(R_{k_1}^{\textrm{2sup}}(i_1,\tau_1)\) and a scale 1 cell contained in \(R_{k_2}^{\textrm{2sup}}(i_2,\tau_2)\) such that the two cells are single diagonally connected. We say two cells \((k_1,i_1,\tau_1)\) and \((k_2,i_2,\tau_2)\) are \emph{support connected with double diagonals} if there exists a scale 1 cell contained in \(R_{k_1}^{\textrm{2sup}}(i_1,\tau_1)\) and a scale 1 cell contained in \(R_{k_2}^{\textrm{2sup}}(i_2,\tau_2)\), such that the two are double diagonally connected.

Recall from \Cref{section:support_connected_paths} the definitions of two cells being well separated and support adjacent.
Finally, we define a sequence of cells \(P=((k_1,i_1,\tau_1),(k_2,i_2,\tau_2),\dots,(k_z,i_z,\tau_z))\) to be a \emph{support connected \(DD\)-path} if the cells in \(P\) are mutually well separated and, for each \(j=1,2,\dots,z-1\), \((k_j,i_k,\tau_j)\) and \((k_{j+1},i_{j+1},\tau_{j+1})\) are support adjacent, support connected with single diagonals or support connected with double diagonals.

\subsection{Multi-scale analysis of \(DD\)-paths}
We now follow the steps of \Cref{section:support_connected_paths}, presenting only the parts where the statements and proofs with \(DD\)-paths differ from how they were for \(D\)-paths.

Define \(\Omega_t\) to be the set of all \(DD\)-paths of cells of scale \(1\) such that the first cell of the path is \((0,0)\) or \((0,0)\) is single diagonally connected to the first cell, and the last cell of the path is the only cell not contained in \([-t,t]^d\times [-t,t]\). Also, define \(\Omega_{\kappa,t}^{\textrm{sup}}\) as the set of all support connected \(DD\)-paths of cells of scale at most \(\kappa\) so that the extended support of the first cell of the path contains \(R_1(0,0)\) or \((0,0)\) is single diagonally connected to a scale 1 cell that is contained in the extended support of the first cell of the path, and the last cell of the path is the only cell whose extended support is not contained in \([-t,t]^d\times [-t,t]\). Then the lemma below states that we can focus on support connected \(DD\)-paths instead of \(DD\)-paths with bad ancestry; the proof is identical to the one of \Cref{lemma:support_connected_paths}.

\begin{figure}[hbt]
  		\includegraphics[width=1 \linewidth]{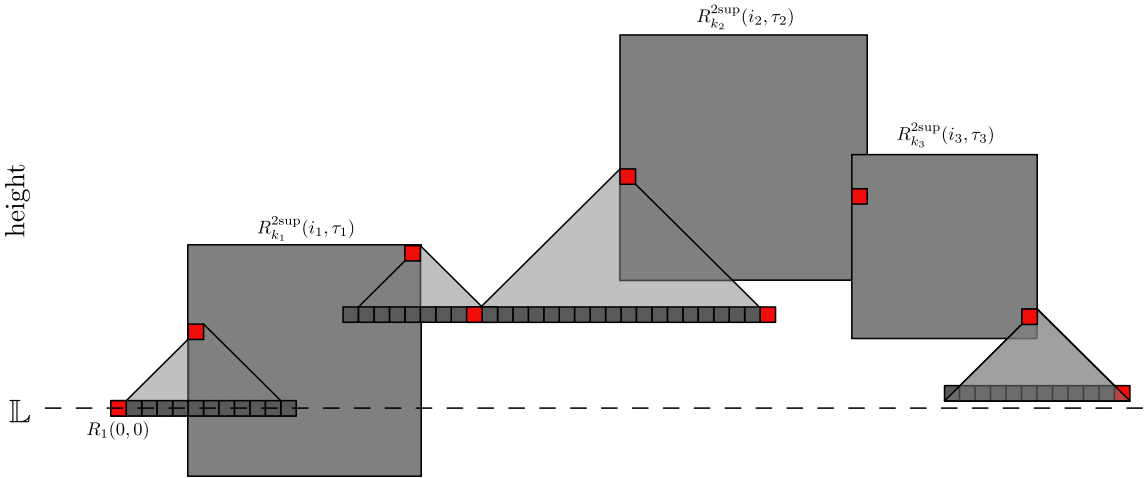}
  		\caption{Example of a \(DD\)-path where the first cell of the path is diagonally connected to \((0,0)\), is double diagonally connected to the second cell, and the second and third cells are adjacent.}
  		\label{fig:DDpaths}
\end{figure}

\begin{lemma}\label{lemma:support_connected_DDpaths}
	We have that
	\begin{align*}
		&\mathbb{P}\left[\exists P\in\Omega_t\textrm{ s.t. all cells of } P\textrm{ have a bad ancestry}\right]\\
		&\leq\mathbb{P}\left[\exists P\in\Omega_{\kappa,t}^{sup}\textrm{ s.t. all cells of }P\textrm{ are multi-scale bad}\right].
	\end{align*}
\end{lemma}

We next have to show that the bound from \Cref{lemma:number_of_paths} holds for \(DD\)-paths as well.

\begin{lemma}\label{lemma:number_of_DDpaths}
	Let \(z\) be a positive integer and \(k_1,k_2,\dots,k_z\geq1\) be fixed. Then, if \(\alpha\) is sufficiently large, the total number of support connected \(DD\)-paths, containing \(z\) cells of scales \(k_1,k_2,\dots,k_z\) is at most \(\exp\left(\frac{c_3}{2}\sum_{j=1}^z\psi_{k_j}\right)\), where \(c_3\) is the same constant as in \Cref{lemma:failure_prob_total} and \(\psi\) is as defined in (\ref{eq:def_psi}).
\end{lemma}

\begin{proof}
	The proof follows the same steps as the proof of \Cref{lemma:number_of_paths}. The only changes are that the first cell of a \(DD\)-path need not contain \((0,0)\) and the number of different relative positions in step 3 of the proof.

	For the former, we note that the extended support of the first cell of the support connected \(DD\)-path still has to contain \((0,0)\) or \((0,0)\) has to be single diagonally connected to a scale 1 cell in the extended support of the first cell. If we define \(\chi_{k_1}\) as in \Cref{lemma:number_of_paths}, then the first case is already counted by \(\chi_{k_1}\). Otherwise, note that if we fix the relative position of the first and final cell of the single diagonal connecting \((0,0)\) to the extended support of the first cell, we only need to control the number of such relative positions, which is done in step 3. Therefore, it only remains to prove step 3 of the proof for \(DD\)-paths.
	
	Consider two consecutive cells of the \(DD\)-path that are single diagonally connected and let \((1,i,\tau)\) be a cell contained in the extended support of the first cell that is single diagonally connected to a cell \((1,i',\tau')\) that is contained in the extended support of the second cell.
	Then, as in the proof of \Cref{lemma:number_of_paths} we can define
	\begin{align*}
		A(x)&=\max_{(b_1,h_1)\in\mathbb{Z}^{d+1}}|\{(b_2,h_2)\in\mathbb{Z}^{d+1}:|h_2-h_1|=x\textrm{ and}\\
		&\qquad\qquad\qquad\textrm{ \((b_1,h_1)\) is diagonally connected to \((b_2,h_2)\)}\}|.
	\end{align*}
	
	Consider now two consecutive cells of the \(DD\)-path that are double diagonally connected and let \((1,i,\tau)\) be a cell contained in the extended support of the first cell that is double diagonally connected to a cell \((1,i',\tau')\) that is contained in the extended support of the second cell. Furthermore, let \((1,i'',\tau'')\) be the cell of the double diagonal that \((1,i,\tau)\) or \((1,i',\tau')\) is diagonally linked to. Then, if \(x\) is the height difference between \((1,i,\tau)\) and \((1,i'',\tau'')\) and \(y\) is the height difference between \((1,i',\tau')\) and \((1,i'',\tau'')\), we can bound the number of different relative positions of \((1,i',\tau')\) with respect to \((1,i,\tau)\), such that the height difference between \((1,i,\tau)\) and \((1,i'',\tau'')\) is \(x\) and the height difference between  \((1,i',\tau')\) and \((1,i'',\tau'')\) is \(y\) by \(A(x+1)A(y+1)\).
	
	Let \(H_k\) be the side length of the cube \(S_k^{\textrm{2sup}}(i)\) relative to \(S_1(i)\), as in the proof of \Cref{lemma:number_of_paths}.
	Therefore, given the \(z\) cells of scales \(k_1,k_2,\dots,k_z\), the maximum number of scale \(1\) diagonal steps contained in all single and double diagonal connections between the cells of the path is at most
	\[
		H:=2\sum_{i=1}^{z-1} H_{k_i}.
	\]
	For notational convenience, when two consecutive cells \((k,i,\tau)\) and \((k',i',\tau')\) of the \(DD\)-path are double diagonally connected, we now consider as part of the path also the cell \((1,i'',\tau'')\) of the double diagonal that both \((k,i,\tau)\) and \((k',i',\tau')\) are diagonally connected to. Then
	letting \(x_i\), for \(i\in\{1,2,\dots,2z-1\}\) be the height difference between two diagonally connected cells, with \(x_i=0\) if the cells are support adjacent, we have that the number of possible configurations of the diagonal steps is at most
	\begin{equation}\label{for:sum_volDD}
		\sum_{y=0}^H\sum_{\substack{x_1,x_2,\dots,x_{2z-1}\\x_1+\dots+x_{2z-1}=y}}A(x_1+1)A(x_2+1)\cdots A(x_{2z-1}+1).
	\end{equation}
	
	See \Cref{fig:DDpaths} for an illustration of one such configuration. As in the proof of \Cref{lemma:number_of_paths}, we have that
	\[
		A(x_1+1)A(x_2+1)\cdots A(x_{2z-1}+1)\leq A\left(\tfrac{y}{2z-1}+1\right)^{2z-1}.
	\]

	Next, using the above bound and 
	\[
		\sum_{\substack{x_1,x_2,\dots,x_{2z-1}:\\x_1+\dots+x_{2z-1}=y}}1={2z+y-2\choose 2z-2},
	\]
	we have that the sum in (\ref{for:sum_volDD}) is smaller than
	\[
		\sum_{y=0}^H{2z+y-2\choose 2z-2}A\left(\tfrac{y}{2z-1}+1\right)^{2z-1}\leq{2z+H\choose 2z}A\left(\tfrac{H}{2z-1}+1\right)^{2z-1},
	\]
	where the binomial inequality used can easily be proven by induction (using Pascal's rule).

	Then, for some positive constants \(C\) and \(C_2\), we have
	\begin{align}
		{2z+H\choose 2z}A\left(\tfrac{H}{2z-1}+1\right)^{2z-1}&\leq C\frac{(2z+H)^{2z}}{(2z)!}\left(\frac{H}{2z-1}+1\right)^{(2z-1)d}\nonumber\\
		&\leq C \frac{(2z+H)^{2z}}{(2z/3)^{2z}}\left(\frac{2H}{2z}\right)^{2zd}\nonumber\\
		&\leq C(3+3H/(2z))^{2z} \left(\frac{H}{z}\right)^{2zd}\nonumber\\
		&\leq C\left(C_2\frac{H}{z}\right)^{4zd}.\nonumber
	\end{align}
	In order to complete the proof, it remains to show that \(C\left(C_2\frac{H}{z}\right)^{4zd}\leq\exp\left\{\frac{c_3}{8}\sum_{j=1}^z\psi_{k_j}\right\}\), which is equivalent to showing that
	\begin{equation}\label{for:1overalphaDD}
		\tilde Cz\log\left(\frac{H}{z}\right)\leq\sum_{j=1}^z\psi_{k_j},
	\end{equation}
	where \(\tilde C\) is some constant. Setting \(m\) and \(\alpha\) sufficiently large, this holds using the same argument as in the proof of \Cref{lemma:number_of_paths}.
\end{proof}
Similar to \Cref{lemma:number_of_paths}, if \(d\geq 3\) we have that \Cref{lemma:number_of_DDpaths} holds also when we set time to be height in the base-height index. For \(d=2\) one can construct a similar counterexample as the one outlined in \Cref{rem:height2}.

\begin{lemma}\label{lemma:weight_of_path_DD}
	Let \(t>0\) and let \(P=((k_1,i_1,\tau_1),(k_2,i_2,\tau_2),\dots,(k_z,i_z,\tau_z))\) be a path in \(\Omega_{\kappa-1,t}^{\textrm{sup}}\). If \(\alpha\) is sufficiently large and \(\kappa=\mathcal{O}(\log t)\), then there exist a positive constant \(c=c(C_M)\) and a value \(C\) independent of \(t\) such that
	\begin{align}
		\sum_{j=1}^z\psi_{k_j}\geq\begin{cases}
			C\frac{\sqrt{t}}{(\log t)^c},&\textrm{for }d=1,\\
			C\frac{t}{(\log t)^c},&\textrm{for }d=2,\\
			Ct,&\textrm{for }d\geq 3.
		\end{cases}
	\end{align}
\end{lemma}
\begin{proof}
	The proof is identical to the proof of \Cref{lemma:weight_of_path}, save for one change. In \Cref{lemma:weight_of_path}, when considering the sum across the cells of the path, we require that \(\sum_{j=1}^z\Delta_{k_j}^{2\textrm{sup}}\geq t/2\). Since we now consider two diagonals per cell instead of just one, the term on the right has to be changed to \(t/3\) in order for the statement to still hold. The rest of the proof is unchanged.
\end{proof}

We now define the analogous set of \(K(i,\tau)\) for \(DD\)-paths. Given an increasing event 
\(E_{\textrm{st}}(i,\tau)\), let \(E(i,\tau)\) be the indicator random variable of \(E_{\textrm{st}}(i,\tau)\).
\begin{mydef}
Let \((i,\tau)\in\mathbb{Z}^{d+1}\). If \(E(i,\tau)=1\), define \(K^*(i,\tau)=\emptyset\). Otherwise define \(K^*(i,\tau)\) as the set
\begin{align*}
	\{&(i',\tau')\in\mathbb{Z}^{d+1}:\,E(i',\tau')=0\textrm{ and }
	\exists\textrm{ a }DD\textrm{-path of bad cells from }(i,\tau)\textrm{ to }(i',\tau')\}.
\end{align*}
\end{mydef}

\begin{prop}\label{prop:exp_tailDD}
	For each \((i,\tau)\in\mathbb{Z}^{d+1}\), let \(E_{\textrm{st}}(i,\tau)\) be an increasing event that is restricted to the super cube \(i\) and the super interval \(\tau\), and let \(\nu_{E_{\textrm{st}}}\) be the probability associated to \(E_{\textrm{st}}\) as defined in \Cref{def:probassoc}. Fix a constant \(\epsilon\in(0,1)\), and integer \(\eta\geq1\) and the ratio \(\beta/\ell^2>0\). Fix also \(w\) such that
	\[
		w\geq\sqrt{\frac{\eta\beta}{c_2\ell^2}\log\left(\frac{8c_1}{\epsilon}\right)},
	\]
	for some constants \(c_1\) and \(c_2\) which depend only on the graph. Then, there exist constants \(c\) and \(C\), and positive numbers \(\alpha_0\) and \(t_0\) that depend on \(\epsilon\), \(\eta\), \(w\) and the ratio \(\beta/\ell^2\) such that if
	\[
	\alpha=\min\left\{C_{M}^{-1}\epsilon^2\lambda_0\ell^d,\log\left(\frac{1}{1-\nu_{E_{\textrm{st}}}((1-\epsilon)\lambda,Q_{(2\eta+1)\ell},Q_{w\ell},\eta\beta)}\right)\right\}\geq\alpha_0,
	\]
	we have for all \(t\geq t_0\) that
	\[
	\mathbb{P}\left[K^*(0,0)\not\subseteq\mathcal{R}^t_1\right]\leq\left\{
	\begin{array}{ll}
		\exp\left\{-C\lambda_0\frac{t}{(\log t)^c}\right\}&\textrm{for }d=2\\
		\exp\left\{-C\lambda_0 t\right\}&\textrm{for }d\geq 3.
	\end{array}
	\right.
	\]
\end{prop}
\begin{proof}
	The proof of this result proceeds in the same way as the proof of \Cref{prop:exp_tail}, by replacing \Cref{lemma:support_connected_paths} with \Cref{lemma:support_connected_DDpaths}, \Cref{lemma:number_of_paths} with \Cref{lemma:number_of_DDpaths} and \Cref{lemma:weight_of_path} with \Cref{lemma:weight_of_path_DD}.
\end{proof}

We now argue that \Cref{prop:exp_tailDD} implies that the Lipschitz surface not only almost surely exists as shown in \Cref{thrm:surface_event_simple}, but that areas of the surface that have non-zero height are finite as well. To see why, denote with \(u_i=(b_i,0)\) sites in \(\mathbb{L}\) and consider a path along the surface \(F\). More precisely, let \(\pi=\{(b_1,F_+(u_1)),(b_2,F_+(u_1)),\dots,\) \((b_n,F_+(u_n))\}\) be such that \(F_+(u_i)\neq 0\) for all \(i\in \{1,\dots,n\}\) and \(\|u_i-u_{i-1}\|_{1}=1\) for all \(i\in\{2,\dots,n\}\). If such a path exists, then both sides of the Lipschitz surface have non-zero height at least at the cells of the path, so one can follow the path \((u_1,u_2\dots,u_n)\) and never reach the Lipschitz surface \(F\). Conversely, if a path \(\pi\) as above that leaves a ball of finite radius does not exist, a self-avoiding path will have to reach the surface in finitely many steps. Furthermore, since time is one of the \(d+1\) dimensions, one cannot construct a time directed path without it containing a cell \((b,F_+(u))\) or \((b,F_-(u))\) for some \(u=(b,0)\in \mathbb{L}\) within a finite number of steps. This follows from the fact that by \Cref{thrm:surface_event_simple} the surface is a.s. finite, so a path can avoid intersecting it indefinitely only if there is always at least one way to construct a path between to the two sides of the surface. If however, paths along which the two sides of the surface have non-zero height cannot have arbitrary length, we get that avoiding the two sides indefinitely is impossible.

To simplify things, we first observe that we can limit ourselves to only the positive Lipschitz open surface, since \(F_+(u)=0\) if and only if \(F_-(u)=0\), by the definition of the two sides of the surface.

Recall from \Cref{section:paths} the definition of a hill \(H_u\). In the following, we will use \(H_i\), \(i\in\mathbb{Z}\) to differentiate between different hills without specifying a cell \(u\in\mathbb{L}\) for which \(H_i=H_u\). We now show that the existence of a path along the surface with only positive heights implies the existence of a sequence of hills that are pairwise intersecting or adjacent. Formally, we define the following.
\begin{mydef}
	We say a hill \(H_i\) is \emph{adjacent} to a hill \(H_{i'}\), if there exist a cell \(u\in H_i\) and a cell \(v\in H_{i'}\) such that \(\|u-v\|_1=1\). We say \(H_i\) and \(H_{i'}\) are \emph{intersecting}, if there exists a cell \(u\) such that \(u\in H_i\) and \(u\in H_{i'}\).
\end{mydef}

\begin{lemma}\label{lemma:sequence_of_hills}
	Write \(u_i=(b_i,0)\in\mathbb{L}\) and let \(\pi=\{(b_1,F_+(u_1)),(b_2,F_+(u_2)),\dots,\allowbreak(b_n,F_+(u_n))\}\) be a path, such that for all \(i\in \{1,\dots,n\}\), \(F_+(u_i)\neq 0\). Then there exists a sequence of hills \(\mathcal{H}=H_1,H_2,\dots,H_k\), \(k\leq n\), such that for every \(u_{\ell}\) there exists a hill \(H_k\in\mathcal{H}\) that contains \(u_{\ell}\), and such that for all \(i\in\{1,\dots k\}\), there exists at least one \(j\neq i\), \(j\in\{1,\dots,k\}\), for which \(H_i\) intersects with or is adjacent to \(H_j\).
\end{lemma}

\begin{proof}
	We will prove the existence of the sequence of hills iteratively. Let \(\mathcal{H}=\emptyset\) be the set of all hills that are part of the sequence already. We then add hills to \(\mathcal{H}\) in the following manner.
	Let \(u=(b,0)\) be the first cell of the path \(P=\left((b_1,0),(b_2,0),\dots,(b_n,0)\right)\) that is not contained in \(\bigcup_{H\in\mathcal{H}}H\). Since \(F_+(u)\neq 0\) by assumption, there has to exist at least one cell \(v\in\mathbb{L}\) such that \(u\in H_{v}\). Since the cell \(u\) is contained in \(H_{v}\) and it is adjacent to at least 1 cell contained in \(\bigcup_{H\in\mathcal{H}}H\) (except for when \(\mathcal{H}=\emptyset\)), we get that  \(H_{v}\) and at least one hill from \(\mathcal{H}\) are adjacent or they intersect. We add \(H_v\) to \(\mathcal{H}\), remove all cells of \(P\) that are contained in \(H_v\) from \(P\), and repeat the procedure.	
	After at most \(n\) steps, the recursion ends and \(\mathcal{H}\) is a set of \(k\) hills for some \(k\leq n\), such that every hill intersects or is adjacent to at least one other hill in the set.
\end{proof}

We now want to show that if the sequence of hills \(\mathcal{H}\) from \Cref{lemma:sequence_of_hills} exists, then a \(DD\)-path exists between any two cells contained in \(\bigcup_{H\in\mathcal{H}}H\).

\begin{lemma}\label{lemma:hills-to-DDpath}
	Let \(\mathcal{H}=H_1,H_2,\cdots H_k\) be a sequence of hills as in \Cref{lemma:sequence_of_hills}. For any two \((b,0),(b',0)\in \bigcup_{H_i\in\mathcal{H}}H_i\), there exists a \(DD\)-path that starts in \((b,0)\) and ends in \((b',0)\).
\end{lemma}

\begin{proof}
	Let \(u_1,u_2,\dots,u_{k}\in\mathbb{L}\) be the cells such that \(H_i=H_{u_i}\) for all \(i\in\{1,2,\dots,k\}\).
	Next, observe that by the definition of \(\mathcal{H}\), there exists a sequence of hills \(H_{i_1},H_{i_2},\dots,H_{i_{\ell}}\) such that \((b,0)\in H_{i_1}\), \((b',0)\in H_{i_{\ell}}\) and every hill in the sequence is adjacent or intersecting with the subsequent hill. For every \(j\in\{1,2,\dots,\ell\}\), let \(v_{i_j}\in H_{i_j}\) be a cell that is contained in \(H_{i_{j+1}}\) or adjacent to a cell in \(H_{i_{j+1}}\).
	
	By definition of a hill, there exists a \(d\)-path \(P_1\) from \(u_{i_1}\) to \((b,0)\). Furthermore, there exists a \(d\)-path \(P_2\) from \(u_{i_1}\) to \(v_{i_1}\) and a \(d\)-path \(P_3\) from \(u_{i_2}\) to \(v_{i_1}\) or a cell that is adjacent to \(v_{i_1}\). By repeating this, we obtain the sequence of cells
	\[
		(b,0),u_{i_1},v_{i_1},u_{i_2},v_{i_2},\dots,v_{i_{\ell-1}},u_{i_{\ell}},(b',0),
	\]
	where there exists a \(d\)-path from the first cell to the second or from the second to the first (or a cell adjacent to it) for every consecutive pair of cells. It remains to show that this implies that there exists a \(DD\)-path from \((b,0)\) to \((b',0)\).

	Note first that similar to \(D\)-paths, every \(d\)-path is also a \(DD\)-path. This follows directly from the fact that \(DD\)-paths are defined as an extension of \(D\)-paths. Next, note that if a sequence of cells \((w_1,w_2,\dots,w_n)\in\mathbb{Z}^{d+1}\) is a \(DD\)-path, then the reverse sequence, i.e. \((w_n,w_{n-1},\dots,w_1)\) is also a \(DD\)-path. This follows trivially from the fact that being adjacent, single diagonally connected and double diagonally connected are all symmetric relationships between cells. Finally, note that if there exists a \(DD\)-path from a cell \(w_1\) to some cell \(w_2\) and there exists a \(DD\)-path from \(w_2\) to \(w_3\), then there exists at least one \(DD\)-path from \(w_1\) to \(w_3\). Once such path can be constructed by concatenating the two \(DD\)-paths and removing any cells in the concatenated path that would result in loops, i.e. if a site appears in the concatenated path more than once, remove from the path all sites between the first and last appearance of the site in the path, as well as the last appearance of the site.
	
	Then, using these facts with the sequence 
	\[
		(b,0),u_{i_1},v_{i_1},u_{i_2},v_{i_2},\dots,v_{i_{\ell-1}},u_{i_{\ell}},(b',0)
	\]
	concludes the lemma.
\end{proof}

We are now ready to prove \Cref{thrm:shell_tail}.

%

\begin{proof}[{Proof of \Cref{thrm:shell_tail}}]
	Note that the open Lipschitz surface exists a.s. by \Cref{thrm:surface}, so we only need to show that it surrounds the origin at some finite distance.

	Assume the converse. Then, for any \(r>0\) there must exist a path of adjacent cells \((0,0)=(b_1,0),\dots,(b_n,0)\) with \(\|(b_n,0)\|_1>r\), such that \(F_+((b_i,0))\neq 0\) for all \(i\in\{1,\dots,n\}\). By \Cref{lemma:sequence_of_hills}, this implies the existence of a sequence of hills such that the first one contains the origin and the last one contains \((b_n,0)\). By \Cref{lemma:hills-to-DDpath}, this gives the existence of a \(DD\)-path from the origin to \((b_n,0)\).

	Note that by \Cref{prop:exp_tailDD}, for \(t\geq t_0\) we have that the probability that such a \(DD\)-path exists is smaller than
	\[
	\mathbb{P}[K^*(0,0)\not\subseteq\mathcal{R}_1^t]\leq\left\{
	\begin{array}{ll}
		\exp\left\{-C\lambda_0\frac{t}{(\log t)^c}\right\}&\textrm{for }d=2\\
		\exp\left\{-C\lambda_0 t\right\}&\textrm{for }d\geq 3.
	\end{array}
	\right.
	\]

	From here, setting \(t=\left(\frac{d}{\ell}+\frac{1}{c\ell^2}\right)^{-1}r\) and using the same steps as in the proof of \Cref{thrm:surface_event_simple} establishes the claim for \(r\geq r_0:=\left(\frac{d}{\ell}+\frac{1}{c\ell^2}\right)t_0\).
\end{proof}

Next, we show that \Cref{corol:percolate_simple} holds. Observe first the following well known geometric property. Let \(B^2\) be the plane spanned by any two base vectors of the base-height index. Recall also the definition of \(\mathbb{L}=\{(x,0),x\in\mathbb{Z}^d\}\), the zero-height hyperplane of \(\mathbb{Z}^{d+1}\). It then holds that
\begin{align*}
	&\mathbb{P}\left[\textrm{zero height cells percolate in }B^2\right]\\
	&\leq \mathbb{P}\left[\textrm{zero height cells percolate in }\mathbb{L}\right],
\end{align*}
since it clearly holds that the first event implies the second. Therefore, it is enough to show that the first probability is positive for \Cref{corol:percolate_simple} to hold.

\begin{corol}\label{corol:percolate}
	Let \(d=2\) and let \(E_{\textrm{st}}(i,\tau)\) be an increasing event restricted to the super cell \((i,\tau)\). If \(\ell\) is sufficiently large and \(\mathbb{P}\left[E_{\textrm{st}}(0,0)\right]\) is large enough, then \(F\cap\mathbb{L}\) percolates within \(\mathbb{L}\) with positive probability.
\end{corol}

\begin{proof}
	Assume without loss of generality that the origin is contained in \(F\cap\mathbb{L}\) and assume that the cluster of \(F\cap\mathbb{L}\) that contains the origin is finite. Let \(x,y\in F\cap\mathbb{L}\) be two cells of this cluster for which \(\|x-y\|_1\) is largest, and let \(k:=\lceil\|x-y\|_1\rceil\). Then, there exists a sequence cells 
	\(v_1,v_2,\dots,v_n\in\mathbb{L}\) for some \(n\geq 2k\) such that \(F_+(v_i)>0\) for all \(i\in\{1,2,\dots,n\}\), any two consecutive cells are adjacent, i.e. \(\|v_i-v_i'\|_{\infty}=1\), and such that \(\|v_n-v_1\|_{\infty}=1\). Furthermore, each such sequence contains at least 2 cells \(u,v\in \mathbb{L}\) for which \(\|u-v\|_1\geq k\). By using Lemmas \ref{lemma:sequence_of_hills} and \ref{lemma:hills-to-DDpath}, this gives that there exists a \(DD\)-path that begins in \(u\) and ends in \(v\). Let \(r_0\) be a sufficiently large value of \(r\) so that \Cref{thrm:shell_tail} holds. We then have for \(k\geq r_0\), by using that the probability space is space and time translation invariant that
	\begin{align*}
		&\mathbb{P}[\textrm{the cluster \(F\cap\mathbb{L}\) around the origin has diameter }k]\\
		&\leq \mathbb{P}[\textrm{a }DD\textrm{-path started at the origin leaves the ball of radius \(k\) centered at the origin}]\\
		&\leq \exp\left\{-C\lambda\frac{\ell k}{(\log\ell k)^c}\right\},
	\end{align*}
	where the second inequality follows from the same argument as in the proof of \Cref{thrm:shell_tail}.
	For \(k< r_0\), we can bound the probability by \(C r_0^2\mathbb{P}[E_{\textrm{st}}(0,0)^c]\) for some positive constant \(C\), since a closed cell implies that the Lipschitz function \(F_+\) is non-zero. Therefore, we get that the probability the zero-height cluster \(F\cap\mathbb{L}\) at the origin is not finite is greater than
	\[
		1-C r_0^2\mathbb{P}[E_{\textrm{st}}(0,0)^c]-\sum_{k\geq r_0}\exp\left\{-C\lambda\frac{\ell k}{(\log\ell k)^c}\right\},
	\]
	which is positive for sufficiently large \(\ell\), if \(\mathbb{P}[E_{\textrm{st}}(0,0)]\) is large enough.
\end{proof}

\appendix
\section{Appendix: Standard results}
\begin{lemma}[Chernoff bound for Poisson]\label{lem:chernoff}
Let \(P\) be a Poisson random variable with mean \(\lambda\). Then, for any \(0<\epsilon<1\),
\[
	\mathbb{P}[P<(1-\epsilon)\lambda] < \exp\{-\lambda\epsilon^2/2\}
\]
and
\[
	\mathbb{P}[P > (1 + \epsilon)\lambda] < \exp\{-\lambda\epsilon^2/4\}.
\]
\end{lemma}

\begin{lemma}\label{lemma:binomial}
	Let \(x,y\in\mathbb{Z}_+\). Then, for any \(c_1,c_2>1\), we have
	\[
		{x+y\choose x}e^{-(c_1x+c_2y)}\leq e^{-(c_1-1)x-(c_2-1)y}.
	\]
\end{lemma}

\begin{proof}
	Since \({x+y\choose x} = {x+y\choose y}\), we can assume that \(x\geq y\). Then we use the inequality \({x+y\choose x} \leq \left(\frac{(x+y)e}{x}\right)^x\) to obtain	
	\[
		{x+y\choose x}e^{-c_1x-c_2y}\leq \left(1+\frac{y}{x}\right)^xe^{-(c_1-1)x-c_2y}\leq e^{-(c_1-1)x-(c_2-1)y}.
	\]
\end{proof}

\bibliographystyle{plain}\addcontentsline{toc}{section}{References}
\bibliography{library.bib}
\end{document}